\documentclass{article}
\usepackage{authblk}
\usepackage{soul}
\usepackage[T1]{fontenc}
\usepackage{fullpage}
\usepackage{xcolor}

\usepackage{graphicx}
\usepackage{placeins}

\usepackage{epstopdf}
\usepackage{amsmath,amssymb,amsthm}
\newtheorem{lem}{Lemma}
\newtheorem{thm}{Theorem}
\DeclareMathOperator*{\argmin}{arg\,min}

\usepackage{float}
\floatstyle{ruled}
\newfloat{algorithm}{tbp}{loa}
\providecommand{\algorithmname}{Algorithm}
\floatname{algorithm}{\protect\algorithmname}
\usepackage{algorithmic}

\usepackage{cleveref}
\usepackage{url}

\newcommand{\cfa}{\gamma}       
\newcommand{\cfb}{\delta}

\newcommand{\vecb}{b}
\newcommand{\matA}{A}

\newcommand{\Herr}[1]{\| x - x_{#1} \|^2_{\matA}}

\begin{document}
\title{Estimating the error in CG-like algorithms\\ for least-squares and least-norm problems}

\author[a]{Jan Pape\v{z}}
\author[b]{Petr Tich\'y}
\affil[a]{Institute of Mathematics of the Czech Academy of Sciences, Prague, Czech Republic}
\affil[b]{Faculty of Mathematics and Physics, Charles University, Prague, Czech Republic}
\setcounter{Maxaffil}{0}
\renewcommand\Affilfont{\itshape\small}

\date{version of \today}
\maketitle

\begin{abstract}
    In [Meurant, Pape\v{z}, Tich\'{y}; Numerical Algorithms~88, 2021], we presented an adaptive estimate for the energy norm of the error in the conjugate gradient (CG) method. In this paper, we extend the estimate to algorithms for solving linear approximation problems with a general, possibly rectangular matrix that are based on applying CG to a system with a positive (semi-)definite matrix build from the original matrix. We show that the resulting estimate preserves its key properties: it can be very cheaply evaluated, and it is numerically reliable in finite-precision arithmetic under some mild assumptions. We discuss algorithms
    based on Hestenes--Stiefel-like implementation (often called CGLS and CGNE in the literature) as well as on bidiagonalization (LSQR and CRAIG),
    and both unpreconditioned and preconditioned variants. The numerical experiments confirm the robustness and very satisfactory behaviour of the estimate. 
\end{abstract}

\section*{Introduction}

Solving linear approximation problems (with a general, possibly rectangular matrix) is a common task in scientific computing. 
In this paper, we consider a least-squares problem
\[
    \min_{z \in \mathbb{R}^n} \| b - A z \|, \qquad A \in \mathbb{R}^{m \times n},\ b \in \mathbb{R}^{m}, \ m \geq n,
\]
and a least-norm problem
\[
    \min_{z \in \mathbb{R}^n} \| z \| \quad \mbox{subject to} \quad Az = b, \qquad A \in \mathbb{R}^{m \times n},\ b \in \mathbb{R}^{m},\ b \in \mathcal{R}(A),    
\]
where $\|\cdot\|$ denotes the Euclidean norm and 
$\mathcal{R}(A)$ is the \emph{range} of $A$, i.e., the space spanned by columns of $A$.

When $A$ is large and sparse, an iterative method is typically used to get an approximation to a (possibly non-unique) solution of the problem. For using any iterative method, one needs an appropriate stopping criterion which typically requires an estimate 
of some quantity that measures the quality of the computed approximation.

Krylov subspace methods are among the most widely used and studied methods for iterative solution of least-squares and least-norm problems. Their mathematical derivation is based on a minimization of a certain error or residual norm over a Krylov subspace (of a successively increasing dimension); see, e.g,~\cite{LieStrBook13}. Typically, there are several algorithms that are mathematically equivalent (they are based on the same minimization problem over the same subspaces) but they differ in the implementation, potentially leading to a different behavior in finite-precision arithmetic.

The purpose of this paper is to extend a (heuristic-based) adaptive error estimate derived in~\cite{MePaTi2021} for the conjugate gradient (CG) method to a class of algorithms for solving the above problems. We consider various algorithmic variants that are mathematically equivalent to applying CG to a system with a symmetric positive {(semi-)definite} matrix $A^TA$ or $AA^T$ build from the original matrix~$A$.
In particular, the algorithms discussed in this paper are CGLS, LSQR, CGNE, and CRAIG (Craig's method based on bidiagonalization).
For each of these algorithms we present the derivation of the estimate and show that the two key properties of the estimate known from CG persist. 
First, the estimate can be evaluated very cheaply. Second, under the assumption that the local orthogonality is preserved during finite-precision computations, and that the maximal attainable accuracy has not yet been reached, the estimate is numerically reliable even if the convergence of the algorithms is strongly influenced by rounding errors. The convergence may be significantly delayed but the delay is reflected in the estimate; it estimates the actual error of the computed approximation. The rounding error analysis concerning local orthogonality for the CG algorithm can be found in \cite{StTi02}; see also \cite{B:Me2006}. Such an analysis would be doable also for the algorithms discussed in this paper, however, it would require many additional technical details.
The maximal attainable accuracy was analyzed, e.g., in \cite{Gr1989} and \cite{SlVoFo1994} in a more general context. Note that some of the estimates discussed in this paper can be found in the literature, see, e.g., \cite{JiTi09,ArGr2012,Ar2013,EsOrSa2019a},
but without using our adaptive technique and also without any discussion on using the estimates in finite-precision arithmetic.

The paper is organized as follows. We first briefly recall the adaptive error estimate of \cite{MePaTi2021}. Then, in \Cref{sec:LSQR_CGLS} we present the least-squares problem and its solution by CGLS and LSQR, proposed in \cite{HeSt1952} and \cite{PaSa1982}, respectively. In \Cref{sec:Craigestim}, we discuss the least-norm problem and the error estimate for Craig's method (proposed originally in~\cite{Cra54thesis,Cr1955}, implemented using CGNE (\cite[p.~504]{B:FaFa1963}), and the algorithm based on Golub--Kahan bidiagonalization \cite{PaSa1982}
that we call CRAIG; see, e.g., \cite{Sa1995a}.
Then, in \Cref{sec:precond}, we discuss the estimate for preconditioned variants of the algorithms. The results of numerical experiments are given in \Cref{sec:numexp} and the paper ends with concluding remarks.
The MATLAB codes of the algorithms with error estimates are available from the GitHub repository~\cite{Github_repo}.

\section{Adaptive error estimate in CG}
\label{sec:recallestimate}

In this section we briefly recall the 
adaptive error estimate derived in \cite{MePaTi2021} for the conjugate gradient method. 

First recall the idea of the CG method for solving a system $\matA x=\vecb$ with a symmetric and positive definite matrix~$\matA$.
Starting with an initial guess~$x_0$ and the associated residual $r_0 = \vecb - \matA{}x_0$, CG generates the approximations 
\[
    x_k \in x_0 + \mathcal{K}_k(\matA,r_0), \qquad \mathcal{K}_k(\matA,r_0) \equiv \mbox{span}\{ r_0, \matA r_0, \ldots, \matA^{k-1} r_0\},
\]
such that
\begin{equation}
    \label{eq:CG_minimization}
    x_k\, = \argmin_{y \in x_0 + \mathcal{K}_k(\matA,r_0)} \| x - y \|_{\matA},
\end{equation}
where $\| \cdot \|_{\matA}$ denotes the $A$-norm, $\| v\|_A^2 = v^T A v$.
The standard Hestenes and Stiefel implementation of the CG method is given in \Cref{alg:cg}. Therein, we put on line 6 a call of the function $\mathtt{adaptive}(\cdot)$ that performs the computation of the adaptive error estimate and is explained below.

Note that
CG  can be applied also to a system $\matA x=\vecb$ with a matrix~$\matA$ that is symmetric positive semi-definite only;
see, e.g.,~\cite{Ka88}.
Then $\| \cdot \|_{\matA}$ is a norm on $\mathcal{R}(\matA)$ only,
and we need an additional assumption so that~\eqref{eq:CG_minimization} makes sense.
In particular, we must assume that the system is consistent, i.e., that $\vecb \in \mathcal{R}(\matA)$, and choose the initial guess $x_0$ from $\mathcal{R}(\matA)$. This assures that $\mathcal{K}_k(\matA,r_0) \subset \mathcal{R}(\matA)$ and $x_k \in \mathcal{R}(\matA)$, for $k = 0, 1, \ldots$, so that the minimization~\eqref{eq:CG_minimization} can be considered over~$\mathcal{R}(\matA)$. 
Note that in finite-precision arithmetic, the singularity of the system matrix~$\matA$ can cause some numerical instabilities, in particular, the error norm can start to increase after reaching the level of maximal attainable accuracy; see \cite[Section~4]{Ka88}.

\begin{algorithm}[htp]
\caption{Conjugate Gradients for $Ax=b$ (CG)} \label{alg:cg}

\begin{algorithmic}[1]
\STATE \textbf{input} $\matA$, $b$, $x_{0}$
\STATE $r_{0}=\vecb-\matA x_{0}$, $p_{0}=r_{0}$
\STATE $\ell = 0$, $\tau = 0.25$\hfill\textcolor{gray}{prescribe the tolerance for error estimation}
\FOR{$k=0,1,\dots$ until convergence}
    \medskip
    \STATE $\cfa_{k}= {\|r_{k}\|^2}/{\|p_k\|_{\matA}^2} = 
    {r_{k}^{T}r_{k}}\, / \, {p_{k}^{T}{\matA}p_{k}}$ \smallskip
    \STATE ${\tt adaptive}\/(k, \ell,\Delta_k = \cfa_{k} \|r_{k}\|^2, \tau)$
    \STATE $x_{k+1}=x_{k}+\cfa_{k}p_{k}$
    \STATE $r_{k+1}=r_{k}-\cfa_{k}\matA p_{k}$
    
    \STATE $\cfb_{k+1}= \| r_{k+1}\|^2 / \| r_{k} \|^2$
    
    \STATE $p_{k+1}=r_{k+1}+\cfb_{k+1}p_{k}$
    
    \smallskip
\ENDFOR
\end{algorithmic}

\end{algorithm}

The adaptive error estimate of \cite{MePaTi2021} is based on the formula (which we call the Hestenes--Stiefel formula),
\begin{equation}
\label{eq:errkkd}
    \Herr{\ell} = \sum^k_{j=\ell} \Delta_j + \Herr{k+1}, \qquad \Delta_j \equiv \cfa_{j} \| r_{j} \|^2,
\end{equation}
for $\ell \leq k$, and on a procedure that adaptively finds $\ell = \ell(k)$, for a given $k$, such that
\[
    \frac{\Herr{\ell} - \sum^k_{j=\ell} \Delta_j}{\Herr{\ell}} \leq \tau
\]
for some prescribed tolerance~$\tau \in (0,1)$ (typically set as $0.25$).
Then the error $\Herr{k}$ is estimated by 
\[
    \Delta_{\ell:k} \equiv \sum^k_{j=\ell} \Delta_j  \approx \Herr{\ell}.
\]
From \eqref{eq:errkkd}, it is clear that $\Delta_{\ell:k}$ represents a lower bound on $\Herr{\ell}$.
The bound $\Delta_{\ell:k}$ is very cheap to evaluate, it is numerically stable \cite{StTi02,StTi05},
and if $\ell$ 
is set properly, it provides a sufficiently accurate estimate.

\begin{algorithm}[htp]
\caption{${\mathtt{adaptive}()}$ --- adaptive choice of the delay $k-\ell$} \label{alg:adaptive}

\begin{algorithmic}[1]
\STATE \textbf{input} $k$, $\ell$, $\Delta_k$, $\tau$
\STATE $\mathrm{TOL} = 10^{-4}$\hfill\textcolor{gray}{set a parameter for error estimation}
\STATE \textbf{output} $\ell$
\STATE set $m$ as the largest index~$j$, $0 \leq j < k$, such that
\[
    \frac{\Delta_{\ell:k}}{\Delta_{j:k}} \leq \mbox{TOL}
\]
if such an $m$ does not exist, set $m=0$
\STATE determine $S$ as
\[
    S = \max_{m \leq j < k} \Delta_{j:k} / \Delta_j
\]
\WHILE{$S \Delta_{k} / \Delta_{\ell:k-1} \leq \tau$ \textbf{and} $\ell < k$}
\STATE accept $\Delta_{\ell:k}$ as an estimate of the error at $\ell$th iteration
\STATE $\ell = \ell+1$
\ENDWHILE
\end{algorithmic}
\end{algorithm}

The function $\mathtt{adaptive}(\cdot)$ for setting the (adaptive) delay $k-\ell$ and computing the error estimate is presented in \Cref{alg:adaptive}. 
As one can observe, in the $k$th CG iteration when the approximation $x_{k+1}$ is computed, the  error estimate is evaluated for some previous approximation $x_\ell$.
For a detailed description, derivation, and reasoning, see~\cite{MePaTi2021}. 
Note that in the original paper~\cite{MePaTi2021}, $\Delta_{\ell:k-1}$ was used on line~7 of \Cref{alg:adaptive} for estimating the error. Since $\Delta_{\ell:k}$ is available and $\Delta_{\ell:k-1} < \Delta_{\ell:k} \leq \Herr{\ell}$, we suggest using a tighter estimate. Obviously, this function can be easily added to the existing CG codes. 

As described in~\cite[Section~4]{MePaTi2021}, the adaptive error estimate can also be extended to preconditioned CG while preserving all key properties. For the sake of simplicity, the resulting algorithm is not presented here and we refer to~\cite[Section~4]{MePaTi2021}.

\section{Estimating the error in least-squares problems}
\label{sec:LSQR_CGLS}
\label{sec:LS}

Let us consider a least-squares problem
\begin{equation}
\label{eq:lsproblem}
    x = \argmin_{z \in \mathbb{R}^n} \| b - A z \|
\end{equation}
where $A \in \mathbb{R}^{m \times n}$, $m \geq n$, is a given matrix and $b \in \mathbb{R}^{m}$ is a right-hand side. It is well known that $x$ is a least-squares solution if and only if $b - Ax \perp \mathcal{R}(A)$, or equivalently, if and only if $x$ is the solution of the system of \emph{normal equations}
\begin{equation}
\label{eq:normeq}
    A^T A x = A^T b.
\end{equation}
If $\mathrm{rank}(A) = n$, then the matrix $A^T A$ is nonsingular and the solution of~\eqref{eq:lsproblem} is unique. Moreover, $A^T A$ is symmetric and positive definite and, therefore, the CG method can be applied to~\eqref{eq:normeq}.

When CG is applied to \eqref{eq:normeq} starting with an initial guess $x_0$ and $r_0 = b - Ax_0$, the approximation~$x_k$ computed at the $k$th iteration satisfies
\[
    x_k \in x_0 + \mathcal{K}_k(A^TA,A^Tr_0),    
\]
and
\begin{equation}
    \label{eq:CGLS_minimization}
    x_k = \argmin_{y \in x_0 + \mathcal{K}_k(A^TA,A^Tr_0)} \| x - y \|_{A^TA}.
\end{equation}
There are several ways how CG for~\eqref{eq:normeq} can be implemented.
For an overview 
we refer to~\cite{BjElSt1998} or to a series of papers \cite{EsOrSa2019a,EsOrSa2019b,EsOrSa2019}.

\begin{algorithm}[htp]
\caption{CGLS}
\label{alg:cgls}

\begin{algorithmic}[1]

\STATE \textbf{input} $A$, $b$, $x_{0}$

\STATE $r_{0}=b-Ax_{0}$

\STATE $s_{0}=p_{0} = A^{T}r_{0}$

\FOR{$k=0,1,\dots$}

\STATE $q_{k}=Ap_{k}$
\STATE $\cfa_{k} = \| s_{k} \|^2 / \| q_{k} \|^2$
\STATE $x_{k+1}=x_{k}+\cfa_{k}p_{k}$
\STATE $r_{k+1}=r_{k}-\cfa_{k}q_{k}$
\STATE $s_{k+1}=A^Tr_{k+1}$
\STATE $\cfb_{k+1}=\| s_{k+1} \|^2/\| s_{k} \|^2$
\STATE $p_{k+1}=s_{k+1}+\cfb_{k+1}p_{k}$

\ENDFOR

\end{algorithmic}
\end{algorithm}
The CGLS algorithm (\Cref{alg:cgls}) has been described already in the CG seminal paper~\cite[p.~424]{HeSt1952}. It corresponds to applying the CG algorithm to the system of normal equations with a slight algebraic rearrangement to avoid the vectors of the form~$A^TAp$. Due to this rearrangement, CGLS has better numerical properties than CG 
(\Cref{alg:cg}) naively applied to \eqref{eq:normeq}. The CGLS algorithm can be found, e.g., in \cite[Section~7.1]{PaSa1982} and \cite[Algorithm~3.1]{BjElSt1998}.
\begin{algorithm}[htp]
\caption{LSQR}
\label{alg:LSQR}

\begin{algorithmic}[1]

\STATE \textbf{input} $A$, $b$

\STATE $\beta_{1}u_{1}=b$\hfill
\textcolor{gray}{$\beta_{i}$ are the normalization coefficients to have $\|u_{i}\| = 1$, $i = 1,2, \ldots$}
\STATE $\alpha_{1}v_{1}=A^{T}u_{1}$
\hfill
\textcolor{gray}{$\alpha_{i}$ are the normalization coefficients to have $\|v_{i}\| = 1$, $i = 1,2, \ldots$}
\STATE $w_1 = v_1$
\STATE $x_0$ = 0
\STATE $\bar{\phi}_1 = \beta_1$
\STATE $\bar{\rho}_1 = \alpha_1$

\FOR{$k=1,2,\dots$}

\STATE $\beta_{k+1}u_{k+1}=Av_{k}-\alpha_{k}u_{k}$
\STATE $\alpha_{k+1}v_{k+1}=A^{T}u_{k+1}-\beta_{k+1}v_{k}$

\STATE $\rho_k = (\bar{\rho}_k^2 + \beta^2_{k+1})^{1/2}$
\STATE $c_k = \bar{\rho}_k / \rho_k$
\STATE $s_k = \beta_{k+1} / \rho_{k}$
\STATE $\theta_{k+1} = s_{k} \alpha_{k+1}$
\STATE $\bar{\rho}_{k+1} = -c_{k} \alpha_{k+1}$
\STATE $\phi_{k} = c_{k} \bar{\phi}_{k}$
\STATE $\bar{\phi}_{k+1} = s_{k} \bar{\phi}_{k}$

\STATE $x_{k}=x_{k-1} + (\phi_{k} / \rho_{k} ) w_{k}$
\STATE $w_{k+1}=v_{k+1} + (\theta_{k+1} / \rho_{k} ) w_{k}$

\ENDFOR

\end{algorithmic}
\end{algorithm}
A mathematically equivalent algorithm to CGLS based on Golub--Kahan bidiagonalization has been derived in~\cite{PaSa1982}. The algorithm is called LSQR and it is given in \Cref{alg:LSQR}.

Since CG can also be applied to a singular system (see the corresponding discussion in \Cref{sec:recallestimate}), the CGLS and LSQR algorithms can also be applied if $\mathrm{rank}(A) < n$. In such case, however, the least-squares solution of~\eqref{eq:lsproblem} is not unique.

\subsection{Error estimation in CGLS and LSQR}
Let $b_{|\mathcal{R}(A)}$ be the orthogonal projection of $b$ onto the range of $A$, and denote
\[
   r= b - b_{|\mathcal{R}(A)}.
\] 
In this section we focus on the estimation of the error
\[
    \| x - x_k \|^2_{A^TA} = \| A(x - x_k) \|^2 = \| b - r - Ax_k \|^2 = \| r_k - r \|^2 = \| r_k \|^2 - \| r \|^2,
\]
where we set $r_k = b - Ax_k$ and used the fact that $r \perp r - r_k$.
Therefore, if $b \in \mathcal{R}(A)$, which means that $r=0$, the error $\| x - x_k \|^2_{A^TA}$ is equal to the fully computable quantity $\| r_k \|^2 = \| b - Ax_k \|^2$ and there is no need for an estimator. 
Hereafter, we therefore assume that $b \notin \mathcal{R}(A)$. Then the norm $\| r \|$ is non-zero and unknown, and the residual norm $\| r_k \|$ alone may not provide a sufficient information to set a proper stopping criterion. This has been discussed, e.g., in \cite{JiTi09}. It seems natural to set a stopping criterion based on comparison of (the norms of) the residual~$r$ and the iterative residual~$r_k$. Such discussion is, however, beyond the scope of this paper.

The error norm $\| x - x_k \|_{A^TA}$ is the relevant quantity for the CGLS and LSQR algorithms as it is minimized within the iterations (it is therefore monotone) and, as we will see in \Cref{sec:precond}, this norm is also minimized if a preconditioning is considered. In many practical situations, the users are also interested in estimating the Euclidean norm of the error $\| x - x_k \|$. This error norm can be efficiently estimated in CG without preconditioning, based on  a direct relation between the Euclidean and the energy norm of the error, see, e.g. \cite{HeSt1952} or \cite{StTi02}, or, based on the relation of CG with Gauss quadrature \cite{GoSt1994,Me2005}. Then similar techniques for improving the accuracy of the estimate, that use an analog of \eqref{eq:errkkd}, can be developed. All these results 
are transferable from CG to CGLS and LSQR, but, again, without preconditioning. When using preconditioning, there is no direct relation between the
Euclidean and the energy norm of the error in general. However, one can bound the Euclidean norm of the error using
$$
    \Vert x - x_k \Vert \leq \frac{1}{\sigma}\, \Vert x - x_k\Vert_{A^TA}
$$
that requires an a~priori knowledge of a lower bound $\sigma$ on the smallest singular value of~$A$; see, e.g., \cite{Ar2013, EsOrSa2019,Hal20}. Clearly, the above upper bound need not represent a good estimate of $\| x - x_k \|$, but, in general, we do not have anything better in hands. Note that to use the above upper bound, we still need to have an estimate of the quantity $\| x - x_k \|_{A^TA}$, which is our aim in this paper.

\subsection{Error estimate in CGLS}
The first step to derive an estimate as in~\cite{MePaTi2021} is to find an expression analogous to~\eqref{eq:errkkd}.
By a simple algebraic manipulation, see \cite[p. 798]{StTi05}, we obtain
\begin{equation}\label{eqn:LSQR_HS}
\left\Vert x-x_{k}\right\Vert _{A^{T}A}^{2}-\left\Vert x-x_{k+1}\right\Vert _{A^{T}A}^{2} = \Vert x_{k+1}-x_{k} \Vert_{A^{T}A}^2 + 2\left(x_{k+1}-x_{k}\right)^{T}A^{T}A\left(x-x_{k+1}\right).
\end{equation}
In the following we use only relations that do hold (up to a small inaccuracy) also during computations in finite-precision arithmetic. To shorten the terminology, we say that these identities hold {\em numerically}.  In particular, we avoid using global orthogonality of vectors that is usually lost quickly.
Since in CGLS
\begin{equation}\label{eqn:rel1}
    x_{k+1}-x_{k} = \cfa_{k} p_k,\qquad s_{k+1} = A^TA(x-x_{k+1})
\end{equation}
we obtain 
$$
   \left\Vert x-x_{k}\right\Vert _{A^{T}A}^{2}-\left\Vert x-x_{k+1}\right\Vert _{A^{T}A}^{2} =  \cfa_k^2 \| p_k \|_{A^TA}^2 
    + 2\cfa_{k} p_k^{T}s_{k+1}.
$$
Using the formula for computing $\gamma_k$ (line 6)
and the (local) orthogonality between $s_{k+1}$ and $p_k$,
\begin{equation}\label{eqn:rel2}
  \cfa_k = \frac{\|s_k\|^2}{\| p_{k} \|_{A^TA}^2},\quad  p_k^T s_{k+1} = 0,
\end{equation}
we get an analog to~\eqref{eq:errkkd},
\begin{equation*}
    \| x - x_{\ell} \|^2_{A^TA} =  \sum_{j=\ell}^{k} \cfa_{j} \|s_j\|^2 + \| x - x_{k+1} \|^2_{A^TA}
\end{equation*}
and the error estimator 
\[
    \Delta^{\mathrm{CGLS}}_{\ell:k} \equiv \sum_{j=\ell}^{k} \cfa_{j} \|s_j\|^2 \approx \| x - x_{\ell} \|^2_{A^TA} = \| r_{\ell} \|^2 - \| r \|^2.
\]
To derive the analog of~\eqref{eq:errkkd} we have used relations 
\eqref{eqn:rel1} and \eqref{eqn:rel2}, which do hold numerically until the level of maximal attainable accuracy is reached; see, e.g., \cite{StTi02}.

\subsection{Error estimate in LSQR}
\label{sec:LSQR_deriveLB}
For LSQR, the situation is more delicate. Here the explanation needs more space and goes hand in hand with a derivation, inspired by ideas of \cite{StTi02}.

The vectors $u_i$ and $v_i$ generated in LSQR (\Cref{alg:LSQR}) by  Golub--Kahan bidiagonalization satisfy
\begin{eqnarray*}
AV_{k} & = & U_{k+1}B_{k}\\
A^{T}U_{k+1} & = & V_{k}B_{k}^{T}+\alpha_{k+1}v_{k+1}e_{k+1}^{T}
\end{eqnarray*}
where $V_k = [v_1, v_2, \ldots, v_k]$, $U_{k+1} = [u_1, u_2, \ldots, u_{k+1}]$, and
\[
B_{k}=\left[\begin{array}{ccccc}
\alpha_{1}\\
\beta_{2} & \ddots\\
 &  \ddots & \ddots\\
 &  & \beta_{k-1} & \alpha_{k-1}\\
 &  &  & \beta_{k} & \alpha_{k}\\
 &  &  &  & \beta_{k+1}
\end{array}\right].
\]
By a simple algebraic manipulation we get 
$$
A^{T}AV_{k}=V_{k}T_k+\alpha_{k+1}\beta_{k+1}v_{k+1}e_{k+1}^{T},\qquad T_k = B_{k}^{T}B_{k}.
$$
so that $v_i$ can be seen as Lanczos vectors generated by  the corresponding three-term recurrence for $A^TA$.

Assuming $x_{0}=0$, the LSQR approximation is given by 
\begin{equation}\label{eq:LSQR}
x_{k}=V_{k}y_{k},\quad y_k = \arg \min\|\beta_{1}e_{1}-B_{k}y\|
\end{equation}
and the corresponding least-squares problem is solved 
using the QR factorization of $B_{k}$. In particular, 
\[
Q_{k} \begin{bmatrix}B_{k} & \beta_{1}e_{1}\end{bmatrix}=\begin{bmatrix}R_{k} & f_{k}\\
 & \tilde{\phi}_{k+1}
\end{bmatrix}=\left[\begin{array}{cccccc}
\rho_{1} & \theta_{2} &  &  &  & \phi_{1}\\
 & \rho_{2} & \theta_{3} &  &  & \phi_{2}\\
 &  & \ddots & \ddots &  & \vdots\\
 &  &  & \rho_{k-1} & \theta_{k} & \phi_{k-1}\\
 &  &  &  & \rho_{k} & \phi_{k}\\
 &  &  &  &  & \tilde{\phi}_{k+1}
\end{array}\right],
\]
and $y_k$ is the solution of $R_k y = f_k$.
Therefore, 
$
T_{k}=B_{k}^{T}B_{k}=R_{k}^{T}R_{k}
$,
so that $R_{k}^{T}R_{k}$ represents Cholesky factorization of $T_{k}$. Since CGLS (CG applied to $A^{T}Ax=A^{T}b$) computes (implicitly) also the Cholesky factorization of $T_{k}$, see, e.g., \cite{MeTi2013}, we get the relation among coefficients that appear in both algorithms,
\begin{equation}\label{eq:coeff}
\left[\begin{array}{cccc}
\frac{1}{\sqrt{\cfa_{0}}}\\
\sqrt{\frac{\delta_{1}}{\cfa_{0}}} & \ddots\\
 & \ddots & \ddots\\
 &  & \sqrt{\frac{\cfb_{k-1}}{\cfa_{k-2}}} & \frac{1}{\sqrt{\cfa_{k-1}}}
\end{array}\right]=\left[\begin{array}{ccccc}
\rho_{1}\\
\theta_{2} & \ddots\\
 &  & \ddots\\
 &  & \theta_{k-1} & \rho_{k-1}\\
 &  &  & \theta_{k} & \rho_{k}
\end{array}\right]
\end{equation}
where $\cfa_{i}$ and $\cfb_{i}$ are the CGLS coefficients; see \Cref{alg:cgls}.
\smallskip

It is well known that during finite-precision computations, the {\em global} orthogonality 
among the vectors in columns of $V_k$ and $U_k$ is usually lost very quickly. In other 
words, one cannot expect that $V_k^T V_k = I = U_k^T U_k$. Our aim is to explain 
that despite the loss of global orthogonality and under 
some natural assumptions which will be specified later,
the Hestenes--Stiefel formula for LSQR, which has the form
\begin{equation}
\left\Vert x-x_{k-1}\right\Vert _{A^{T}A}^{2}-\left\Vert x-x_{k}\right\Vert _{A^{T}A}^{2}=\phi_{k}^{2},\label{eq:keyidentity}
\end{equation}
still holds numerically.
Note that we cannot directly apply results of \cite{StTi02,StTi05}
since LSQR (Algorithm~\ref{alg:LSQR}) uses different recurrences than CG
(Algorithm~\ref{alg:cg}). We do not present here a detailed rounding error 
analysis for the formula \eqref{eq:keyidentity} like in \cite{StTi02,StTi05}.
Instead, we focus on a justification based on a convenient derivation of the formula \eqref{eq:keyidentity} and the discussion about preserving local orthogonality.
We start with the following lemma; see also \cite[p.~52]{PaSa1982}.
\begin{lem}\label{lem:LSQR}
For the quantities generated by LSQR (Algorithm~\ref{alg:LSQR}) it holds that 
\begin{eqnarray}
A^{T}(b-Ax_{k}) & = & -\theta_{k+1}\phi_{k}v_{k+1}.\label{eq:mulvk1}
\end{eqnarray}
\end{lem}

\begin{proof}
Note that $y_{k}$ is the solution of the least-squares problem \eqref{eq:LSQR} so that $B_{k}^{T}B_{k}y_{k}=B_{k}^{T}\beta_{1}e_{1}$.
Since $x_{k}=V_{k}y_{k}$ and $AV_{k}=U_{k+1}B_{k}$, we get
\begin{eqnarray*}
A^{T}(b-Ax_{k}) & = & A^{T}U_{k+1}\left(\beta_{1}e_{1}-B_{k}y_{k}\right)\\
 & = & \left(V_{k}B_{k}^{T}+\alpha_{k+1}v_{k+1}e_{k+1}^{T}\right)\left(\beta_{1}e_{1}-B_{k}y_{k}\right)\\
 & = & V_{k}\left(B_{k}^{T}\beta_{1}e_{1}-B_{k}^{T}B_{k}y_{k}\right)-\alpha_{k+1}v_{k+1}e_{k+1}^{T}B_{k}y_{k}\\
 & = & -\alpha_{k+1}v_{k+1}\left[e_{k+1}^{T}B_{k}y_{k}\right]\\
 & = & \alpha_{k+1}v_{k+1}e_{k+1}^{T}\left(\beta_{1}e_{1}-B_{k}y_{k}\right),
\end{eqnarray*}
where we have used $e_{k+1}^{T}e_1=0$.
Using the QR factorization 
\[
\begin{bmatrix}B_{k} & \beta_{1}e_{1}\end{bmatrix}=Q_{k}^{T}\begin{bmatrix}R_{k} & f_{k}\\
 & \bar{\phi}_{k+1}
\end{bmatrix}
\]
and $R_{k}y_{k}=f_{k}$ we obtain 
\[
\beta_{1}e_{1}-B_{k}y_{k}=Q_{k}^{T}\begin{bmatrix}0\\
\bar{\phi}_{k+1}
\end{bmatrix};
\]
see \cite[equation (4.8)]{PaSa1982}. Hence, the last component of $\beta_{1}e_{1}-B_{k}y_{k}$
is given by $-c_{k}\bar{\phi}_{k+1}$. 
Moreover, using $\alpha_{k+1}c_{k}\bar{\phi}_{k+1}=\left(\alpha_{k+1}s_{k}\right)\left(c_{k}\bar{\phi}_{k}\right)=\theta_{k+1}\phi_{k}$
we get 
\begin{eqnarray*}
A^{T}(b-Ax_{k}) & = & -\alpha_{k+1}c_{k}\bar{\phi}_{k+1}v_{k+1}=
-\theta_{k+1}\phi_{k}v_{k+1}.
\end{eqnarray*}
\end{proof}

Since the relations $x_{k}=V_{k}y_{k}$ and $AV_{k}=U_{k+1}B_{k}$
hold also for the computed vectors and coefficients
 (up to an inaccuracy comparable to machine precision $\varepsilon$, norms of the vectors, and $\|A\|^{2}$), and since
the computation of $R_{k}$ is backward stable, we can expect that
\eqref{eq:mulvk1} holds numerically, until the level of maximal attainable accuracy is reached. 

Now we derive an analog of \eqref{eq:keyidentity} 
without using any orthogonality assumption. We will see that the resulting formula will contain {\em local} orthogonality terms 
\begin{equation}
\text{\ensuremath{\ell_{k}^{(w)}\equiv\frac{\theta_{k+1}}{\rho_{k}}v_{k+1}^{T}w_{k}}}
\label{eq:localterms}
\end{equation}
which are zero in exact arithmetic.

\begin{thm}\label{thm:LSQR}
For the quantities generated by LSQR (Algorithm~\ref{alg:LSQR}) it holds that 
\[
\left\Vert x-x_{k-1}\right\Vert _{A^{T}A}^{2}-\left\Vert x-x_{k}\right\Vert _{A^{T}A}^{2}=\phi_{k}^{2}\left({1+\ell_{k-1}^{(w)}}-\ell_{k}^{(w)}\right).
\]
\end{thm}

\begin{proof}
By a simple algebraic manipulation we obtain
\begin{equation*}
\left\Vert x-x_{k-1}\right\Vert _{A^{T}A}^{2}-\left\Vert x-x_{k}\right\Vert _{A^{T}A}^{2} = \left(x_{k}-x_{k-1}\right)^{T}A^{T}A\left(x-x_{k-1}\right) + \left(x_{k}-x_{k-1}\right)^{T}A^{T}A\left(x-x_{k}\right),
\end{equation*}
which is just a minor modification of the formula \eqref{eqn:LSQR_HS}.
Using the relations 
\[
A^{T}A(x-x_{k})=-\theta_{k+1}\phi_{k}v_{k+1},\quad x_{k}-x_{k-1}=\frac{\phi_{k}}{\rho_{k}}w_{k}
\]
we obtain 
\[
\left\Vert x-x_{k-1}\right\Vert _{A^{T}A}^{2}-\left\Vert x-x_{k}\right\Vert _{A^{T}A}^{2}=\phi_{k}^{2}\left( \left[-\frac{\theta_k \phi_{k-1}}{\rho_k\phi_k}\right]w_{k}^{T}v_{k}-\frac{\theta_{k+1}}{\rho_{k}}v_{k+1}^{T}w_{k}\right).
\]
Note that 
\begin{equation}
\label{eq:sknorm1}    
-\frac{\theta_k \phi_{k-1}}{\rho_k\phi_k} = 
-\frac{(s_{k-1} \alpha_k) (c_{k-1}\bar{\phi}_{k-1})}{\rho_k\phi_k}
=
\frac{\bar{\rho}_k \bar{\phi}_{k}}{\rho_k\phi_k}
=
1
\end{equation}
so that
\[
\left\Vert x-x_{k-1}\right\Vert _{A^{T}A}^{2}-\left\Vert x-x_{k}\right\Vert _{A^{T}A}^{2}=\phi_{k}^{2}\left(w_{k}^{T}v_{k}-\frac{\theta_{k+1}}{\rho_{k}}v_{k+1}^{T}w_{k}\right).
\]

Finally, from line~19 of \Cref{alg:LSQR},
\[
v_{k}^{T}w_{k}=
v_{k}^{T}\left(v_{k}+\frac{\theta_{k}}{\rho_{k-1}}w_{k-1}\right)=
1+\frac{\theta_{k}}{\rho_{k-1}}v_{k}^{T}w_{k-1}.
\]
\end{proof}

As a consequence, \eqref{eq:keyidentity} should hold numerically, if the local orthogonality terms \eqref{eq:localterms} are small (meaning that their magnitude is much less than one) and if \eqref{eq:mulvk1}  holds numerically (if the level of maximal attainable accuracy have not been reached yet).
Note that \eqref{eq:mulvk1} has been used  in the proof of Theorem~\ref{thm:LSQR}.
Hence, the problem of justification of the identity \eqref{eq:keyidentity} in finite-precision arithmetic is in this way reduced to the problem of bounding  local orthogonality between the computed vectors $v_{k+1}$ and $w_k$. 

To understand better the terms $\ell_k^{(w)}$, let us realize that LSQR computes in exact arithmetic the same approximations $x_k$ as CGLS. Therefore, 
\begin{equation}
\label{eq:sknorm2}
    s_k = A^T (b-Ax_k) = -\theta_{k+1} \phi_k v_{k+1},
\end{equation}
where $s_k = A^T r_k$ are computed in \Cref{alg:cgls}. By comparing  $x_k-x_{k-1}$ in both algorithms, we obtain
$$
    \cfa_{k-1} p_{k-1} = \frac{\phi_k}{\rho_k} w_k\,.
$$
Therefore, CGLS and LSQR
compute the same vectors, just scaled differently. In particular, 
$w_k$ is a multiple of $p_{k-1}$ and $v_{k+1}$ is a multiple of $s_k$. Both algorithms can be seen
as a variant of CG applied to the system of normal equations $A^TAx=A^Tb$.
Finally, realizing that $\cfa_{k-1}=\rho_k^{-2}$, see \eqref{eq:coeff}, we obtain
$$
    \ell_{k}^{(w)} =\frac{\theta_{k+1}}{\rho_{k}}v_{k+1}^{T}w_{k}
    = -\frac{s_k^Tp_{k-1}}{\rho^2_k\phi^2_k}
    = -\frac{s_k^Tp_{k-1}}{\Vert s_{k-1}\Vert^2},
$$
where the last equality follows from \eqref{eq:sknorm1} and \eqref{eq:sknorm2}, which gives $\| s_{k-1} \|^2 = \theta^2_{k} \phi^2_{k-1}$.
Therefore, the local orthogonality term $\ell_k^{(w)}$ 
can be seen as a counterpart of the term that was 
analyzed in detail for CG; see \cite[Theorem 9.1]{StTi02}.
Based on results of \cite{StTi02}
one can expect that the size of
$\ell_k^{(w)}$ can be bounded
by machine precision 
$\varepsilon$ multiplied by some factor that 
can depend on $\kappa(A^TA)$ (or its analog if $A$ is singular), dimension of the problem and the number of iterations. However, a proper rounding error analysis leading to the mentioned result is beyond the scope of this paper.
Below we check the validity of \eqref{eq:keyidentity} 
numerically, by plotting the relative error
\begin{equation}
\label{eq:relativeerrornumexp}
   \frac{\left| \big( \left\Vert x-x_{k-1}\right\Vert _{A^{T}A}^{2}-\left\Vert x-x_{k}\right\Vert _{A^{T}A}^{2} \big) - \phi_{k}^{2} \right| }{\left\Vert x-x_{k-1}\right\Vert _{A^{T}A}^{2}-\left\Vert x-x_{k}\right\Vert _{A^{T}A}^{2}}
\end{equation}
for the challenging (in terms of numerical stability) test case from~\cite{BjElSt1998} with the conditioning $\kappa(A^{T}A) = 10^{12}$; see Figure~\ref{fig:validity_of_9}. We observe that the relative error is (approximately) inversely proportional to the error~$\left\Vert x-x_{k}\right\Vert _{A^{T}A}$ and stays significantly below 1 until the maximal attainable accuracy is reached. 

\begin{figure}[htp]
    \centering
    \includegraphics[width=0.5\textwidth]{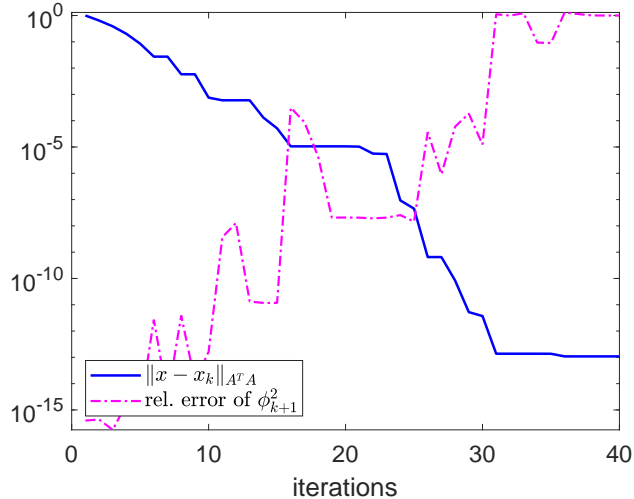}
    \caption{Test case from~\cite[p.~733]{BjElSt1998}: error norm and the relative error \eqref{eq:relativeerrornumexp}.}
    \label{fig:validity_of_9}
\end{figure}
\smallskip

In summary, \eqref{eq:keyidentity} results in 
\begin{equation} \label{eq:HSlsqr}
    \| x - x_{\ell} \|^2_{A^TA} =  \sum_{j=\ell}^{k}{\phi}_{j+1}^2 + \| x - x_{k+1} \|^2_{A^TA}
\end{equation}
yielding the error estimator 
\[
    \Delta^{\mathrm{LSQR}}_{\ell:k} \equiv \sum_{j=\ell}^{k}{\phi}_{j+1}^2 \approx \| x - x_{\ell} \|^2_{A^TA} = \| r_{\ell} \|^2 - \| r \|^2.
\]
The error estimator contains only scalars that are available during the LSQR iterations and it is very cheap to evaluate. It should be reliable also during finite-precision computations, if local orthogonality between $v_k$ and $w_{k-1}$ is well preserved and until the level of maximal attainable accuracy is reached. 

\section{Estimating the error in least-norm problems}
\label{sec:Craigestim}

Let us now consider a least-norm problem
\begin{equation}
\label{eq:lnproblem}
x = \argmin_{z\in \mathbb{R}^n} \| z \| \quad \mbox{subject to} \quad Az = b,
\end{equation}
where $A \in \mathbb{R}^{m \times n}$ and $b \in \mathbb{R}^{m}$, $b \in \mathcal{R}(A)$.
A method of choice for solving this (consistent) problem might be Craig's method that is described in the following section.

\subsection{Craig's method}
\label{sec:Craig}

The idea of the method is to write the solution $x$ of~\eqref{eq:lnproblem} as $x = A^T y$ for a proper vector~$y$ and compute the approximations $x_k = A^T y_k$, where $y_k$ is an approximation to~$y$ given by the CG method applied to the system
\begin{equation}
    \label{eq:normeq2}
    AA^T y = b,
\end{equation}
starting with some $y_0$. 
Denoting $r_0 = b - AA^Ty_0 = b - Ax_0$, the approximations $y_k$ satisfy
\[
    y_k \in y_0 + \mathcal{K}_k(AA^T,r_0), \qquad \| y - y_k\|_{AA^T} = \min_{z \in y_0 + \mathcal{K}_k(AA^T,r_0)} \|y - z \|_{AA^T}.
\]
A simple algebraic manipulation shows that
\begin{equation}
\label{eq:Craigxminim}
    \|y - y_k \|^2_{AA^T} = \| x - x_k \|^2,
\end{equation}
which means that Craig's method minimizes the Euclidean norm of the error of~$x_k$ over the affine space 
\[
    A^Ty_0 + A^T\mathcal{K}_k(AA^T,r_0) = x_0 + \mathcal{K}_k(A^TA,A^Tr_0).
\]
The above discussion can be applied also if the matrix~$AA^T$ is singular; see the discussion in Section~\ref{sec:recallestimate}.

Analogously to solving the normal equations~\eqref{eq:normeq}, there are several ways how the CG method for~\eqref{eq:normeq2} can be applied. A study of numerical stability of various implementations was presented in~\cite{BjElSt1998}.

The first algorithm was proposed by Craig in \cite{Cra54thesis,Cr1955}, developed originally as an algorithm for solving unsymmetric systems (see the review in MathSciNet written by Forsythe). 
Fadeev and Fadeeva developed almost the same algorithm~\cite[p. 504]{B:FaFa1963} with the difference that instead of direction vectors, the algorithm updates the vectors $g_{k}\equiv A^{T}p_{k}$. The resulting algorithm is called CGNE and is listed in \Cref{alg:CGNE}. The same algorithm, up to a different notation, is also given in \cite[Algorithm 2.1]{BjElSt1998} when setting $\mu=0$ therein.
\begin{algorithm}[htp]
\caption{CGNE}
\label{alg:CGNE}

\begin{algorithmic}[1]

\STATE \textbf{input} $A$, $b$, $x_{0}$

\STATE $r_{0}=b-Ax_{0}$

\STATE $p_{0}=A^{T}r_{0}$

\FOR{$k=0,1,2,\dots$}

\STATE $\cfa{}_{k}= {\|r_{k}\|^{2}}/{\|p_{k}\|^{2}}$

\STATE $x_{k+1}=x_{k}+\cfa{}_{k}p_{k}$

\STATE $r_{k+1}=r_{k}-\cfa{}_{k}Ap_{k}$

\STATE $\cfb_{k+1}={\|r_{k+1}\|^{2}}/{\|r_{k}\|^{2}}$

\STATE $p_{k+1}=A^{T}r_{k+1}+\cfb_{k+1}p_{k}$

\ENDFOR

\end{algorithmic}
\end{algorithm}

Paige \cite{Pa1974} and Paige and Saunders \cite{PaSa1982} developed another mathematically equivalent version of this algorithm using the Golub--Kahan bidiagonalization. The algorithm can also compute the approximations~$y_{k}$ to $AA^{T}y=b$. We present it in \Cref{alg:Craig} and call it CRAIG, to be consistent with the notation introduced by Paige and Saunders.

\begin{algorithm}[htp]
\caption{CRAIG}
\label{alg:Craig}

\begin{algorithmic}[1]

\STATE \textbf{input} $A$, $b$

\STATE $\zeta_{0}=-1$ 

\STATE $x_{0}=y_{0}=d_{0}=0$

\STATE $\beta_{1}u_{1}=b$
\hfill
\textcolor{gray}{$\beta_{i}$ are the normalization coefficients to have $\|u_{i}\| = 1$, $i = 1,2, \ldots$}

\STATE $\alpha_{1}v_{1}=A^{T}u_{1}$
\hfill \textcolor{gray}{$\alpha_{i}$ are the normalization coefficients to have $\|v_{i}\| = 1$, $i = 1,2, \ldots$}

\FOR{$k=1,2,\dots$}

\STATE $\zeta_{k}=-\frac{\beta_{k}}{\alpha_{k}}\zeta_{k-1}$

\STATE $x_{k}=x_{k-1}+\zeta_{k}v_{k}$

\STATE $d_{k}=\left(u_{k}-\beta_{k}d_{k-1}\right)/\alpha_{k}$

\STATE $y_{k}=y_{k-1}+\zeta_{k}d_{k}$
\hfill \textcolor{gray}{lines 9--10 can be removed if the approximation $y_k$ is not needed}

\STATE $\beta_{k+1}u_{k+1}=Av_{k}-\alpha_{k}u_{k}$

\STATE $\alpha_{k+1}v_{k+1}=A^{T}u_{k+1}-\beta_{k+1}v_{k}$

\ENDFOR

\end{algorithmic}
\end{algorithm}
We now present a way to estimate the error~$\| x - x_k \|^2$ in
both implementations of Craig's method. Similarly to \eqref{eqn:LSQR_HS}, by a simple algebraic manipulation 
we obtain
\begin{equation}\label{eqn:Craig_HS}
\left\Vert x-x_{k}\right\Vert^2-\left\Vert x-x_{k+1}\right\Vert^2 = \Vert x_{k+1}-x_{k}\Vert^2 + 2\left(x_{k+1}-x_{k}\right)^{T}\left(x-x_{k+1}\right).
\end{equation}

\subsection{Error estimate in CGNE}
\label{sec:CGNE}
Let us first realize that, in CGNE,
$$
\Vert x_{k+1}-x_{k}\Vert^2 = \cfa_k^2 \| p_k\|^2 = \cfa_k \| r_k\|^2,
$$
where we have used $\cfa_k \| p_k\|^2 = \| r_k\|^2$.

Consider formally the recurrence
for computing vectors $y_k$ such that $x_k=A^T y_k$, 
and the recurrence for computing the corresponding direction 
vectors $s_k$ such that $p_k=A^T s_k$,
\begin{eqnarray*}
    y_{k+1} &=& y_k + \cfa_{k} s_k,\\
     s_{k+1} &=& r_{k+1} + \cfb_{k+1} s_k, 
\end{eqnarray*}
with $y_0$ satisfying $x_0=A^T y_0$ and $s_0=r_0$.  Then the last term on the right hand side of \eqref{eqn:Craig_HS} corresponds to 
$$
\left(x_{k+1}-x_{k}\right)^{T}\left(x-x_{k+1}\right) = 
\left(y_{k+1}-y_{k}\right)^{T}A\left(x-x_{k+1}\right)
=
\cfa_{k} s_k^T r_{k+1},
$$
so that 
$$
\left\Vert x-x_{k}\right\Vert^2-\left\Vert x-x_{k+1}\right\Vert^2 = \cfa_k \| r_k\|^2 + 2\cfa_{k} s_k^T r_{k+1}.
$$
Using an analog of results from \cite{StTi02} for CG applied to $AA^Ty=b$ 
one can expect that the size of the term $s_k^T r_{k+1}$ should be negligible
in comparison to $ \Vert r_{k}\Vert^2$
and that the identity 
\[
    \| x - x_{k} \|^2 - \| x - x_{k+1} \|^2  
    = \cfa_k \| r_k\|^2 
\]
holds numerically, until the level of maximal attainable accuracy is reached.
Therefore,
for CGNE we obtain an analog to~\eqref{eq:errkkd} in the form,
\[
    \| x - x_{{\ell}} \|^2 = \sum_{j={\ell}}^{{k}} \cfa_j \| r_j \|^2 + \| x - x_{{k+1}} \|^2
\]
giving the error estimator
\[
    \Delta^{\mathrm{CGNE}}_{{\ell:k}} \equiv \sum_{j={\ell}}^{{k}} \cfa_j \| r_j \|^2 \approx \| x - x_{{\ell}} \|^2.
\]

\subsection{Error estimate in CRAIG}
\label{sec:CRAIG}
Considering the formula \eqref{eqn:Craig_HS} for the iteration index $k-1$, 
and using $x_{k}=A^Ty_k$ we obtain
$$
\left\Vert x-x_{k-1}\right\Vert^2-\left\Vert x-x_{k}\right\Vert^2 = \Vert x_{k}-x_{k-1}\Vert^2 + 2\left(y_{k}-y_{k-1}\right)^{T}A\left(x-x_{k}\right).
$$
Let us first realize that
$$
\Vert x_{k}-x_{k-1}\Vert^2 = \zeta_k^2 \| v_k\|^2 = \zeta_k^2.
$$
Using a similar technique as in Lemma~\ref{lem:LSQR} one can prove that 
$$
    A(x-x_k) = AA^T(y-y_k)= - \zeta_{k}\beta_{k+1}u_{k+1},
$$
so that, from line~10 of \Cref{alg:Craig},
$$
\left(y_{k}-y_{k-1}\right)^{T}A\left(x-x_{k}\right)
=
-\zeta_{k}^2 \beta_{k+1} d_k^T u_{k+1}.
$$
Therefore, assuming that the maximal level of accuracy has not been reached yet
and that local orthogonality between $d_k$ (direction vector in CG for \eqref{eq:normeq2}) and $u_{k+1}$ (scaled residual vector in CG for \eqref{eq:normeq2}) is well preserved,
the identity 
\[
    \| x - x_{k-1} \|^2 -\| x - x_{k} \|^2 = \zeta_k^2
\]
leading to 
\[
    \| x - x_{{\ell}} \|^2 = \sum_{j={\ell}}^{{k}} \zeta^2_{j+1}  + \| x - x_{{k+1}} \|^2
\]
holds numerically. A detailed rounding error analysis concerning the preservation of local orthogonality in CRAIG is beyond the scope of this paper. Finally, the error estimator for CRAIG has the form
\[
    \Delta^{\mathrm{CRAIG}}_{{\ell:k}} \equiv \sum_{j={\ell}}^{{k}} \zeta^2_{j+1} \approx \| x - x_{{\ell}} \|^2.
\]

\section{Error estimation in preconditioned algorithms}
\label{sec:precond}

In this section we present the preconditioned variants of the algorithms and derive the associated error estimates. 
We restrict ourselves to the class of split preconditioners for normal equations, i.e., the preconditioners formally transforming the systems \eqref{eq:normeq} and \eqref{eq:normeq2} to
\[
    \big( L^{-1}A^TAL^{-T} \big) \big( L^T x\big) = L^{-1}A^T b \quad \mathrm{and} \quad 
    \big( L^{-1}AA^TL^{-T} \big) \big( L^T y\big) = L^{-1} b,\ \mbox{respectively}.
\]
In \cite[Section~1]{BruMaTu14}, several preconditioning techniques for least-squares problems are discussed. Most of them can be represented as a split preconditioner for normal equations. In general, typical representatives of split preconditioners are incomplete factorizations of matrices~$A^TA$, or $AA^T$. Efficient codes for computing such factorizations without explicitly forming~$A^TA$, or~$AA^T$ are available; see, e.g., {\tt HSL\_MI35} from \cite{HSL}.

First we will discuss preconditioning for CGLS and LSQR, and then for CGNE and CRAIG.

\subsection{Preconditioned CGLS and LSQR}

Let a nonsingular matrix $L\in \mathbb{R}^{n\times n}$ be given. For least-squares problems, we consider the modification of the original problem \eqref{eq:lsproblem} in the form
\begin{equation}
    x =  \argmin_{z \in \mathbb{R}^n} \| b - \underbrace{AL^{-T}}_{\hat{A}}\underbrace{L^Tz}_{\phantom{\hat{I}}\hat{z}\phantom{\hat{I}}} \|
    \;=\;
    \argmin_{\hat{z} \in \mathbb{R}^n} \| b - \hat{A}\hat{z} \|,
\label{eq:prec_LS}
\end{equation}
leading to the corresponding system of normal equations
\begin{equation}
\underbrace{L^{-1} A^T}_{{\hat{A}}^T} \underbrace{AL^{-T}}_{\hat{A}}
\underbrace{L^Tx}_{\phantom{\hat{I}}\hat{x}\phantom{\hat{I}}}
= \underbrace{L^{-1} A^T}_{{\hat{A}}^T} b.
\label{eq:prec_normeq}
\end{equation}
Hence, $L$ can be seen as a split preconditioner for the matrix $A^TA$.
Let $\hat{x}_k$ be the $k$th CG approximate solution for the preconditioned system, and define by $x_k = L^{-T}\hat{x}_k$ an approximation to the solution~$x$ of the original problem \eqref{eq:lsproblem}. Recalling from~\eqref{eq:prec_normeq} that $\hat{x} = L^{T}x$ and $\hat{A} = AL^{-T}$, we have
\begin{equation}
    \| \hat{x} - \hat{x}_{k} \|^2_{\hat{A}^T\hat{A}} 
    = (\hat{x} - \hat{x}_{k})^T \hat{A}^T\hat{A} (\hat{x} - \hat{x}_{k})
    = (x - x_k)^T L L^{-1} A^TA L^{-T} L^T (x - x_k) = \| {x} - {x}_{k} \|^2_{{A}^T{A}}.
\label{eq:prec_LS_errors}
\end{equation}

The CGLS algorithm for solving~\eqref{eq:prec_LS}, based on \Cref{alg:cgls}, is given in \Cref{alg:form_pcgls}.
Simple algebraic manipulations show that the residual $r_k = b - Ax_k$ for $x_k = L^{-T}\hat{x}_k$ satisfies $r_k = \hat{r}_k$. This allows us to write a variant of preconditioned CGLS as in \Cref{alg:pcgls}, where the approximations~$x_k$ and the associated residuals~$r_k$ are explicitly computed.

Analogously to \Cref{sec:LSQR_CGLS}, we can show that the errors satisfy 
\begin{equation*}
    \| \hat{x} - \hat{x}_{k} \|^2_{\hat{A}^T\hat{A}} =  \sum_{j={\ell}}^{{k}} \hat{\cfa}_{j} \|\hat{s}_j\|^2 + \| \hat{x} - \hat{x}_{{k+1}} \|^2_{\hat{A}^T\hat{A}}.
\end{equation*}
Therefore we can estimate the error in preconditioned CGLS using the estimator
\[
    \Delta^{\mathrm{PCGLS}}_{{\ell:k}} \equiv \sum_{j={\ell}}^{{k}} \hat{\cfa}_{j} \|\hat{s}_j\|^2 \approx \| x - x_{{\ell}} \|^2_{A^TA},
\]
with the same favourable properties as $\Delta^{\mathrm{CGLS}}_{{\ell:k}}$.

The algorithm and the estimator can easily be modified also for the case when a preconditioner is not available in the factorized form. The \emph{factorization-free} variant of PCGLS can be found, e.g., in \cite[Algorithm~4]{RegSau20}.

\noindent
\begin{minipage}{0.46\textwidth}
\begin{algorithm}[H]
\caption{Formally preconditioned CGLS}
\label{alg:form_pcgls}

\begin{algorithmic}[1]

\STATE \textbf{input} $\hat{A}$, $\hat{b}$, $\hat{x}_{0}$

\STATE $\hat{r}_{0}=\hat{b}-\hat{A}\hat{x}_{0}$

\STATE $\hat{s}_{0}=\hat{p}_{0} = \hat{A}^{T}\hat{r}_{0}$

\FOR{$k=0,1,\dots$}

\STATE $\hat{q}_{k}=\hat{A}\hat{p}_{k}$
\STATE $\hat{\cfa}_{k} = \| \hat{s}_{k} \|^2 / \| \hat{q}_{k} \|^2$
\STATE $\hat{x}_{k+1}=\hat{x}_{k}+\hat{\cfa}_{k}\hat{p}_{k}$
\STATE $\hat{r}_{k+1}=\hat{r}_{k}-\hat{\cfa}_{k}\hat{q}_{k}$
\STATE $\hat{s}_{k+1}=\hat{A}^T\hat{r}_{k+1}$
\STATE $\hat{\cfb}_{k+1}=\| \hat{s}_{k+1} \|^2/\| \hat{s}_{k} \|^2$
\STATE $\hat{p}_{k+1}=\hat{s}_{k+1}+\hat{\cfb}_{k+1}\hat{p}_{k}$

\ENDFOR
\\[\baselineskip]

\end{algorithmic}
\end{algorithm}
\end{minipage}
\hfill
\begin{minipage}{0.46\textwidth}
\begin{algorithm}[H]
\caption{Preconditioned CGLS}
\label{alg:pcgls}

\begin{algorithmic}[1]

\STATE \textbf{input} $A$, $b$, $x_0$, $L$

\STATE $r_{0}= {b}-{A}{x}_{0}$

\STATE $\hat{s}_{0}=\hat{p}_{0} = L^{-1}{A}^T {r}_{0}$

\FOR{$k=0,1,\dots$}

\STATE $\hat{t}_{k}=L^{-T}\hat{p}_{k}$
\STATE $\hat{q}_{k}={A}t_{k}$
\STATE $\hat{\cfa}_{k} = \| \hat{s}_{k} \|^2 / \| \hat{q}_{k} \|^2$
\STATE ${x}_{k+1}={x}_{k}+\hat{\cfa}_{k}\hat{t}_{k}$
\STATE ${r}_{k+1}={r}_{k}-\hat{\cfa}_{k}\hat{q}_{k}$
\STATE $\hat{s}_{k+1}=L^{-1}{A}^T {r}_{k+1}$
\STATE $\hat{\cfb}_{k+1}=\| \hat{s}_{k+1} \|^2/\| \hat{s}_{k} \|^2$
\STATE $\hat{p}_{k+1}=\hat{s}_{k+1}+\hat{\cfb}_{k+1}\hat{p}_{k}$

\ENDFOR

\end{algorithmic}
\end{algorithm}
\end{minipage}

\bigskip

For LSQR we can proceed analogously to CGLS and apply LSQR directly to~\eqref{eq:prec_LS}. After stopping iterations with~$\hat{x}_k$ in hands, the approximation to the solution~$x$ of the original system is computed as $x_k = L^{-T}\hat{x}_k$. The resulting algorithm is given in \Cref{alg:form_pLSQR}.

From the definition of the estimator $\Delta^{\mathrm{LSQR}}_{{\ell:k}}$ (considered for \eqref{eq:prec_LS}) and \eqref{eq:prec_LS_errors}, we derive the estimator for preconditioned LSQR
\[
    \Delta^{\mathrm{PLSQR}}_{{\ell:k}} \equiv \sum_{j={\ell}}^{{k}}\hat{\phi}_{j+1}^{\,2} \approx \| \hat{x} - \hat{x}_{{\ell}} \|^2_{\hat{A}^T\hat{A}} = \| {x} - {x}_{{\ell}} \|^2_{{A}^T{A}}.
\]

\begin{algorithm}[htp]
\caption{Preconditioned LSQR}
\label{alg:form_pLSQR}

\begin{algorithmic}[1]

\STATE \textbf{input} $A$, $b$, $L$

\STATE $\hat{\beta}_{1}\hat{u}_{1}=b$
\hfill\textcolor{gray}{$\hat{\beta}_{i}$ are the normalization coefficients to have $\|\hat{u}_{i}\| = 1$, $i = 1,2, \ldots$}
\STATE $\hat{\alpha}_{1}\hat{v}_{1}=L^{-1}A^{T}\hat{u}_{1}$
\hfill \textcolor{gray}{$\hat{\alpha}_{i}$ are the normalization coefficients to have $\|\hat{v}_{i}\| = 1$, $i = 1,2, \ldots$}
\STATE $\hat{w}_1 = \hat{v}_1$
\STATE $\hat{x}_0$ = 0
\STATE $\hat{\bar{\phi}}_1 = \hat{\beta}_1$
\STATE $\hat{\bar{\rho}}_1 = \hat{\alpha}_1$

\FOR{$k=1,2,\dots$}

\STATE $\hat{\beta}_{k+1}\hat{u}_{k+1}=\hat{A}\hat{v}_{k}-\hat{\alpha}_{k}\hat{u}_{k}
=AL^{-T}\hat{v}_{k}-\hat{\alpha}_{k}\hat{u}_{k}$
\STATE $\hat{\alpha}_{k+1}\hat{v}_{k+1}=\hat{A}^{T}\hat{u}_{k+1}-\hat{\beta}_{k+1}\hat{v}_{k}
=L^{-1}{A}^{T}\hat{u}_{k+1}-\hat{\beta}_{k+1}\hat{v}_{k}$
\STATE $\hat{\rho}_k = (\hat{\bar{\rho}}_k^2 + \hat{\beta}^2_{k+1})^{1/2}$
\STATE $\hat{c}_k = \hat{\bar{\rho}}_k / \hat{\rho}_k$
\STATE $\hat{s}_k = \hat{\beta}_{k+1} / \hat{\rho}_{k}$
\STATE $\hat{\theta}_{k+1} = \hat{s}_{k} \hat{\alpha}_{k+1}$
\STATE $\hat{\bar{\rho}}_{k+1} = -\hat{c}_{k} \hat{\alpha}_{k+1}$
\STATE $\hat{\phi}_{k} = \hat{c}_{k} \hat{\bar{\phi}}_{k}$
\STATE $\hat{\bar{\phi}}_{k+1} = \hat{s}_{k} \hat{\bar{\phi}}_{k}$

\STATE $\hat{x}_{k}=\hat{x}_{k-1} + (\hat{\phi}_{k} / \hat{\rho}_{k} ) \hat{w}_{k}$
\STATE $\hat{w}_{k+1}=\hat{v}_{k+1} + (\hat{\theta}_{k+1} / \hat{\rho}_{k} ) \hat{w}_{k}$

\ENDFOR
\STATE $x_{k}=L^{-T}\hat{x}_k$
\end{algorithmic}
\end{algorithm}

Factorization-free variant of preconditioned LSQR can be found, e.g., in \cite[Algorithm~2]{ArrBetHar14}.

\subsection{Preconditioned CGNE and CRAIG}
Let a nonsingular $L\in \mathbb{R}^{m \times m}$ be given. We consider the following modification 
of the original problem \eqref{eq:lnproblem}
\[
    x=\argmin_{z\in\mathbb{R}^{n}} \| z \| \quad \mbox{subject to} \quad \underbrace{L^{-1}A}_{\hat{A}} z = \underbrace{L^{-1}b}_{\hat{b}}\,,
\]
and, as before, the approximate solutions $x_k$ can be obtained 
when applying the considered algorithms to the underlying preconditioned system 
\[
    \underbrace{L^{-1}A}_{\hat{A}} \underbrace{A^TL^{-T}}_{\hat{A}^T} \underbrace{L^T y}_{\hat{y}} = \underbrace{L^{-1}b}_{\hat{b}}.
\]
Therefore, $L$ can be seen as a split preconditioner for the matrix $AA^T$.
We note that no transformation of the computed approximation is needed to get an approximation to the solution~$x$. Analogously to \Cref{sec:Craig}, we can show that the preconditioned variants minimize the Euclidean norm of the error $\| x - x_k \|^2$, now over the affine space $x_0 + \mathcal{K}_k(L^{-1}A^TAL^{-T},A^TL^{-T}L^{-1}r_0)$.

\begin{algorithm}[bt]
\caption{Preconditioned CGNE}
\label{alg:pCGNE}

\begin{algorithmic}[1]

\STATE \textbf{input} ${A}$, ${b}$, ${x}_{0}$, $L$

\STATE $\hat{r}_{0}=L^{-1}({b}-{A}x_{0})$

\STATE $\hat{p}_{0}={A}^{T}L^{-T}\hat{r}_{0}$

\FOR{$k=0,1,2,\dots$}

\STATE $\hat{\cfa}{}_{k}= {\|\hat{r}_{k}\|^{2}}/{\|\hat{p}_{k}\|^{2}}$

\STATE $x_{k+1}=x_{k}+\hat{\cfa}{}_{k}\hat{p}_{k}$

\STATE $\hat{r}_{k+1}=\hat{r}_{k}-\hat{\cfa}{}_{k}L^{-1}{A}\hat{p}_{k}$

\STATE $\hat{\cfb}_{k+1}={\|\hat{r}_{k+1}\|^{2}}/{\|\hat{r}_{k}\|^{2}}$

\STATE $\hat{p}_{k+1}={A}^{T}L^{-T}\hat{r}_{k+1}+\hat{\cfb}_{k+1}\hat{p}_{k}$

\ENDFOR

\end{algorithmic}
\end{algorithm}
For CGNE (\Cref{alg:CGNE}) we obtain the preconditioned variant, \Cref{alg:pCGNE}.
Whenever the residuals $r_k = b - Ax_k$ are needed within the iterations, we can use the transformation $\hat{r}_k = L^{-1}r_k$ and replace lines~2 and~7 in \Cref{alg:pCGNE} by
\begin{align*}
    &2: {r}_{0} = {b}-{A}x_{0}; \quad \hat{r}_{0} = L^{-1}r_0 \\
    &7: {r}_{k+1}={r}_{k}-\hat{\cfa}{}_{k}{A}\hat{p}_{k}; \quad \hat{r}_{k} = L^{-1}r_k
\end{align*}

We can similarly precondition the GRAIG algorithm (\Cref{alg:Craig}). A version without computing the $y_{k}$ vectors is given in \Cref{alg:pCraig}.

\begin{algorithm}[ht]
\caption{Preconditioned CRAIG}
\label{alg:pCraig}

\begin{algorithmic}[1]

\STATE \textbf{input} $A$, $b$, $L$

\STATE $\hat{\zeta}_{0}=-1$

\STATE $x_{0}=0$

\STATE $\hat{\beta}_{1}\hat{u}_{1}=L^{-1}b$
\hfill \textcolor{gray}{$\hat{\beta}_{i}$ are the normalization coefficients to have $\|\hat{u}_{i}\| = 1$, $i = 1,2, \ldots$}

\STATE $\hat{\alpha}_{1}v_{1}=A^{T}L^{-T}\hat{u}_{1}$
\hfill \textcolor{gray}{$\hat{\alpha}_{i}$ are the normalization coefficients to have $\|\hat{v}_{i}\| = 1$, $i = 1,2, \ldots$}

\FOR{$k=1,2,\dots$}

\STATE $\hat{\zeta}_{k}=-\frac{\hat{\beta}_{k}}{\hat{\alpha}_{k}}\hat{\zeta}_{k-1}$

\STATE $x_{k}=x_{k-1}+\hat{\zeta}_{k}\hat{v}_{k}$

\STATE $\hat{\beta}_{k+1}\hat{u}_{k+1}=L^{-1}A\hat{v}_{k}-\hat{\alpha}_{k}\hat{u}_{k}$

\STATE $\hat{\alpha}_{k+1}\hat{v}_{k+1}=A^{T}L^{-T}\hat{u}_{k+1}-\hat{\beta}_{k+1}\hat{v}_{k}$

\ENDFOR

\end{algorithmic}
\end{algorithm}

Proceeding as in \Cref{sec:CGNE} and \Cref{sec:CRAIG}, we obtain the error estimators in PCGNE and PCRAIG,
\[
    \Delta^{\mathrm{PCGNE}}_{{\ell:k}} \equiv \sum_{j={\ell}}^{{k}} \hat{\cfa}_j \| \hat{r}_j \|^2 \approx \| x - x_{{\ell}} \|^2,
    \quad
     \Delta^{\mathrm{PCRAIG}}_{{\ell:k}} \equiv \sum_{j={\ell}}^{{k}} \hat{\zeta}^2_{j+1} \approx \| x - x_{{\ell}} \|^2.
\]

\section{Numerical experiments}
\label{sec:numexp}

For the numerical tests we consider several matrices and systems from the SuiteSparse\footnote{\url{https://sparse.tamu.edu}, \cite{SuiteSparse}} matrix collection, with the sizes

\medskip
\setlength{\tabcolsep}{2pt}
\begin{tabular}{lrcl}
\texttt{illc1033}*:& 1033 &$\times$& 320 \\
\texttt{illc1850}*:& 1850 &$\times$& 712 \\
\texttt{well1033}*:& 1033 &$\times$& 320 \\
\texttt{well1850}*:& 1850 &$\times$& 712 \\
\texttt{sls}: & 1\,748\,122 &$\times$& 62\,729 \\
\end{tabular}
\hfill
\begin{tabular}{lrcl}
\texttt{Delor64K}:& 64\,719 &$\times$& 1\,785\,345 \\
\texttt{Delor338K}:& 343\,236 &$\times$& 887\,058 \\
\texttt{flower\_7\_4}: & 27\,693 &$\times$& 67\,593 \\
\texttt{cat\_ears\_4\_4}: & 19\,020 &$\times$& 44\,448 \\
\texttt{lp\_pilot}: & 1441 &$\times$& 4860
\end{tabular}
\setlength{\tabcolsep}{6pt}
\medskip

\noindent
The matrices marked with an asterisk come together with the right-hand side~$b$. For the remaining problems, we generated~$b$ as follows (in MATLAB notation)
\begin{verbatim}
    x = ones(size(A,2),1);
    x(2:2:end) = -2;
    x(5:5:end) = 0;
    b_LN = A*x;
    b_LS = b_LN + randn(size(b_LN))*norm(b_LN);
\end{verbatim}
where $b_\mathrm{LN}$ is used for least-norm problems and $b_\mathrm{LS}$ for least-squares. The exact solution is then computed using MATLAB build-in function \verb+lsqminnorm+.

As in~\cite{MePaTi2021}, we plot the error quantity together with the (adaptive) lower bound and compare the adaptive value of delay~$k-\ell$ with the \emph{ideal} value, i.e.~the minimal delay ensuring the prescribed accuracy of the bound.

\subsection{Least-squares problems}

In Figures \ref{fig:illc1033}--\ref{fig:sls} we plot the results for our adaptive error estimate in least-squares problems solved by CGLS and LSQR. We can see that the estimate mostly follows the error very tightly with nearly the optimal delay, even in the cases with almost stagnation where the optimal delay is very large; cf.~\Cref{fig:illc1033}. We observe some underestimation in initial iterations for \texttt{illc1033} matrix but in later iterations, where one typically needs an estimate for stopping the solver, the delay is close to the ideal value. For the matrix \texttt{sls}, we note that the adaptively chosen value of delay is higher than needed.

\begin{figure}[htp]
    \centering
    \begin{minipage}{0.45\textwidth}
    \centering
    matrix \texttt{\detokenize{illc1033}}, CGLS\\
    \includegraphics[width=\textwidth]{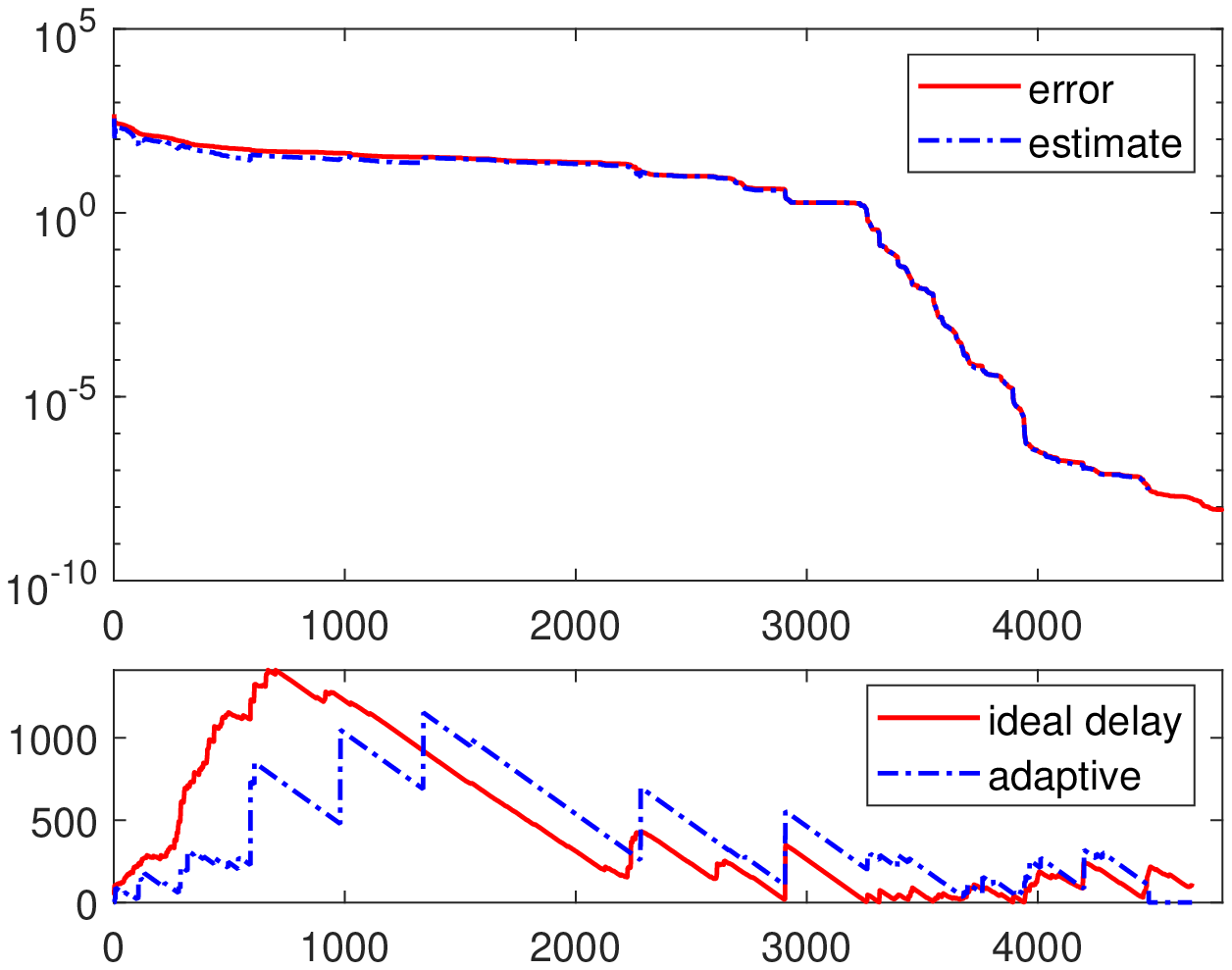}\\
    \textsf{iterations}
    \end{minipage}
    \hfill
    \begin{minipage}{0.45\textwidth}
    \centering
    matrix \texttt{\detokenize{illc1033}}, LSQR\\
    \includegraphics[width=\textwidth]{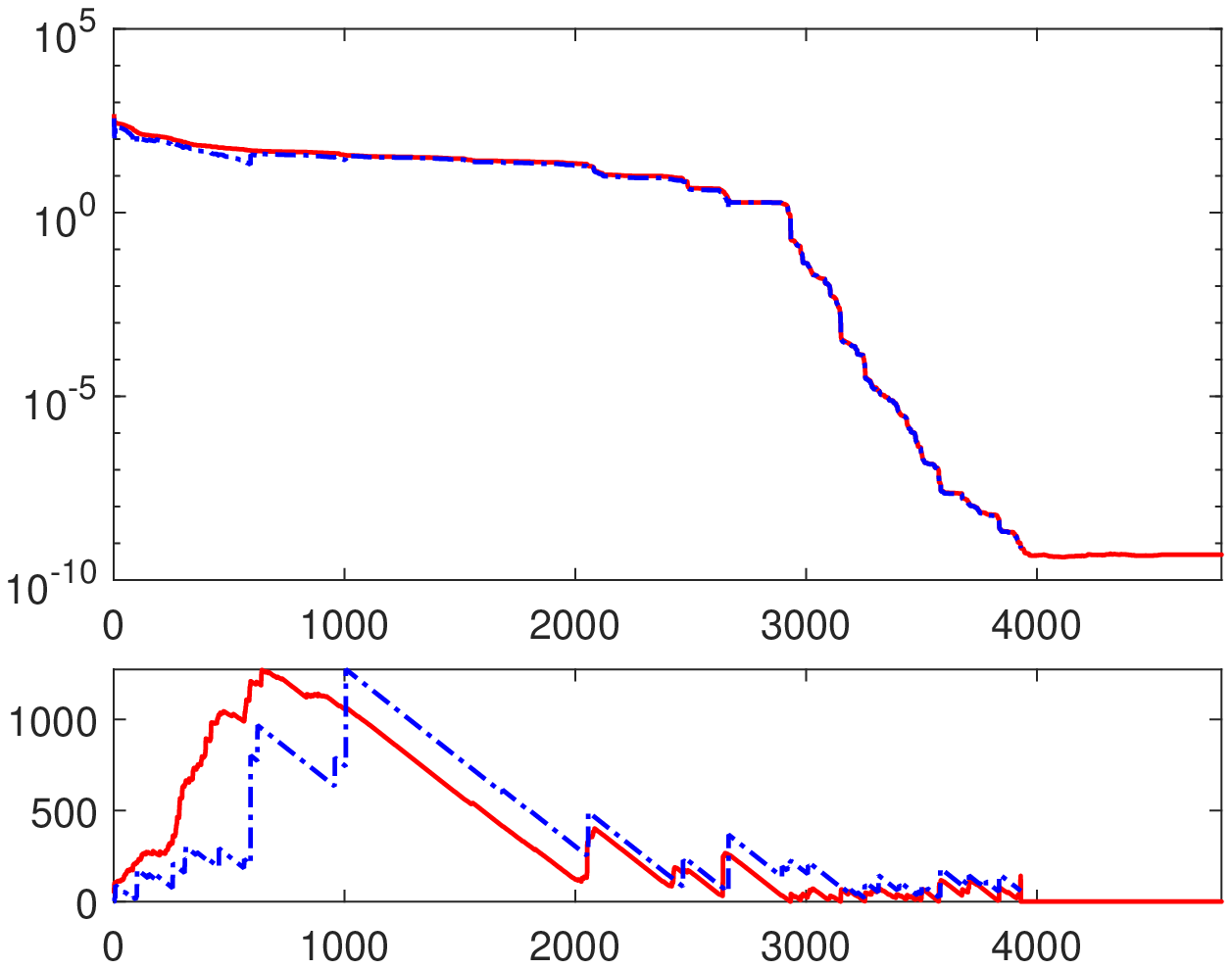}\\
    \textsf{iterations}
    \end{minipage}
    \caption{Matrix \texttt{\detokenize{illc1033}}, CGLS (left) and LSQR (right): error~$\| x-x_k \|_{A^TA}$ and adaptive error estimate (top), adaptively chosen delay $k-\ell$ and its ideal value (bottom)}
    \label{fig:illc1033}
\end{figure}

\begin{figure}[htp]
    \centering
    \begin{minipage}{0.45\textwidth}
    \centering
    matrix \texttt{\detokenize{illc1850}}, CGLS\\
    \includegraphics[width=\textwidth]{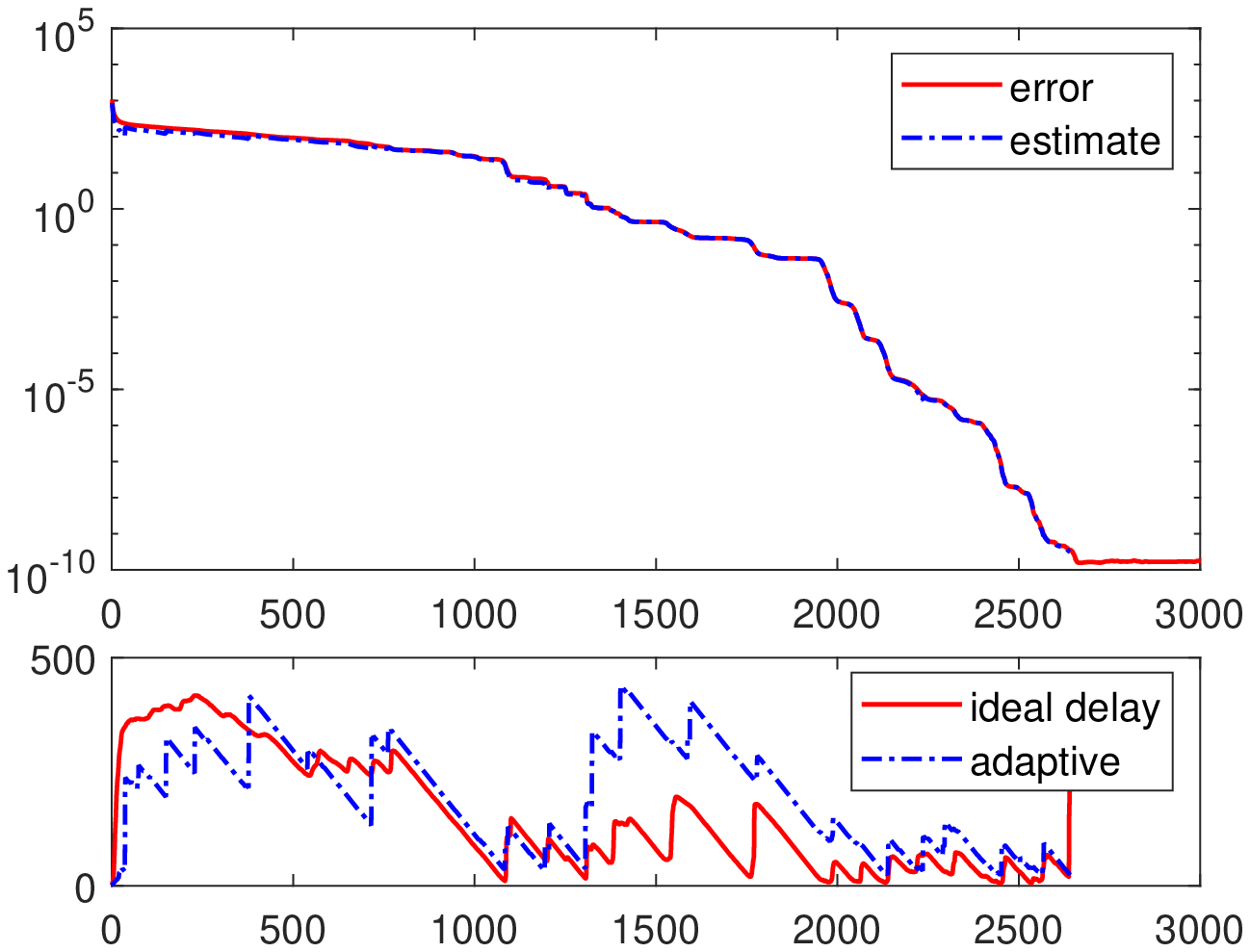}\\
    \textsf{iterations}
    \end{minipage}
    \hfill
    \begin{minipage}{0.45\textwidth}
    \centering
    matrix \texttt{\detokenize{illc1850}}, LSQR\\
    \includegraphics[width=\textwidth]{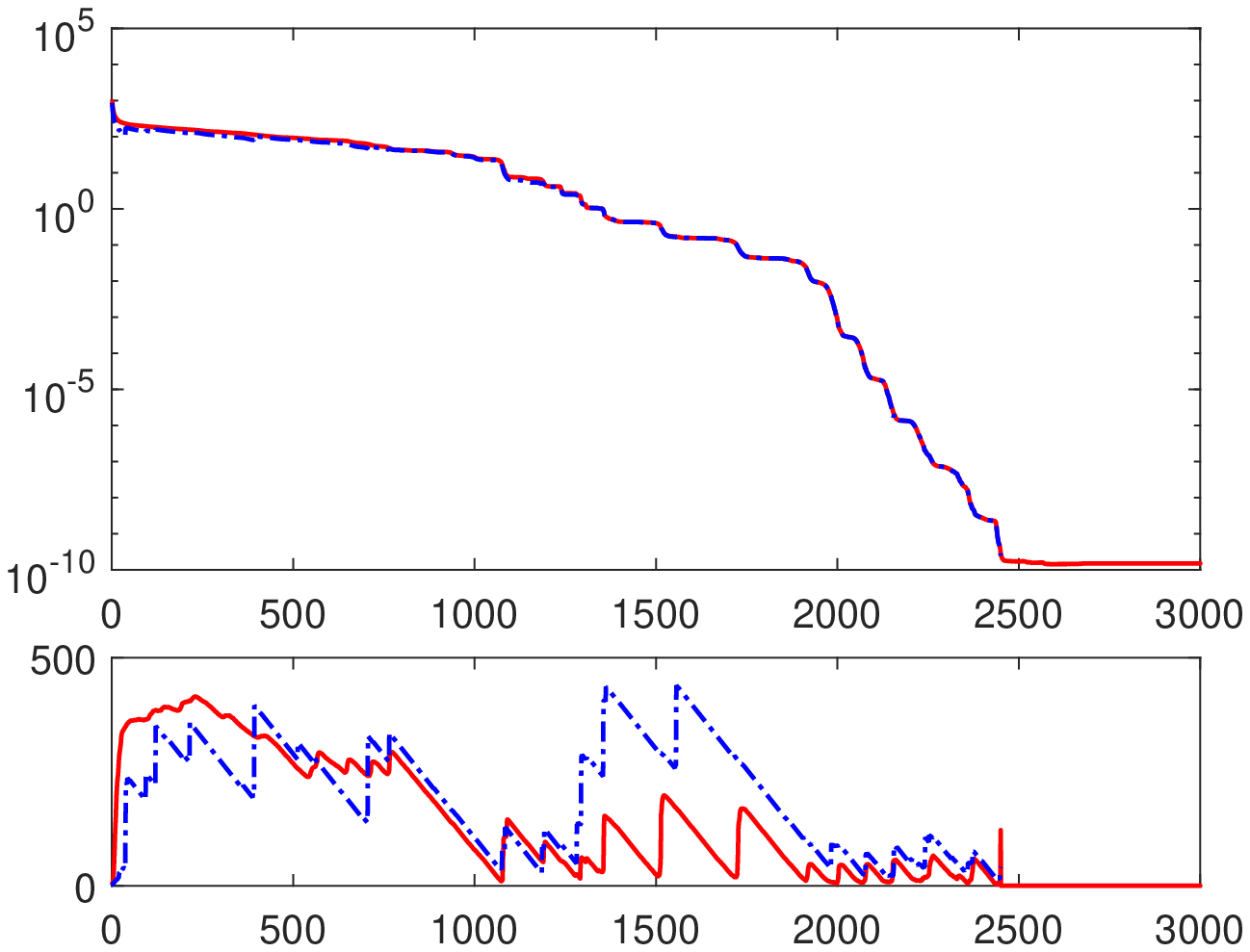}\\
    \textsf{iterations}
    \end{minipage}
    \caption{Matrix \texttt{\detokenize{illc1850}}, CGLS (left) and LSQR (right): error~$\| x-x_k \|_{A^TA}$ and adaptive error estimate (top), adaptively chosen delay $k-\ell$ and its ideal value (bottom)}
    \label{fig:illc1850}
\end{figure}

\begin{figure}[htp]
    \centering
    \begin{minipage}{0.45\textwidth}
    \centering
    matrix \texttt{\detokenize{well1033}}, CGLS\\
    \includegraphics[width=\textwidth]{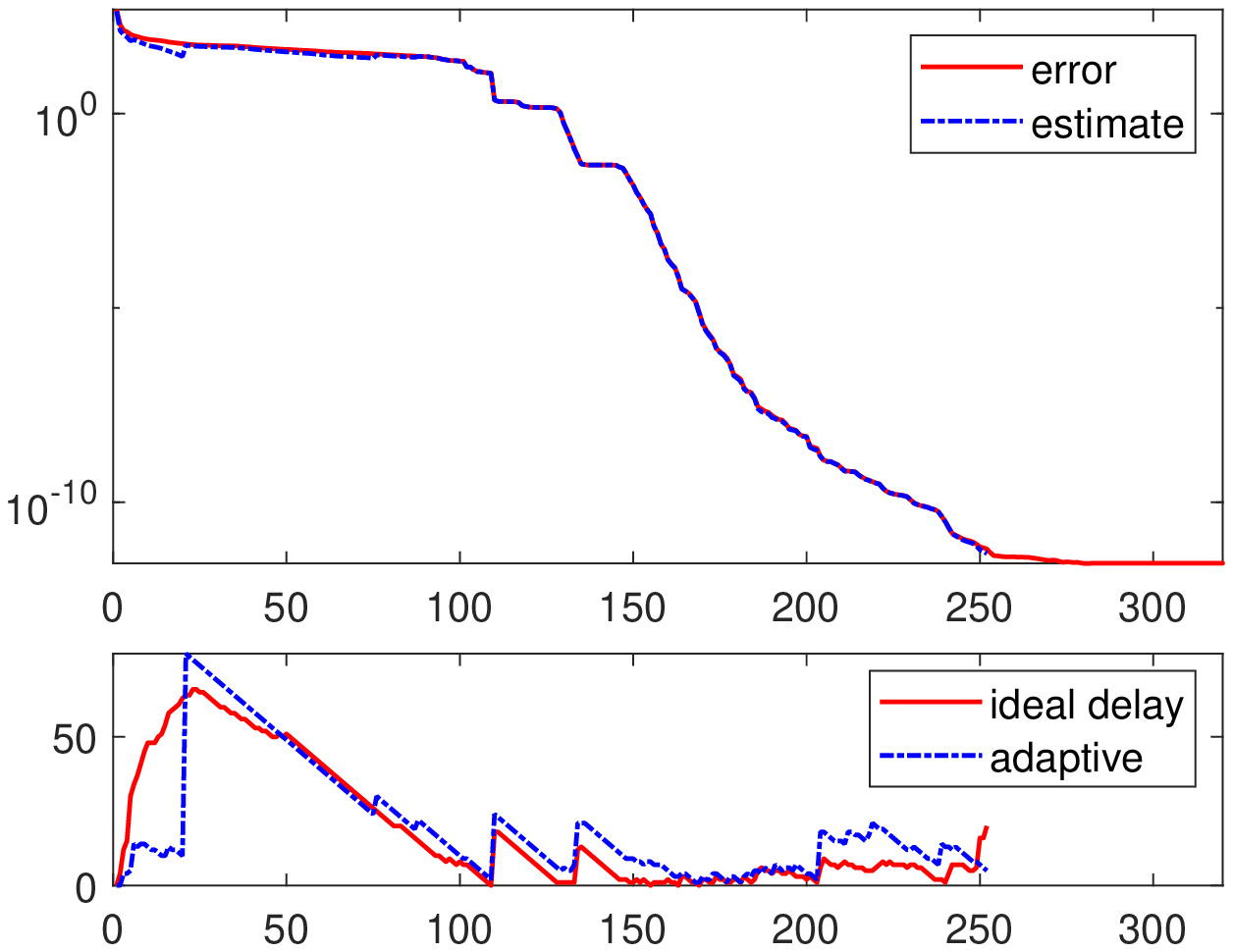}\\
    \textsf{iterations}
    \end{minipage}
    \hfill
    \begin{minipage}{0.45\textwidth}
    \centering
    matrix \texttt{\detokenize{well1033}}, LSQR\\
    \includegraphics[width=\textwidth]{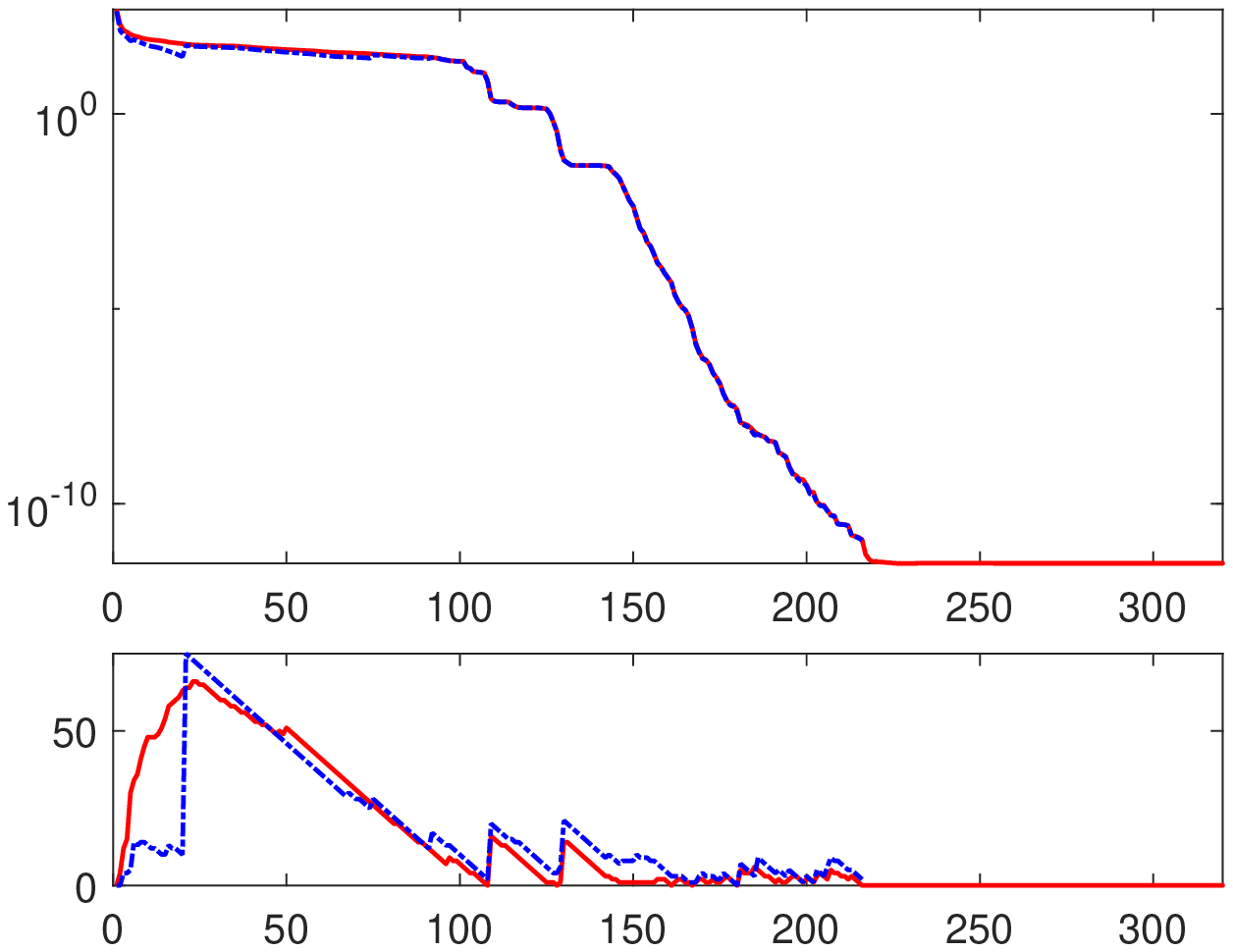}\\
    \textsf{iterations}
    \end{minipage}
    \caption{Matrix \texttt{\detokenize{well1033}}, CGLS (left) and LSQR (right): error~$\| x-x_k \|_{A^TA}$ and adaptive error estimate (top), adaptively chosen delay $k-\ell$ and its ideal value (bottom)}
    \label{fig:well1033}
\end{figure}

\begin{figure}[htp]
    \centering
    \begin{minipage}{0.45\textwidth}
    \centering
    matrix \texttt{\detokenize{well1850}}, CGLS\\
    \includegraphics[width=\textwidth]{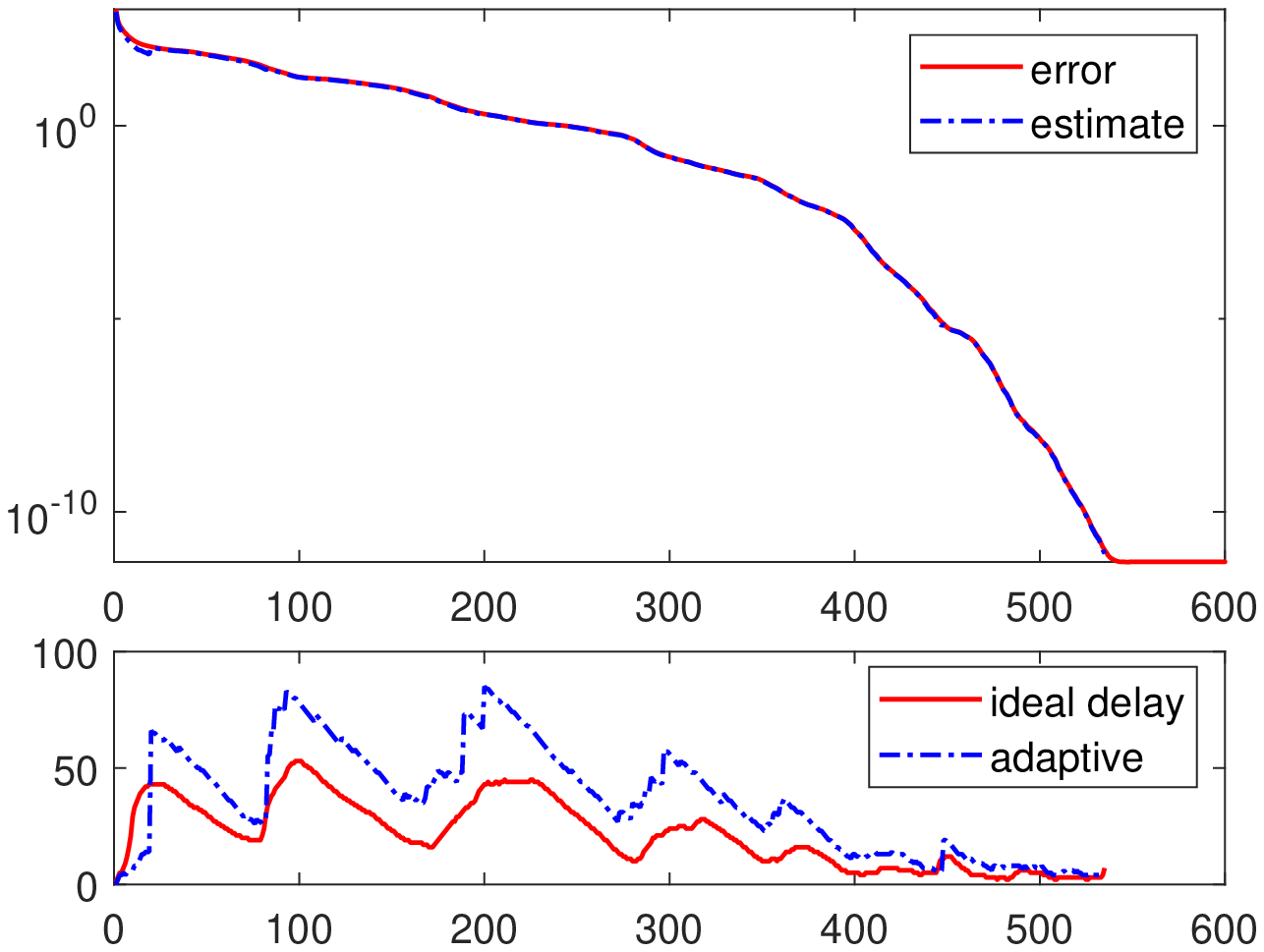}\\
    \textsf{iterations}
    \end{minipage}
    \hfill
    \begin{minipage}{0.45\textwidth}
    \centering
    matrix \texttt{\detokenize{well1850}}, LSQR\\
    \includegraphics[width=\textwidth]{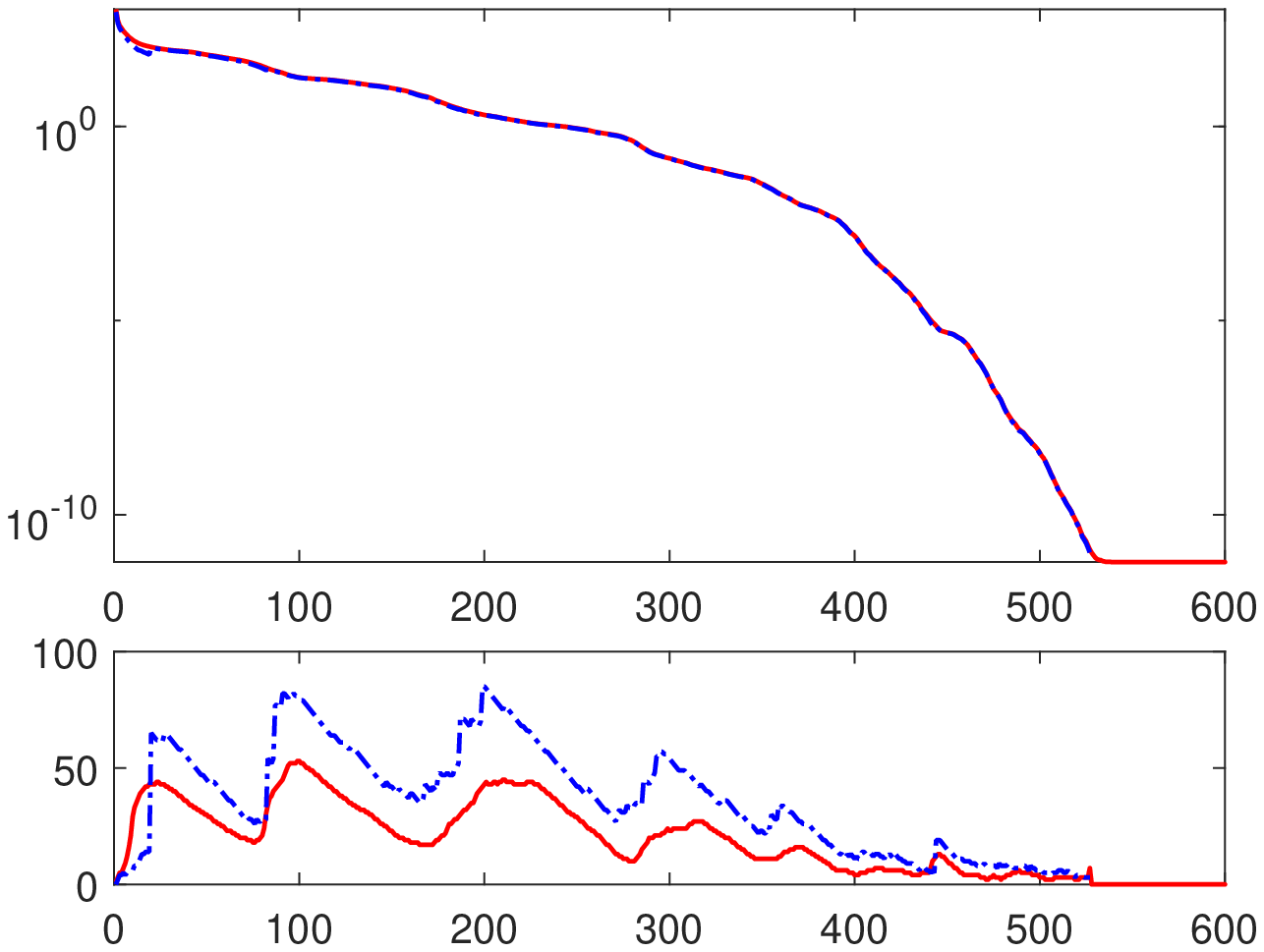}\\
    \textsf{iterations}
    \end{minipage}
    \caption{Matrix \texttt{\detokenize{well1850}}, CGLS (left) and LSQR (right): error~$\| x-x_k \|_{A^TA}$ and adaptive error estimate (top), adaptively chosen delay $k-\ell$ and its ideal value (bottom)}
    \label{fig:well1850}
\end{figure}

\begin{figure}[htp]
    \centering
    \begin{minipage}{0.45\textwidth}
    \centering
    matrix \texttt{\detokenize{sls}}, CGLS\\
    \includegraphics[width=\textwidth]{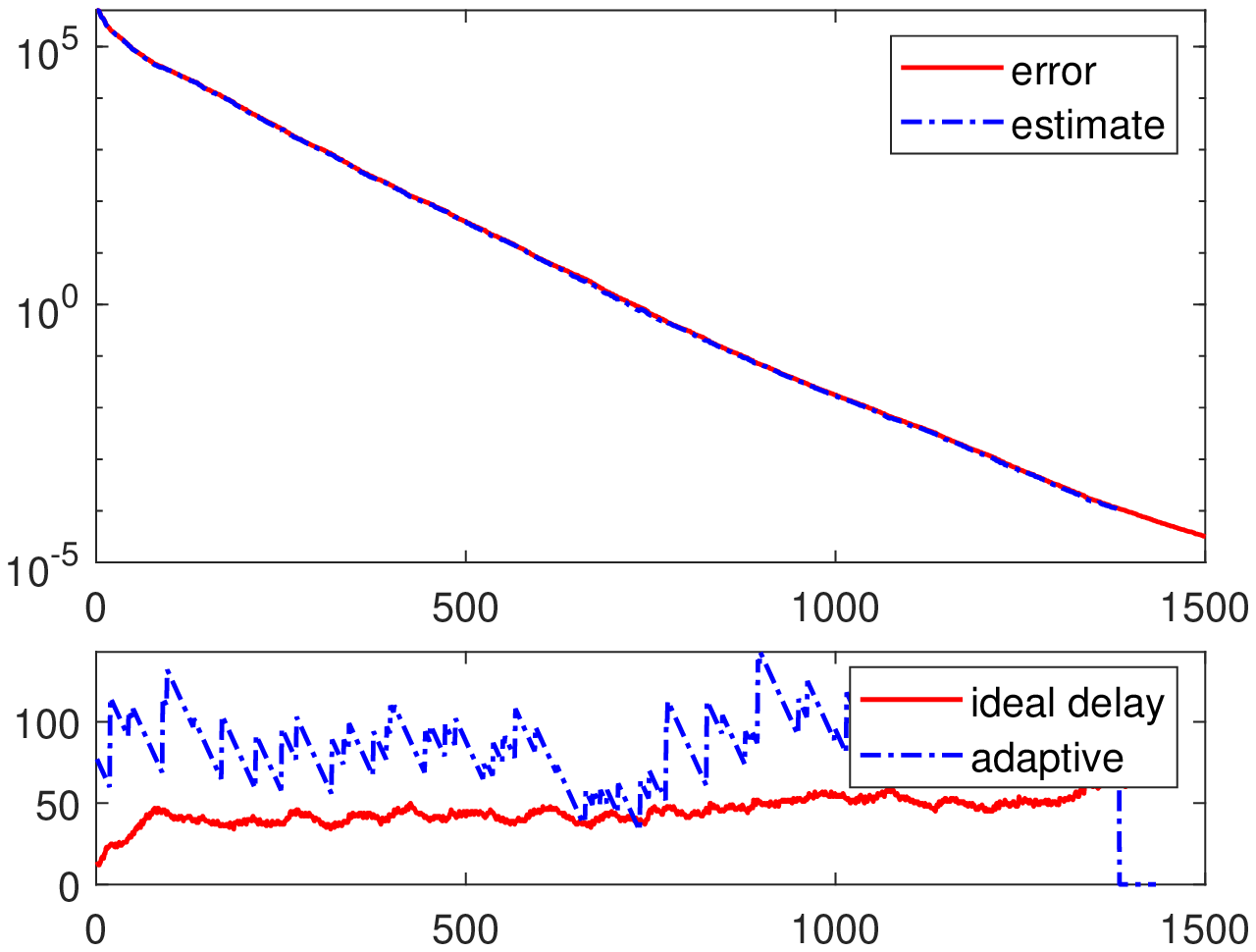}\\
    \textsf{iterations}
    \end{minipage}
    \hfill
    \begin{minipage}{0.45\textwidth}
    \centering
    matrix \texttt{\detokenize{sls}}, LSQR\\
    \includegraphics[width=\textwidth]{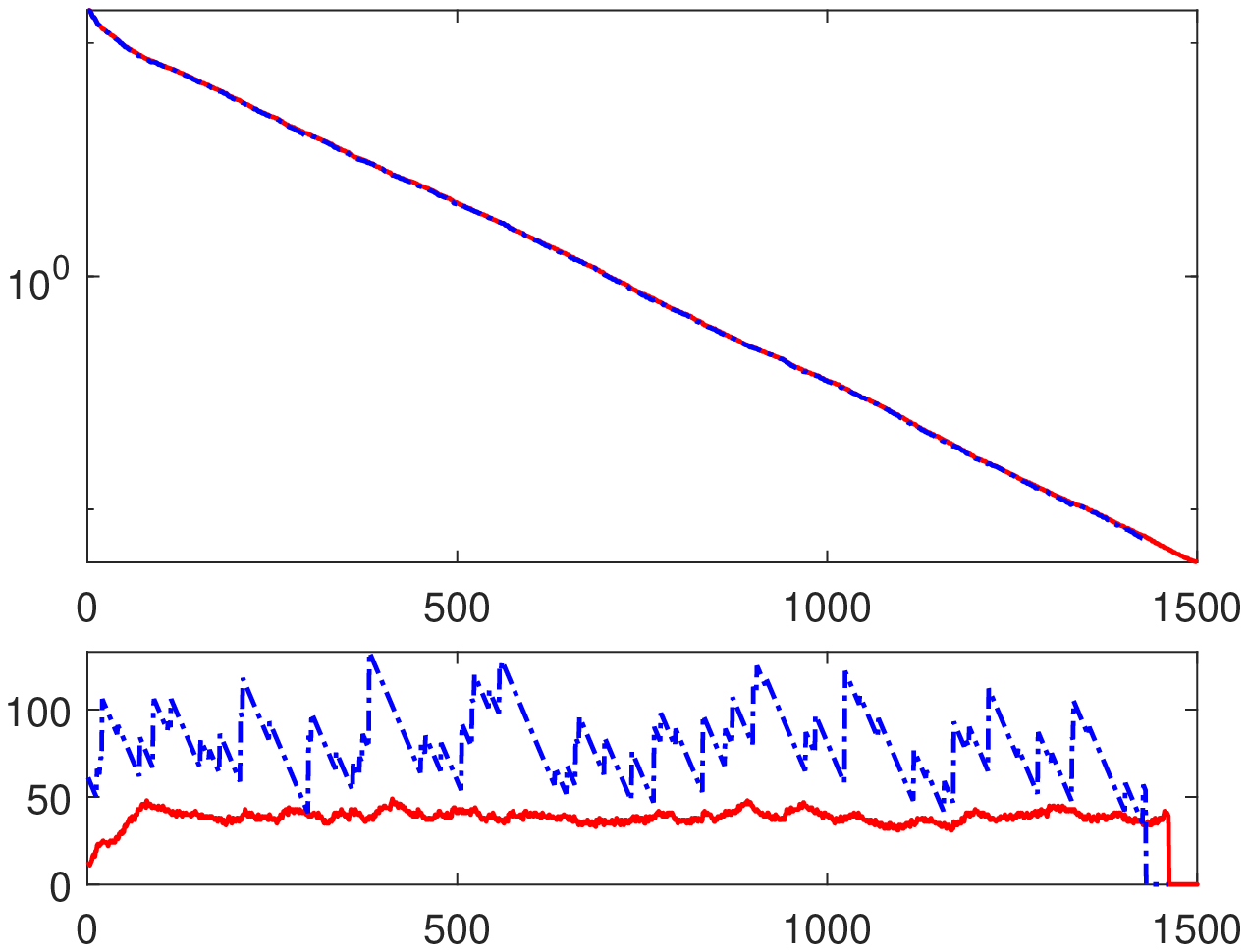}\\
    \textsf{iterations}
    \end{minipage}
    \caption{Matrix \texttt{\detokenize{sls}}, CGLS (left) and LSQR (right): error~$\| x-x_k \|_{A^TA}$ and adaptive error estimate (top), adaptively chosen delay $k-\ell$ and its ideal value (bottom)}
    \label{fig:sls}
\end{figure}

\FloatBarrier

\subsection{Least-norm problems}

The results for least-norm problems solved by CGNE and CRAIG are given in Figures \ref{fig:Delor64K}--\ref{fig:lp_pilot}. Here we observe very satisfactory behaviour with slightly higher delay for \texttt{Delor338K} and \texttt{lp\_pilot}; nevertheless, the adaptively chosen delay nicely follows the increases and decreases of the ideal value.

\begin{figure}[htp]
    \centering
    \begin{minipage}{0.45\textwidth}
    \centering
    matrix \texttt{\detokenize{Delor64K}}, CGNE\\
    \includegraphics[width=\textwidth]{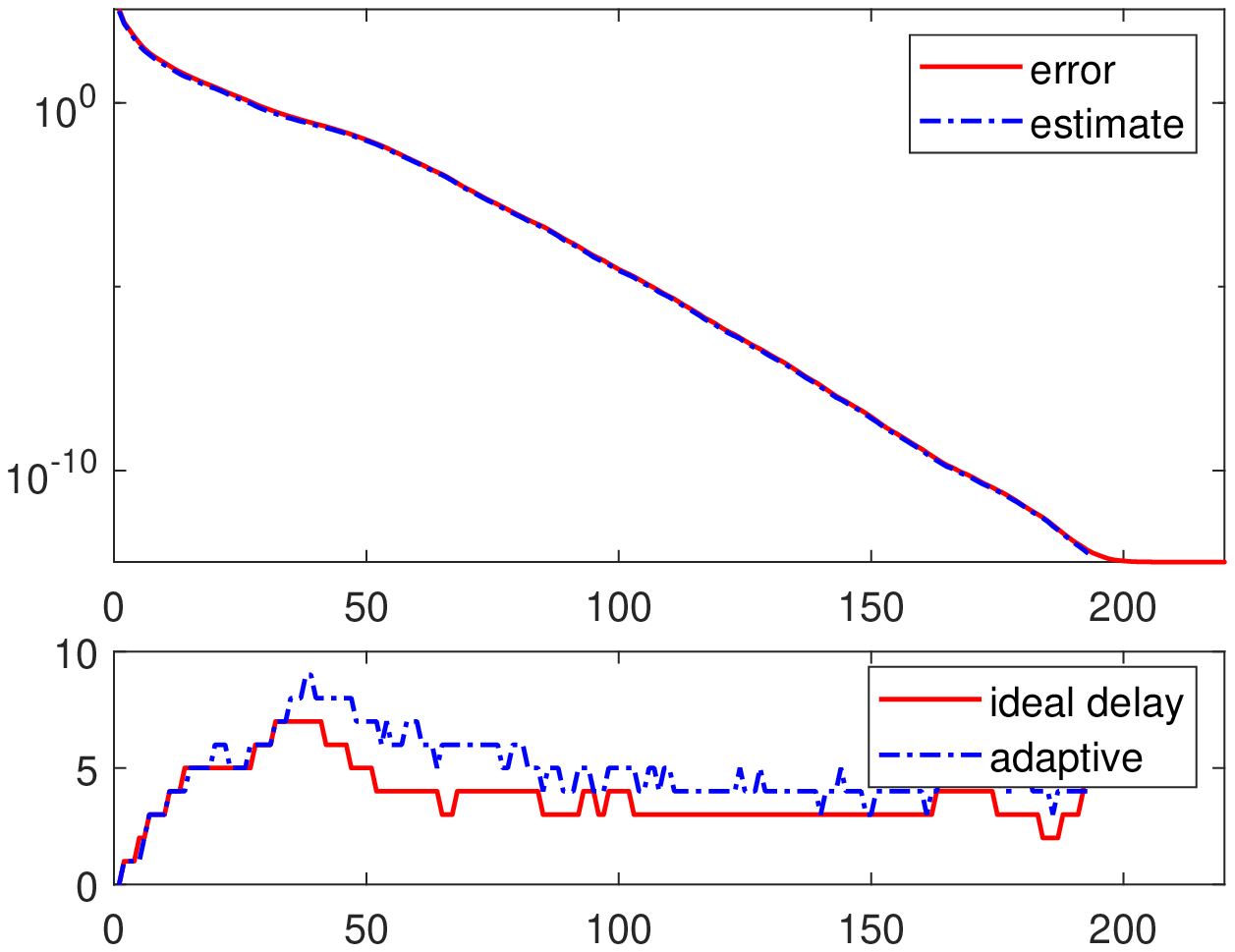}\\
    \textsf{iterations}
    \end{minipage}
    \hfill
    \begin{minipage}{0.45\textwidth}
    \centering
    matrix \texttt{\detokenize{Delor64K}}, CRAIG\\
    \includegraphics[width=\textwidth]{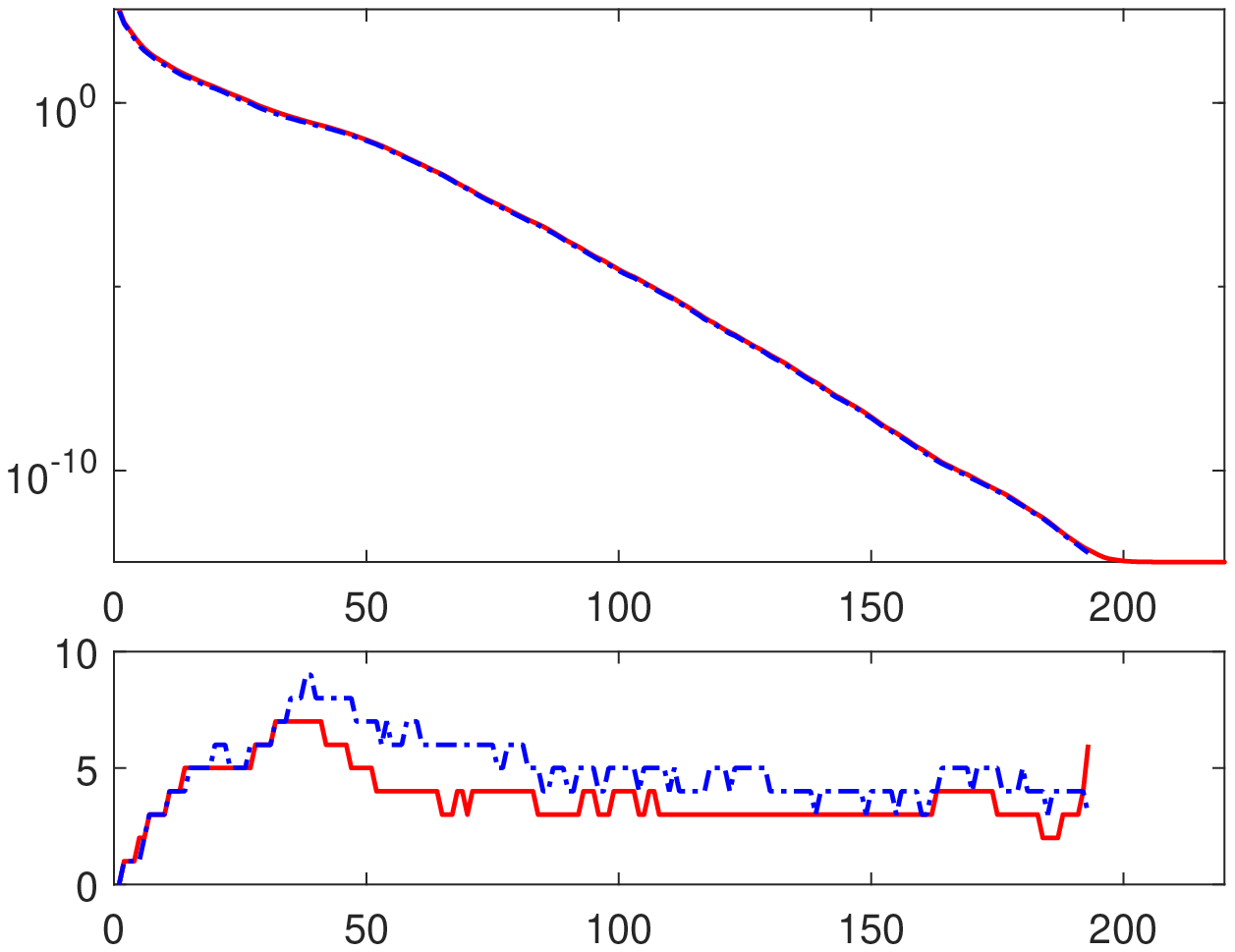}\\
    \textsf{iterations}
    \end{minipage}
    \caption{Matrix \texttt{\detokenize{Delor64K}}, CGNE (left) and CRAIG (right): error~$\| x-x_k \|$ and adaptive error estimate (top), adaptively chosen delay $k-\ell$ and its ideal value (bottom)}
    \label{fig:Delor64K}
\end{figure}

\begin{figure}[htp]
    \centering
    \begin{minipage}{0.45\textwidth}
    \centering
    matrix \texttt{\detokenize{Delor338K}}, CGNE\\
    \includegraphics[width=\textwidth]{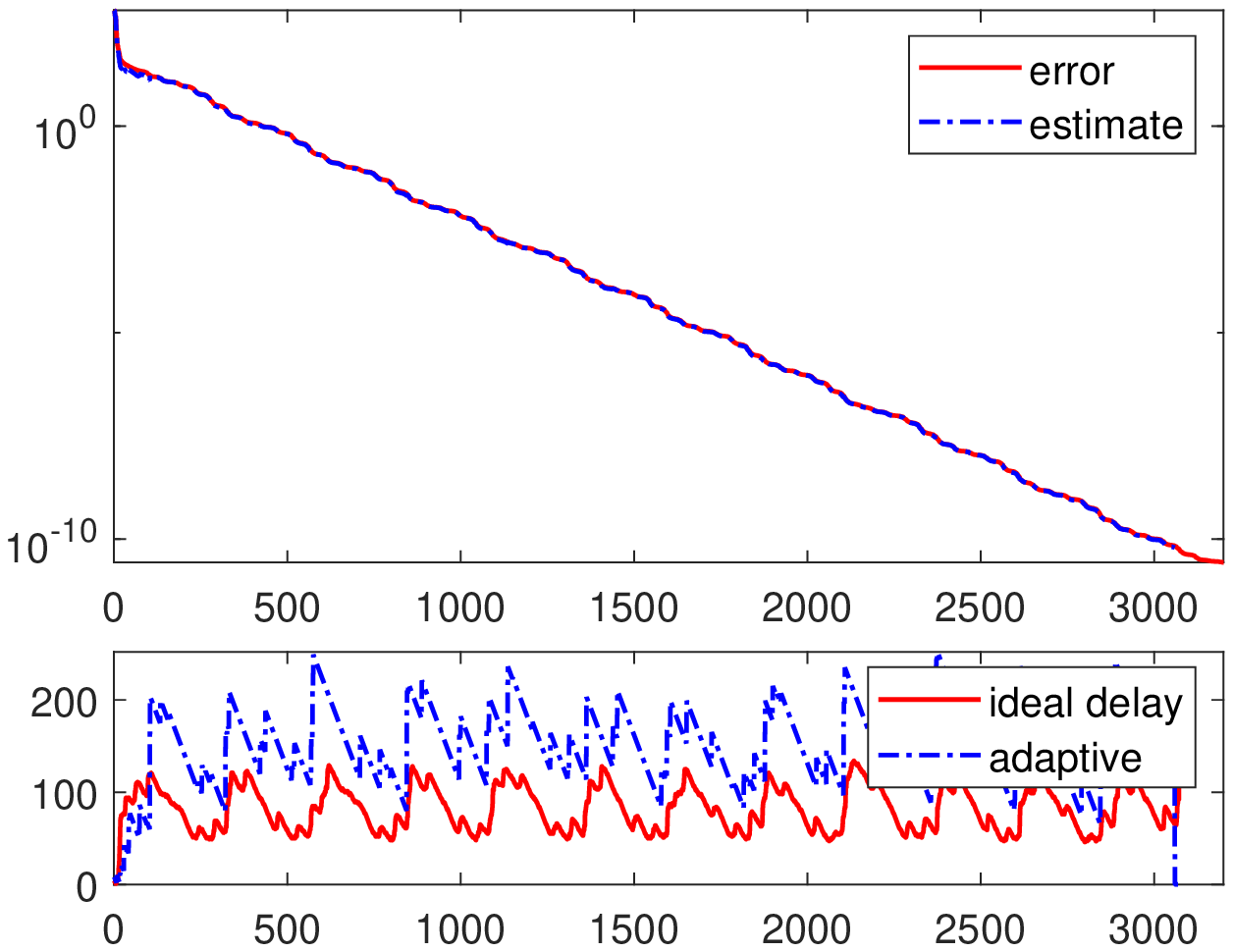}\\
    \textsf{iterations}
    \end{minipage}
    \hfill
    \begin{minipage}{0.45\textwidth}
    \centering
    matrix \texttt{\detokenize{Delor338K}}, CRAIG\\
    \includegraphics[width=\textwidth]{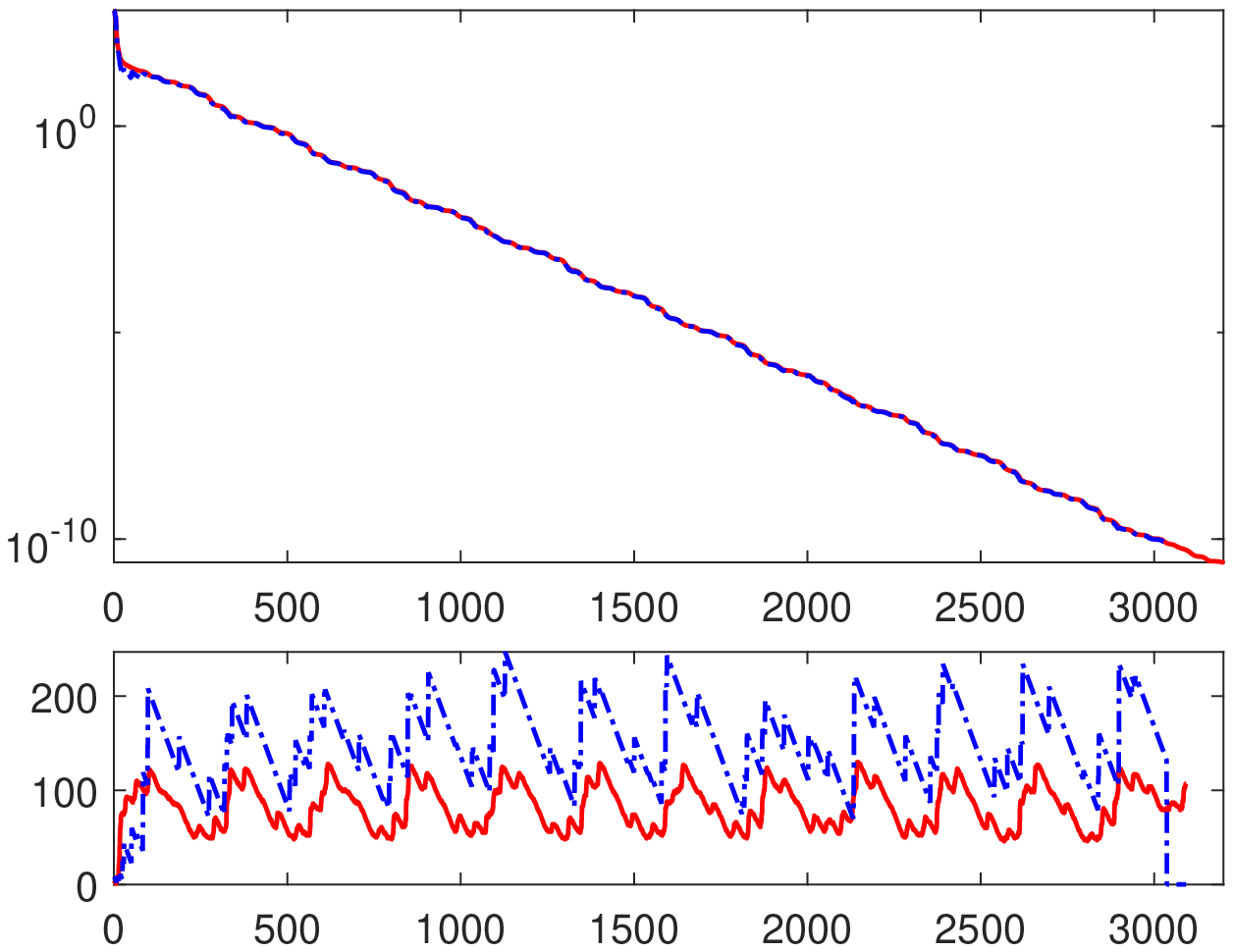}\\
    \textsf{iterations}
    \end{minipage}
    \caption{Matrix \texttt{\detokenize{Delor338K}}, CGNE (left) and CRAIG (right): error~$\| x-x_k \|$ and adaptive error estimate (top), adaptively chosen delay $k-\ell$ and its ideal value (bottom)}
    \label{fig:Delor338K}
\end{figure}

\begin{figure}[htp]
    \centering
    \begin{minipage}{0.45\textwidth}
    \centering
    matrix \texttt{\detokenize{flower_7_4}}, CGNE\\
    \includegraphics[width=\textwidth]{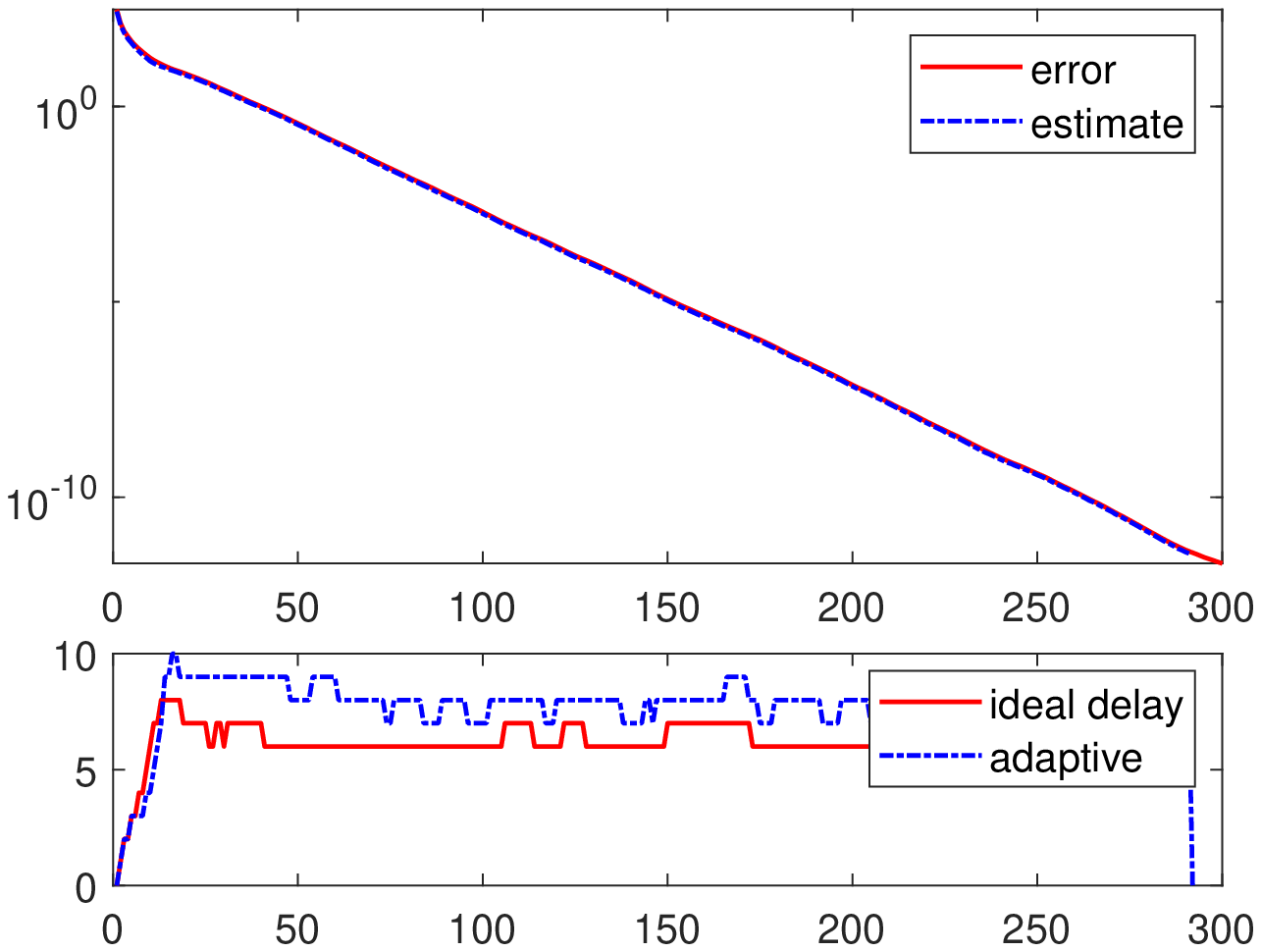}\\
    \textsf{iterations}
    \end{minipage}
    \hfill
    \begin{minipage}{0.45\textwidth}
    \centering
    matrix \texttt{\detokenize{flower_7_4}}, CRAIG\\
    \includegraphics[width=\textwidth]{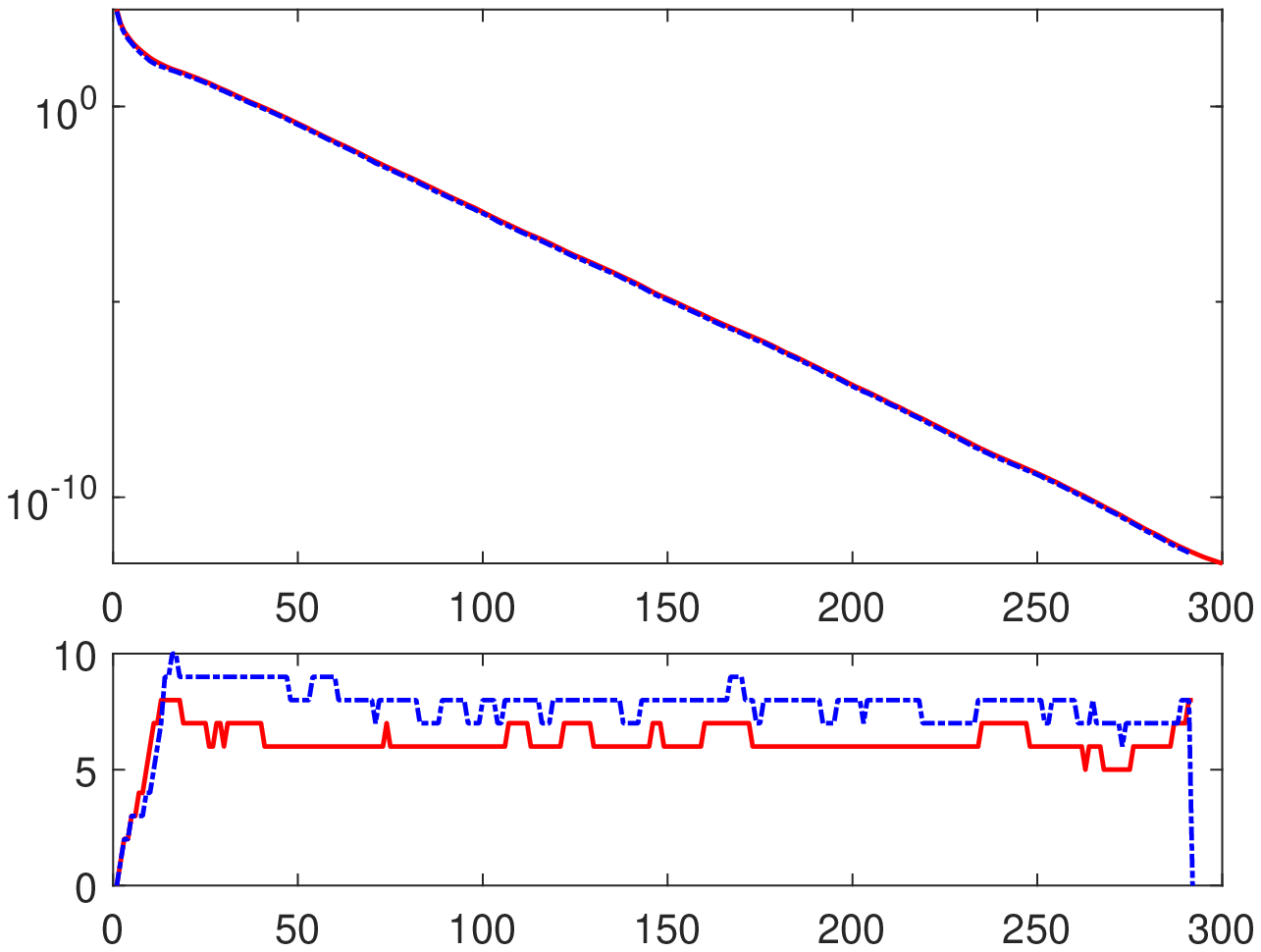}\\
    \textsf{iterations}
    \end{minipage}
    \caption{Matrix \texttt{\detokenize{flower_7_4}}, CGNE (left) and CRAIG (right): error~$\| x-x_k \|$ and adaptive error estimate (top), adaptively chosen delay $k-\ell$ and its ideal value (bottom)}
    \label{fig:flower_7_4}
\end{figure}

\begin{figure}[htp]
    \centering
    \begin{minipage}{0.45\textwidth}
    \centering
    matrix \texttt{\detokenize{cat_ears_4_4}}, CGNE\\
    \includegraphics[width=\textwidth]{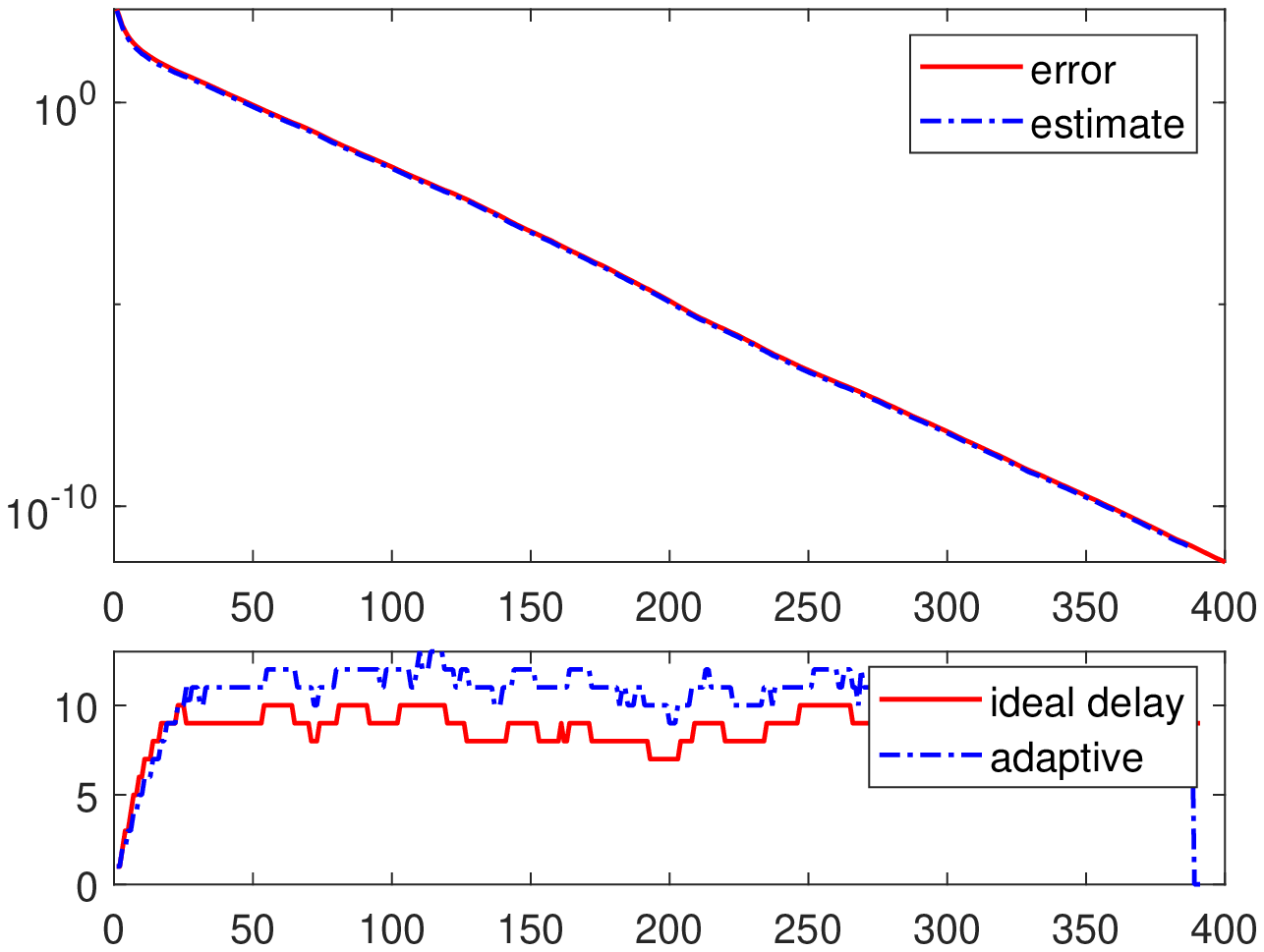}\\
    \textsf{iterations}
    \end{minipage}
    \hfill
    \begin{minipage}{0.45\textwidth}
    \centering
    matrix \texttt{\detokenize{cat_ears_4_4}}, CRAIG\\
    \includegraphics[width=\textwidth]{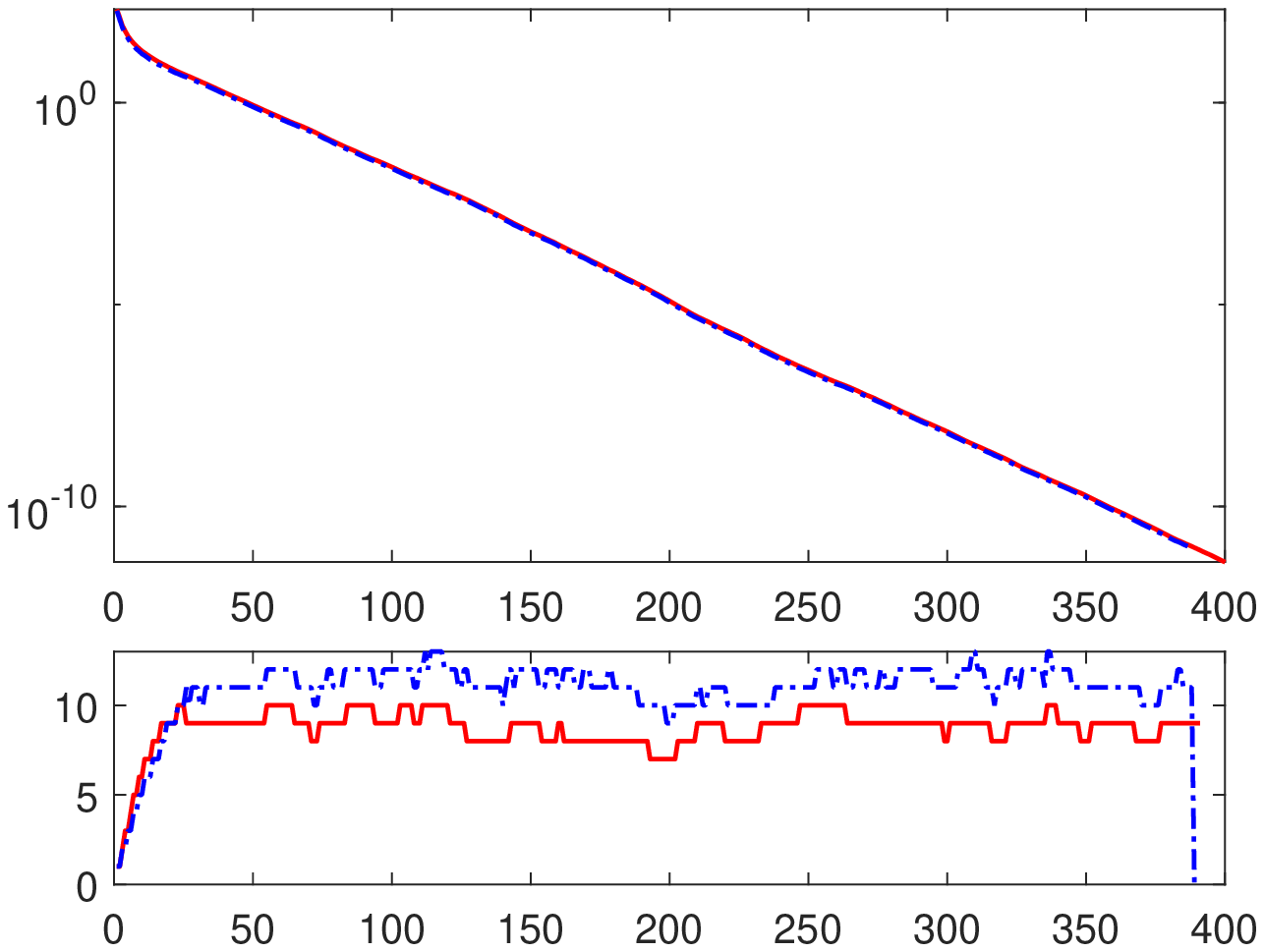}\\
    \textsf{iterations}
    \end{minipage}
    \caption{Matrix \texttt{\detokenize{cat_ears_4_4}}, CGNE (left) and CRAIG (right): error~$\| x-x_k \|$ and adaptive error estimate (top), adaptively chosen delay $k-\ell$ and its ideal value (bottom)}
    \label{fig:cat_ears_4_4}
\end{figure}

\begin{figure}[htp]
    \centering
    \begin{minipage}{0.45\textwidth}
    \centering
    matrix \texttt{\detokenize{lp_pilot}}, CGNE\\
    \includegraphics[width=\textwidth]{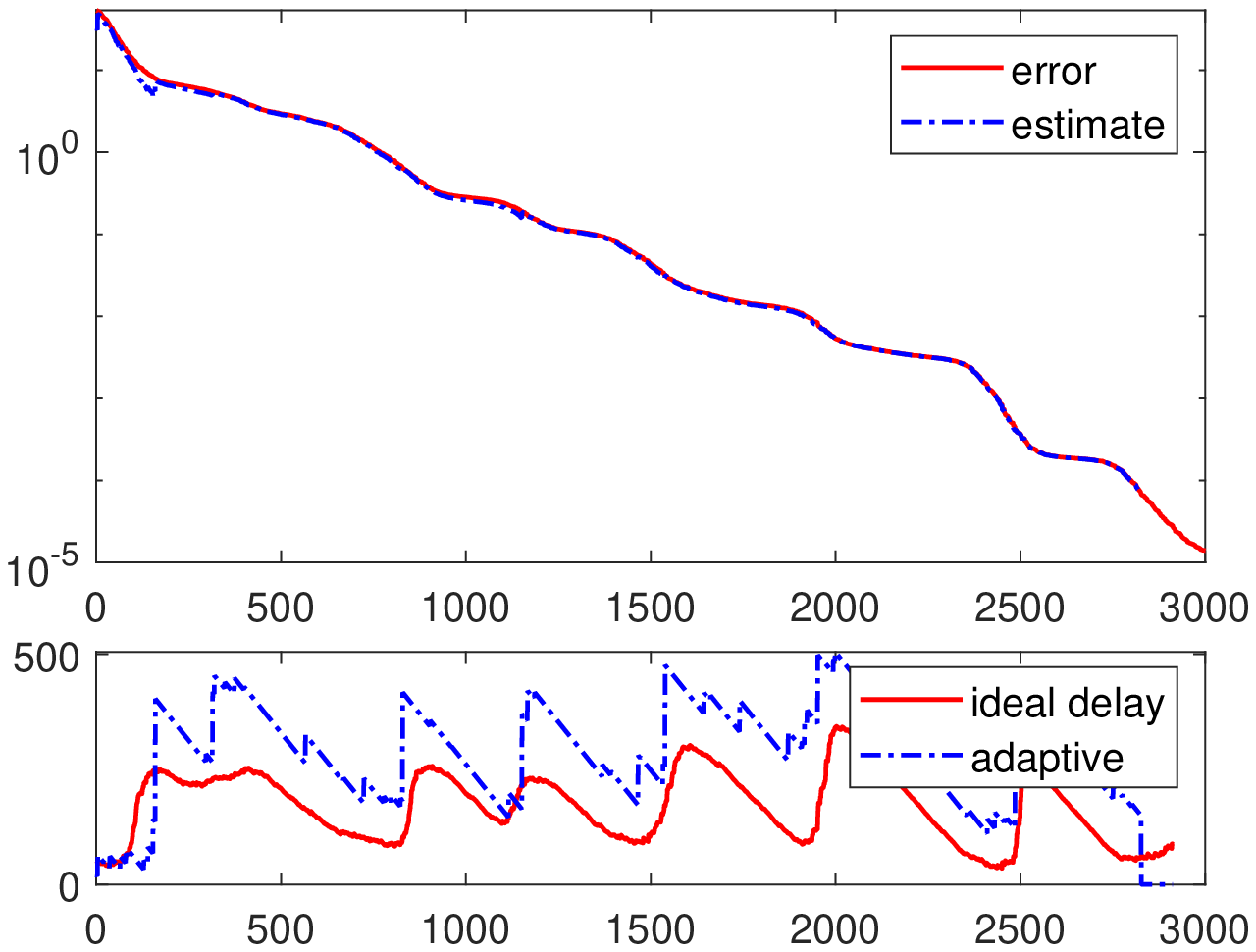}\\
    \textsf{iterations}
    \end{minipage}
    \hfill
    \begin{minipage}{0.45\textwidth}
    \centering
    matrix \texttt{\detokenize{lp_pilot}}, CRAIG\\
    \includegraphics[width=\textwidth]{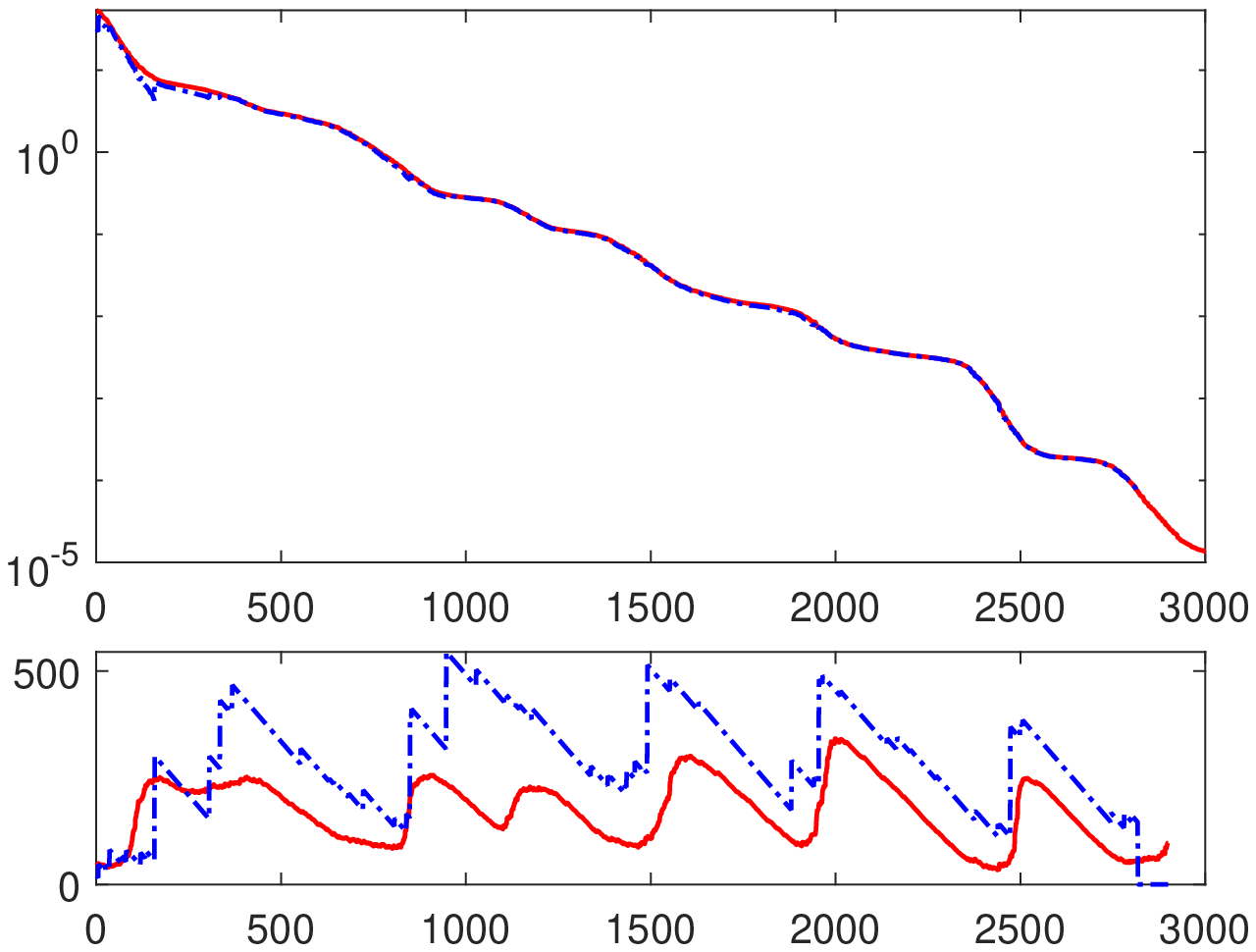}\\
    \textsf{iterations}
    \end{minipage}
    \caption{Matrix \texttt{\detokenize{lp_pilot}}, CGNE (left) and CRAIG (right): error~$\| x-x_k \|$ and adaptive error estimate (top), adaptively chosen delay $k-\ell$ and its ideal value (bottom)}
    \label{fig:lp_pilot}
\end{figure}

\FloatBarrier

\subsection{Preconditioned least-squares problems}

In Figures \ref{pfig:illc1033}--\ref{pfig:sls} we plot the results for our adaptive error estimate in preconditioned least-squares problems solved by CGLS and LSQR. The preconditioner~$L$ is constructed from~$A$ using MATLAB interface of {\tt HSL\_MI35} from \cite{HSL}, i.e., computing the incomplete Cholesky decomposition of $A^TA$, without explicitly forming it.
In the figures, we observe a very satisfactory behavior of the adaptive error estimate.

\begin{figure}[htp]
    \centering
    \begin{minipage}{0.45\textwidth}
    \centering
    matrix \texttt{\detokenize{illc1033}}, PCGLS\\
    \includegraphics[width=\textwidth]{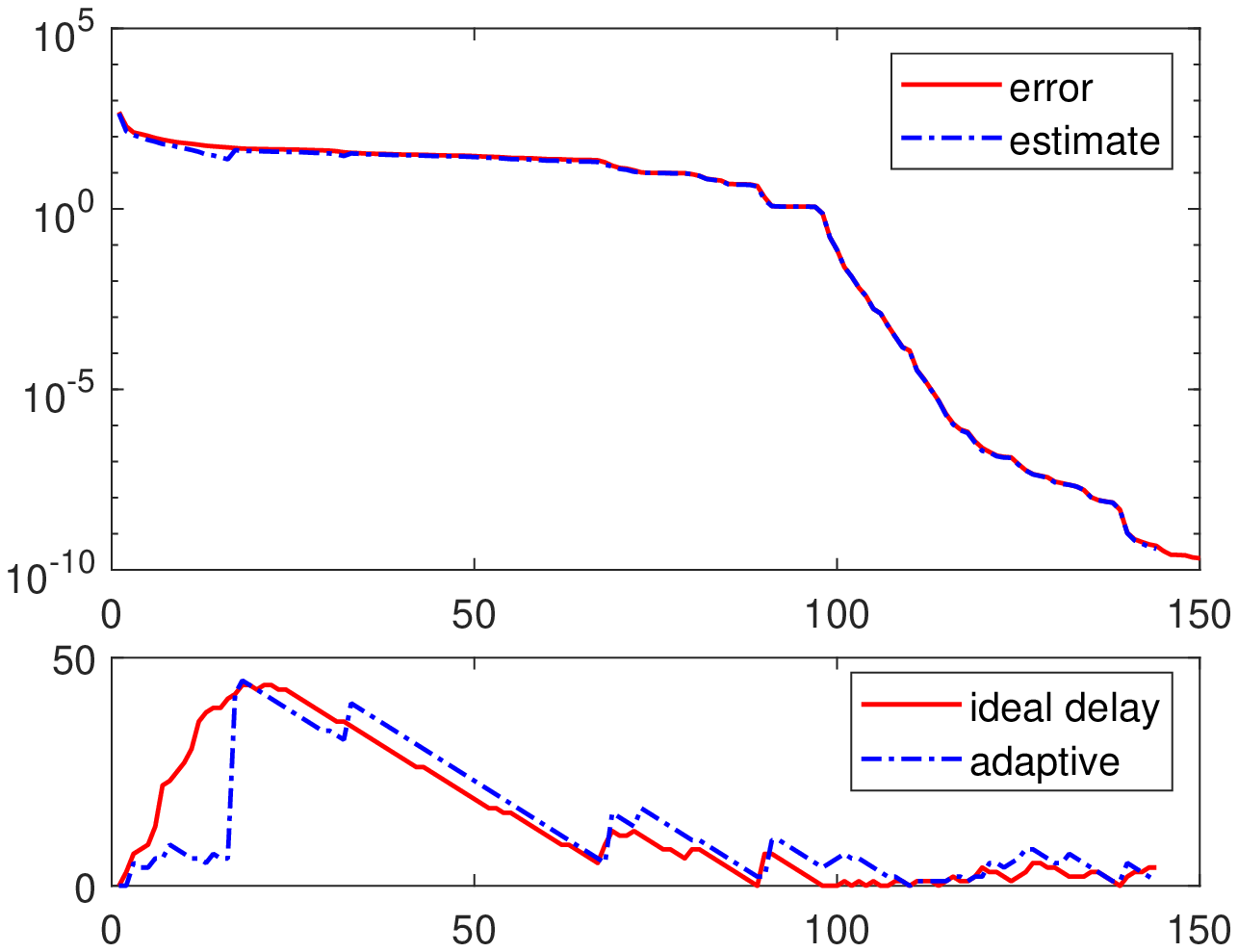}\\
    \textsf{iterations}
    \end{minipage}
    \hfill
    \begin{minipage}{0.45\textwidth}
    \centering
    matrix \texttt{\detokenize{illc1033}}, PLSQR\\
    \includegraphics[width=\textwidth]{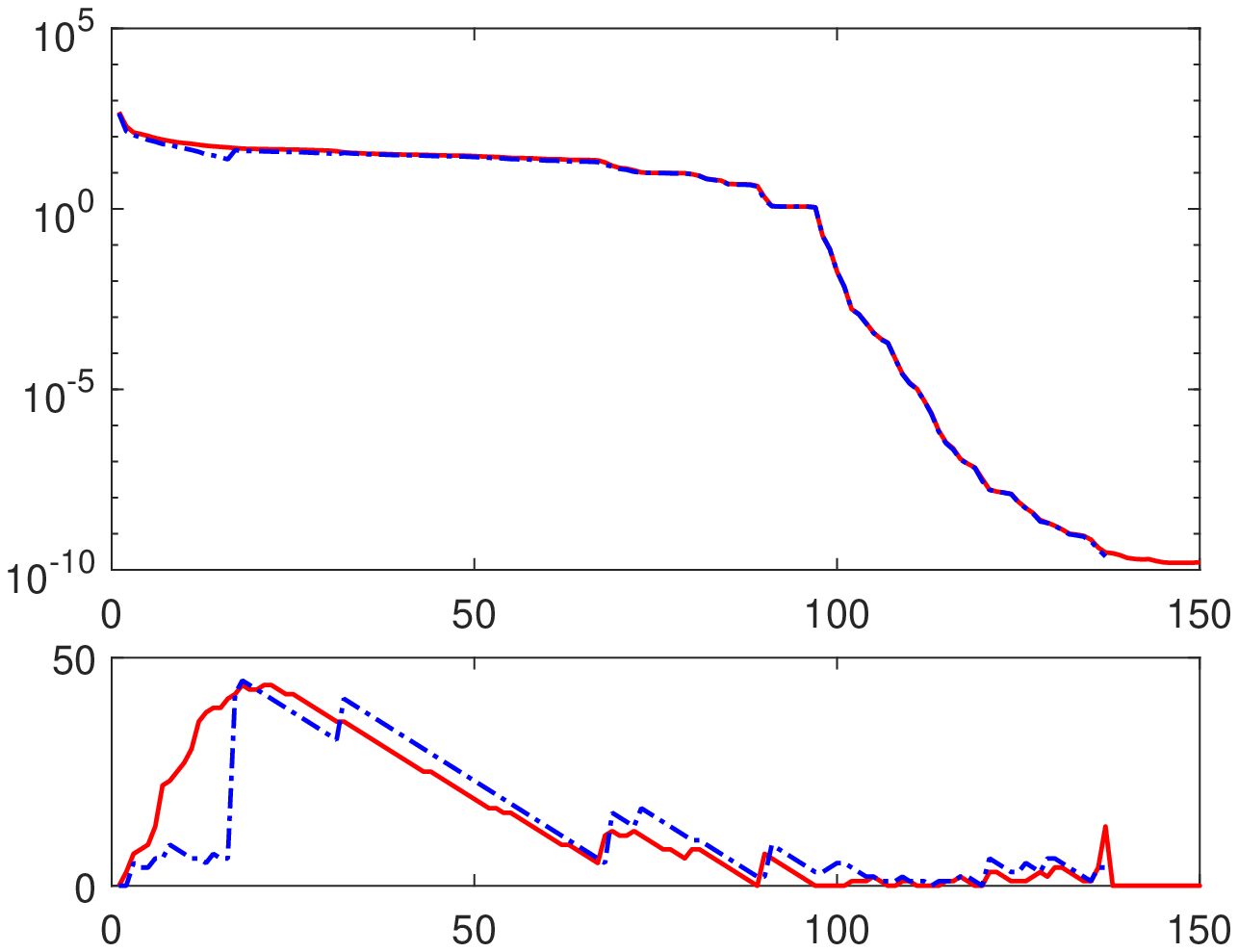}\\
    \textsf{iterations}
    \end{minipage}
    \caption{Matrix \texttt{\detokenize{illc1033}} with preconditioning, PCGLS (left) and PLSQR (right): error~$\| x-x_k \|_{A^TA}$ and adaptive error estimate (top), adaptively chosen delay $k-\ell$ and its ideal value (bottom)}
    \label{pfig:illc1033}
\end{figure}

\begin{figure}[htp]
    \centering
    \begin{minipage}{0.45\textwidth}
    \centering
    matrix \texttt{\detokenize{illc1850}}, PCGLS\\
    \includegraphics[width=\textwidth]{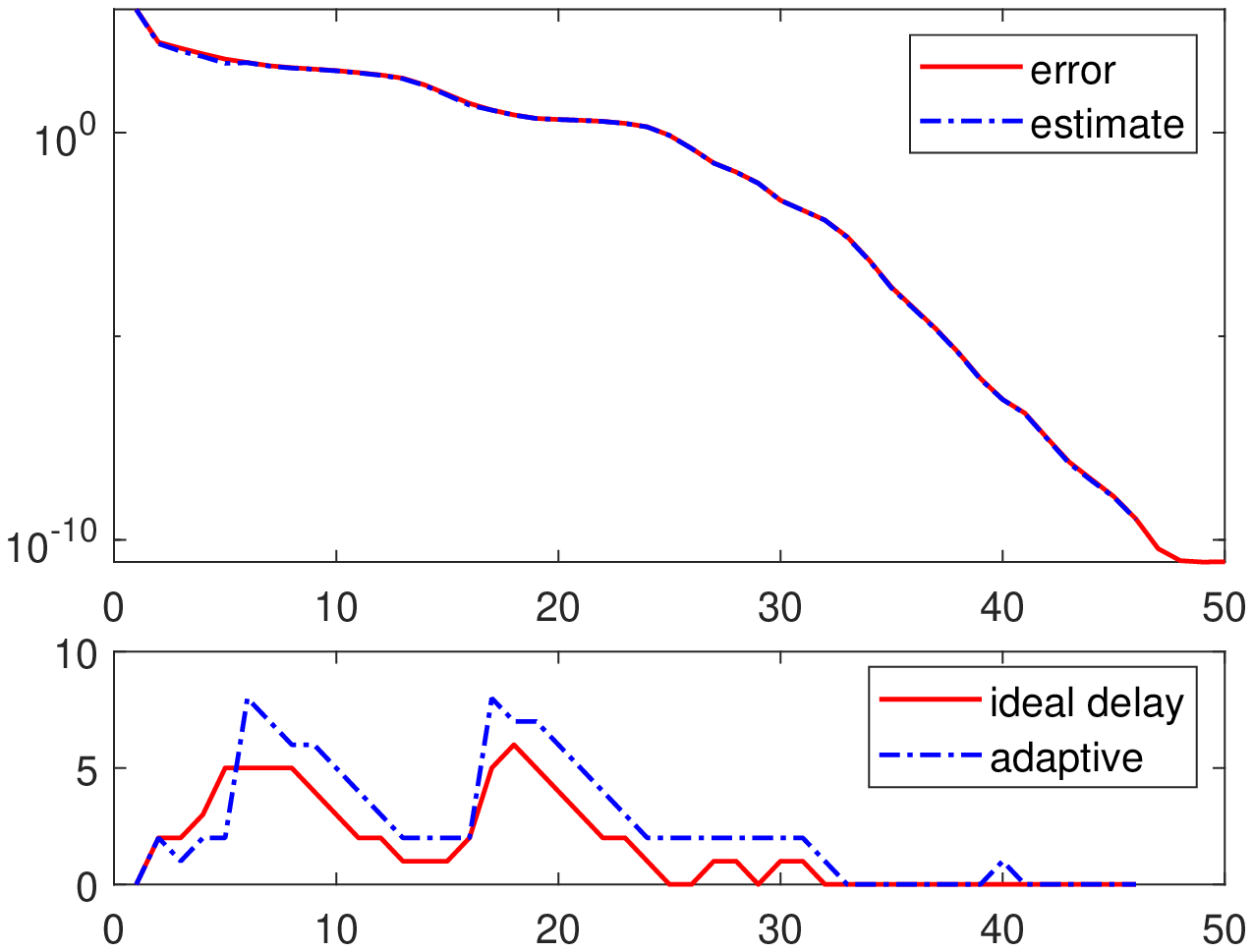}\\
    \textsf{iterations}
    \end{minipage}
    \hfill
    \begin{minipage}{0.45\textwidth}
    \centering
    matrix \texttt{\detokenize{illc1850}}, PLSQR\\
    \includegraphics[width=\textwidth]{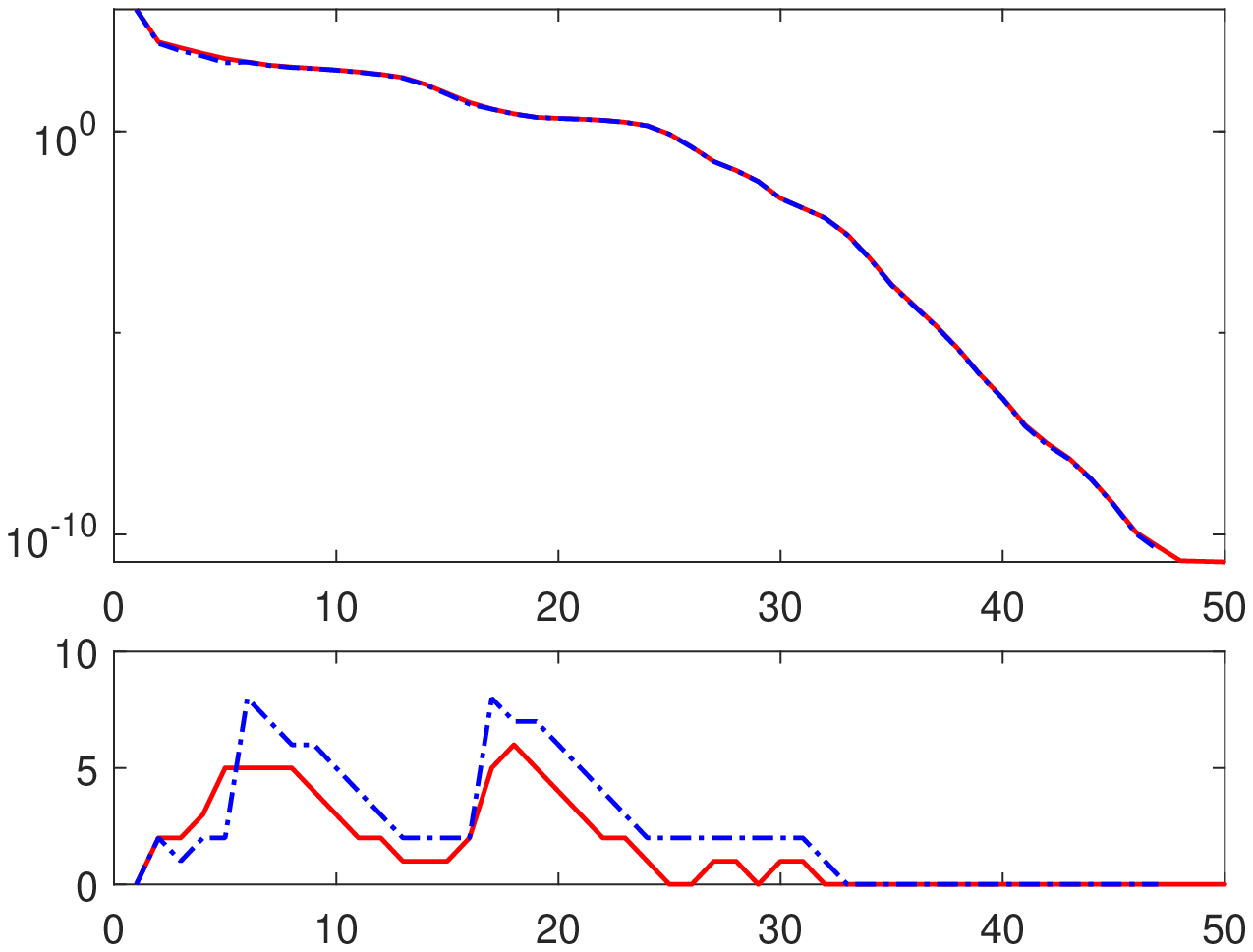}\\
    \textsf{iterations}
    \end{minipage}
    \caption{Matrix \texttt{\detokenize{illc1850}} with preconditioning, PCGLS (left) and PLSQR (right): error~$\| x-x_k \|_{A^TA}$ and adaptive error estimate (top), adaptively chosen delay $k-\ell$ and its ideal value (bottom)}
    \label{pfig:illc1850}
\end{figure}

\begin{figure}[htp]
    \centering
    \begin{minipage}{0.45\textwidth}
    \centering
    matrix \texttt{\detokenize{well1033}}, PCGLS\\
    \includegraphics[width=\textwidth]{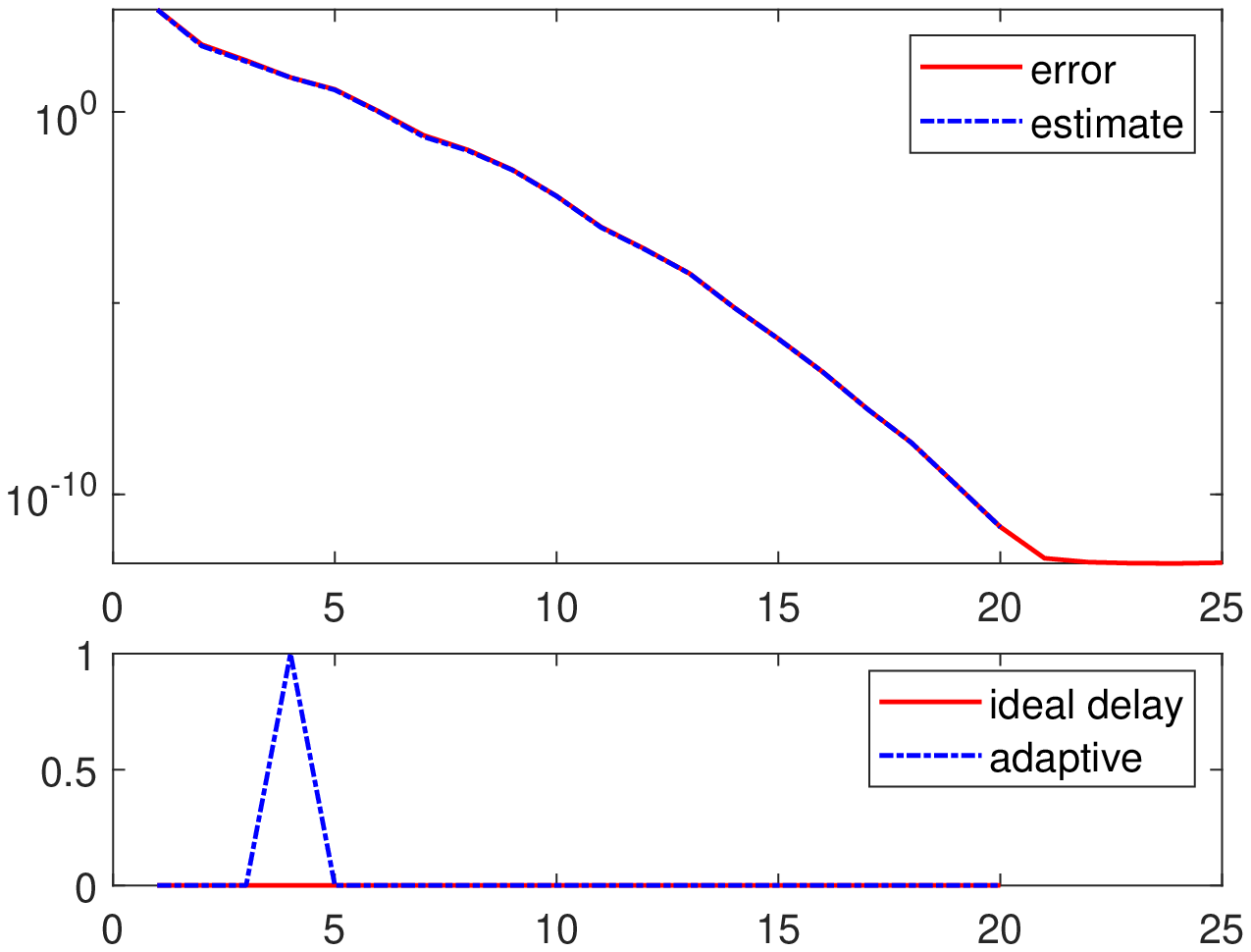}\\
    \textsf{iterations}
    \end{minipage}
    \hfill
    \begin{minipage}{0.45\textwidth}
    \centering
    matrix \texttt{\detokenize{well1033}}, PLSQR\\
    \includegraphics[width=\textwidth]{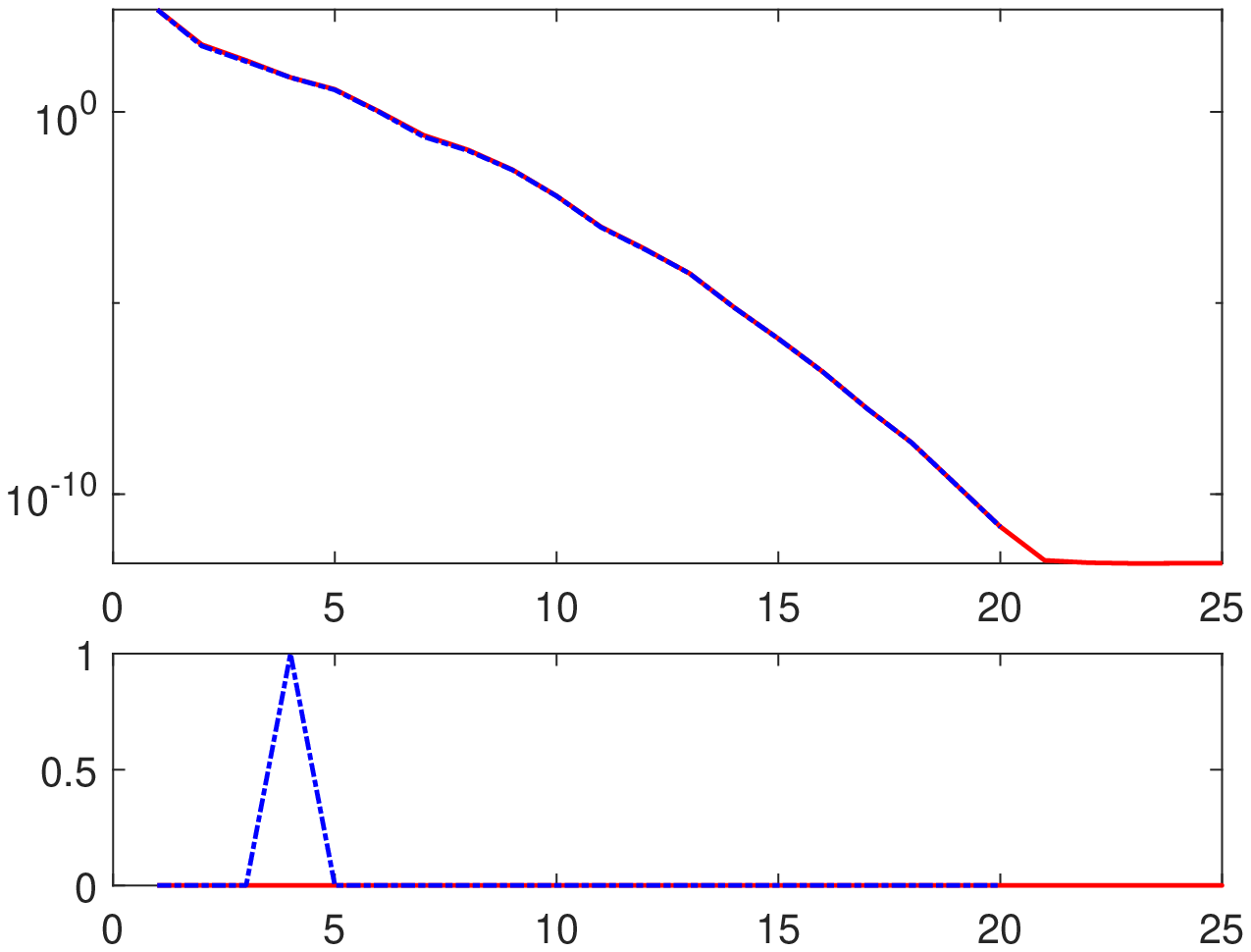}\\
    \textsf{iterations}
    \end{minipage}
    \caption{Matrix \texttt{\detokenize{well1033}} with preconditioning, PCGLS (left) and PLSQR (right): error~$\| x-x_k \|_{A^TA}$ and adaptive error estimate (top), adaptively chosen delay $k-\ell$ and its ideal value (bottom)}
    \label{pfig:well1033}
\end{figure}

\begin{figure}[htp]
    \centering
    \begin{minipage}{0.45\textwidth}
    \centering
    matrix \texttt{\detokenize{well1850}}, PCGLS\\
    \includegraphics[width=\textwidth]{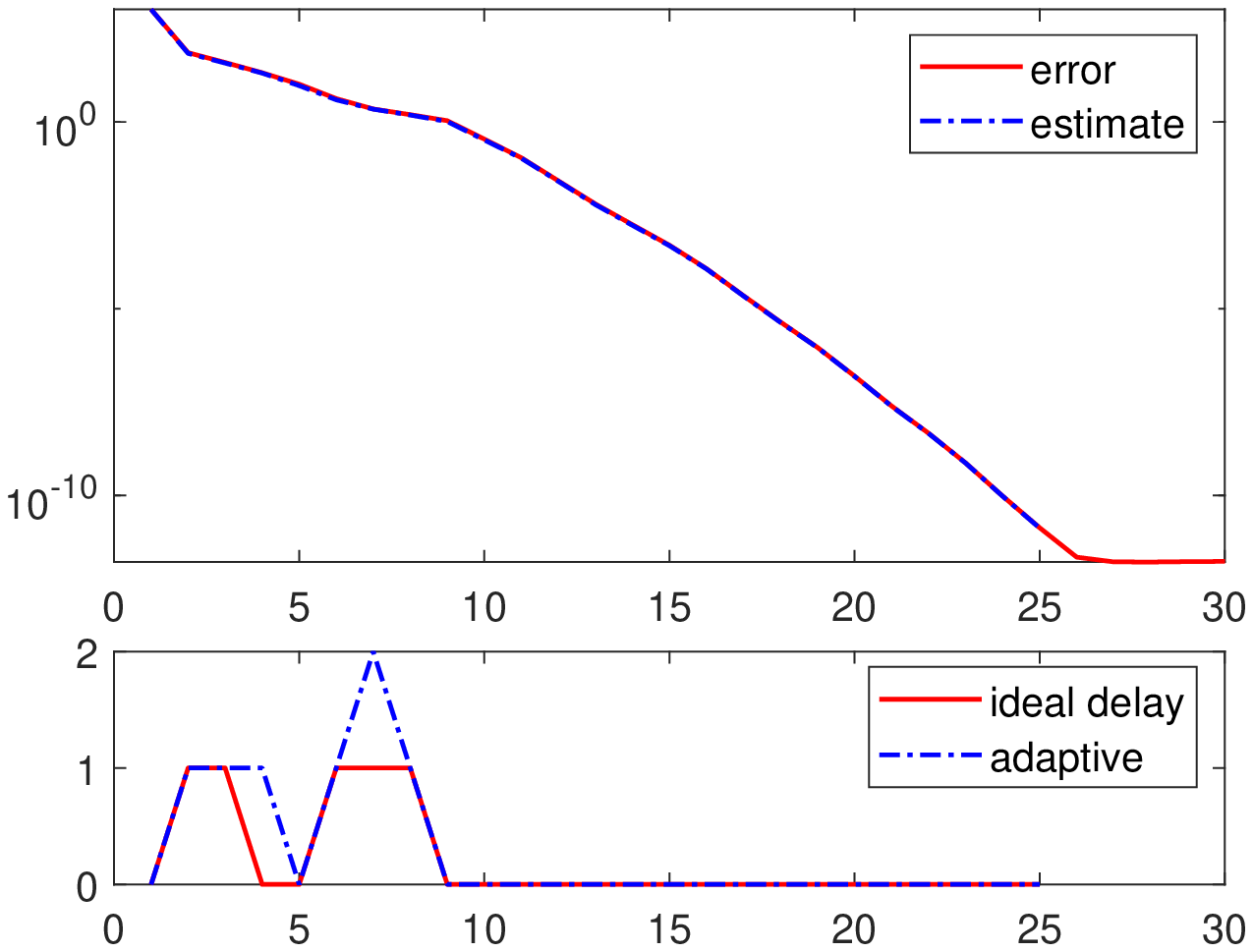}\\
    \textsf{iterations}
    \end{minipage}
    \hfill
    \begin{minipage}{0.45\textwidth}
    \centering
    matrix \texttt{\detokenize{well1850}}, PLSQR\\
    \includegraphics[width=\textwidth]{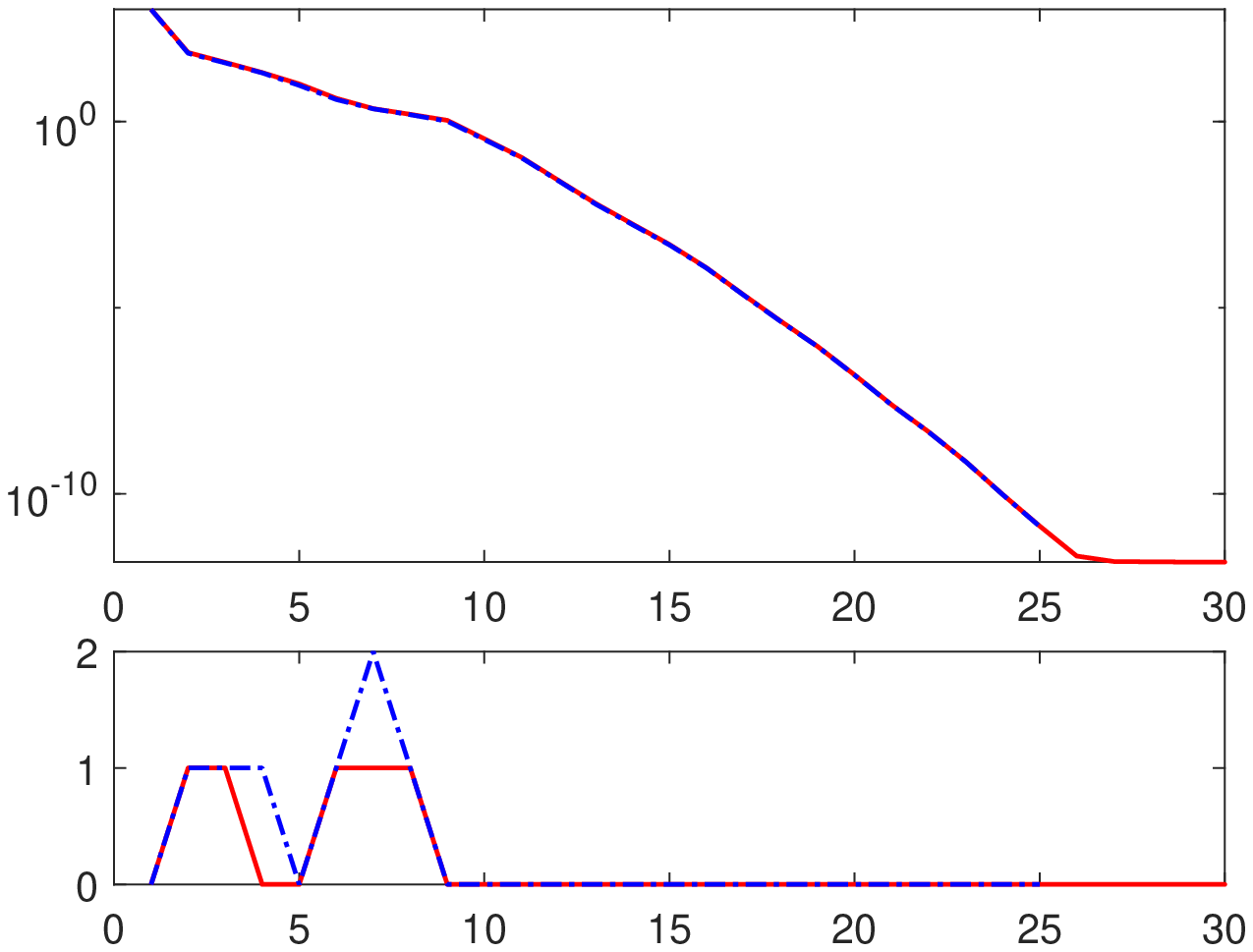}\\
    \textsf{iterations}
    \end{minipage}
    \caption{Matrix \texttt{\detokenize{well1850}} with preconditioning, PCGLS (left) and PLSQR (right): error~$\| x-x_k \|_{A^TA}$ and adaptive error estimate (top), adaptively chosen delay $k-\ell$ and its ideal value (bottom)}
    \label{pfig:well1850}
\end{figure}

\begin{figure}[htp]
    \centering
    \begin{minipage}{0.45\textwidth}
    \centering
    matrix \texttt{\detokenize{sls}}, PCGLS\\
    \includegraphics[width=\textwidth]{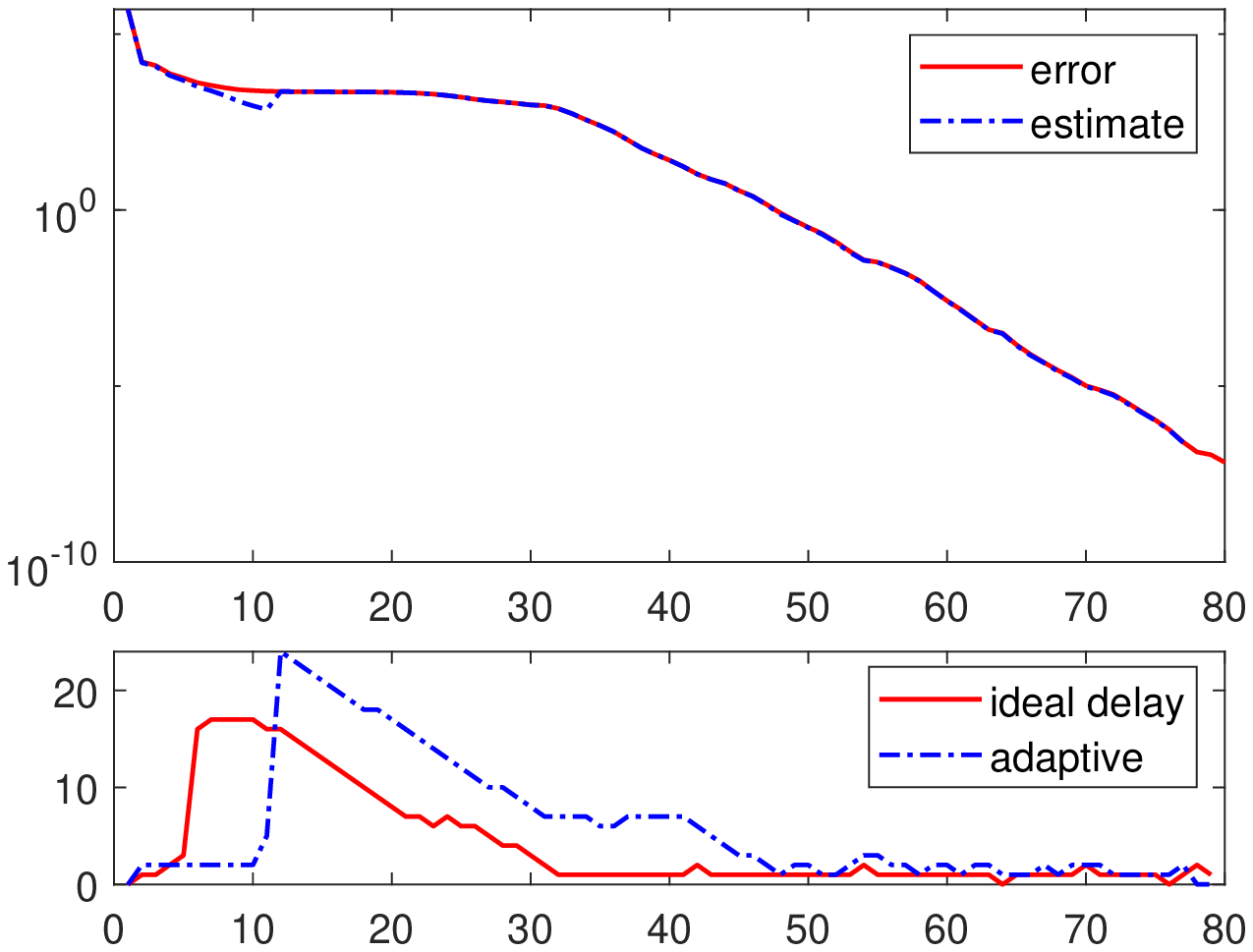}\\
    \textsf{iterations}
    \end{minipage}
    \hfill
    \begin{minipage}{0.45\textwidth}
    \centering
    matrix \texttt{\detokenize{sls}}, PLSQR\\
    \includegraphics[width=\textwidth]{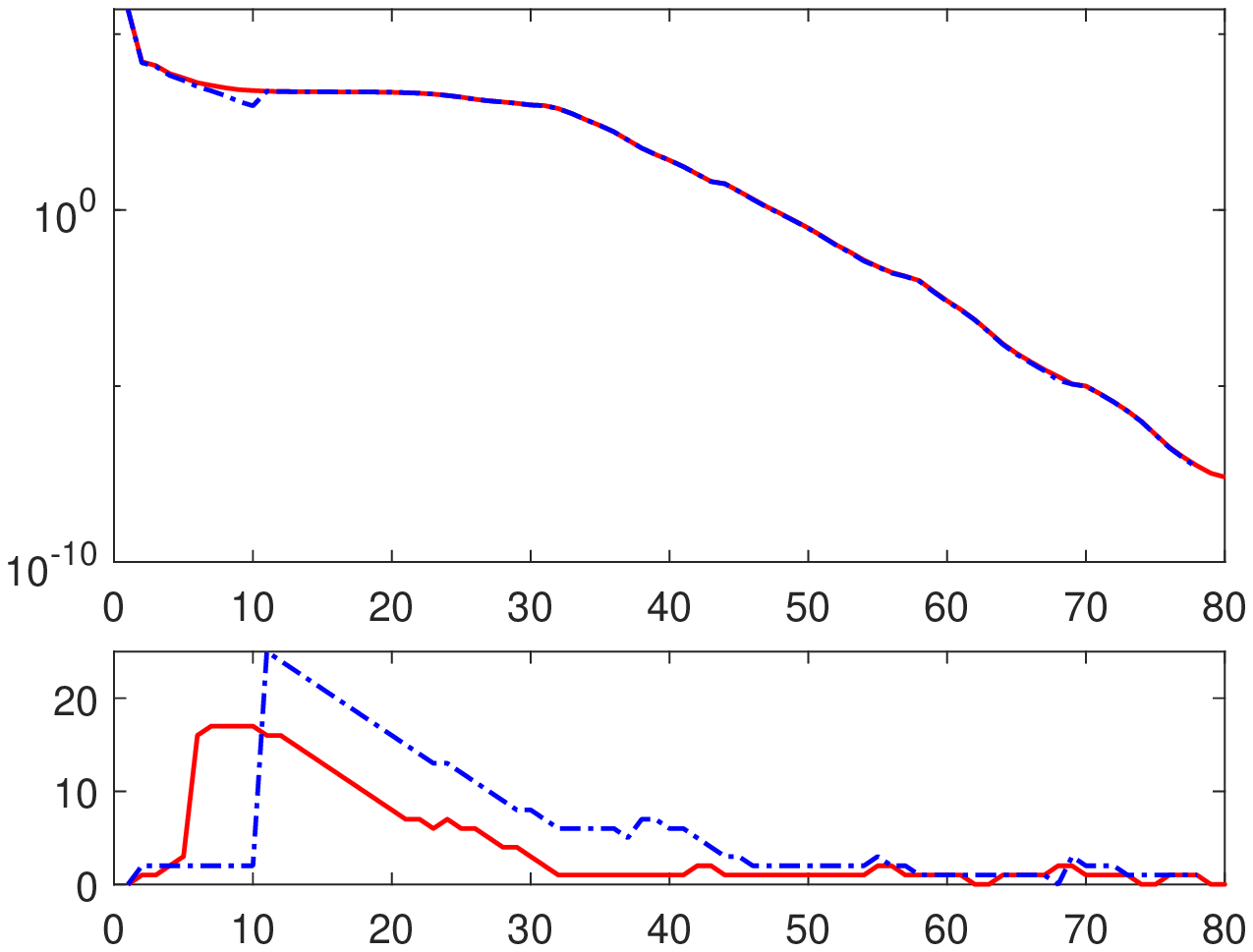}\\
    \textsf{iterations}
    \end{minipage}
    \caption{Matrix \texttt{\detokenize{sls}} with preconditioning, PCGLS (left) and PLSQR (right): error~$\| x-x_k \|_{A^TA}$ and adaptive error estimate (top), adaptively chosen delay $k-\ell$ and its ideal value (bottom)}
    \label{pfig:sls}
\end{figure}

\FloatBarrier

\subsection{Preconditioned least-norm problems}

The results for preconditioned least-norm problems solved by CGNE and CRAIG are given in Figures \ref{pfig:Delor64K}--\ref{pfig:lp_pilot}. As in the previous subsection, {\tt HSL\_MI35} is used to construct the preconditioner~$L$. The experiments confirm that the estimate can be reliably used also for preconditioned least-norm problems.

\begin{figure}[htp]
    \centering
    \begin{minipage}{0.45\textwidth}
    \centering
    matrix \texttt{\detokenize{Delor64K}}, PCGNE\\
    \includegraphics[width=\textwidth]{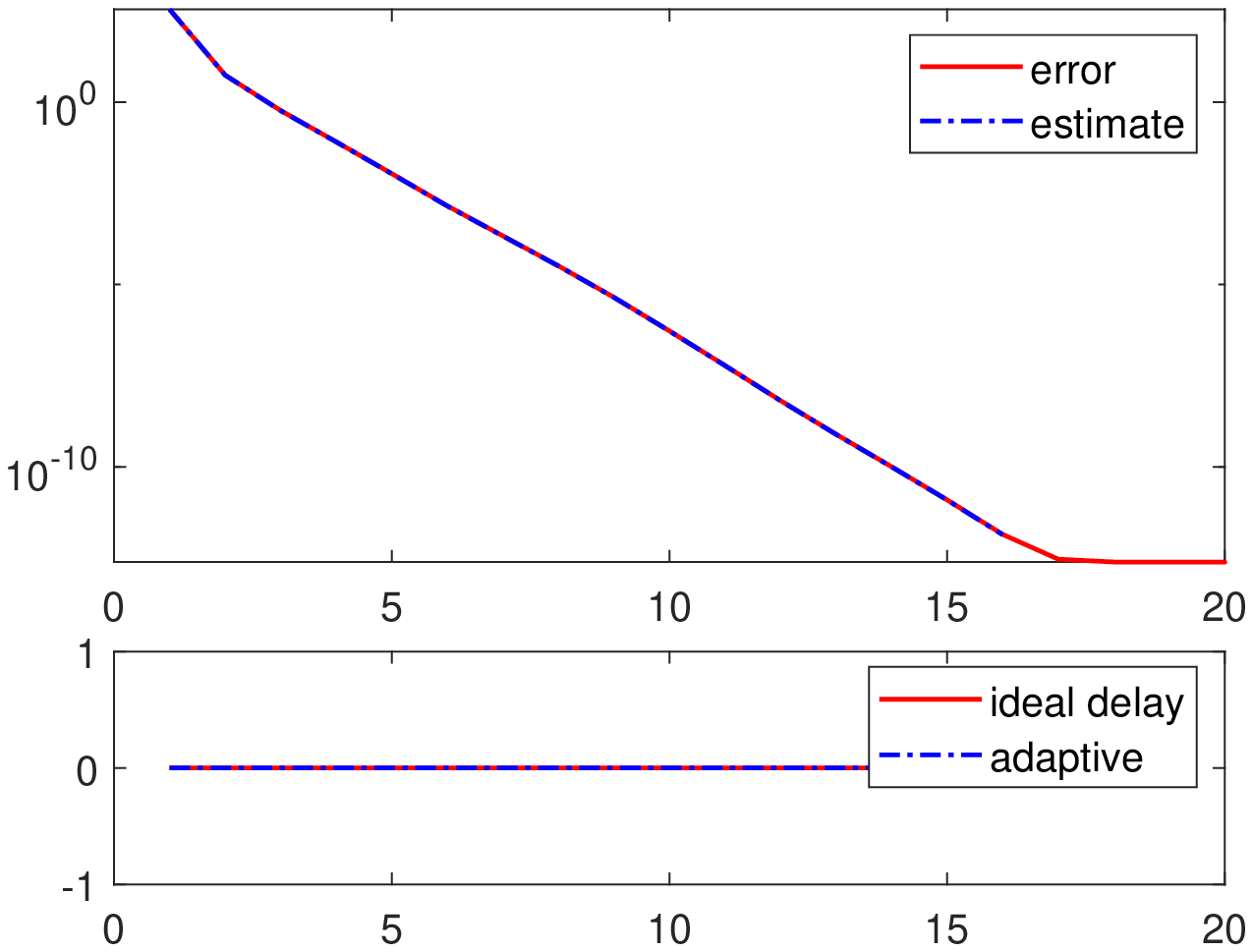}\\
    \textsf{iterations}
    \end{minipage}
    \hfill
    \begin{minipage}{0.45\textwidth}
    \centering
    matrix \texttt{\detokenize{Delor64K}}, PCRAIG\\
    \includegraphics[width=\textwidth]{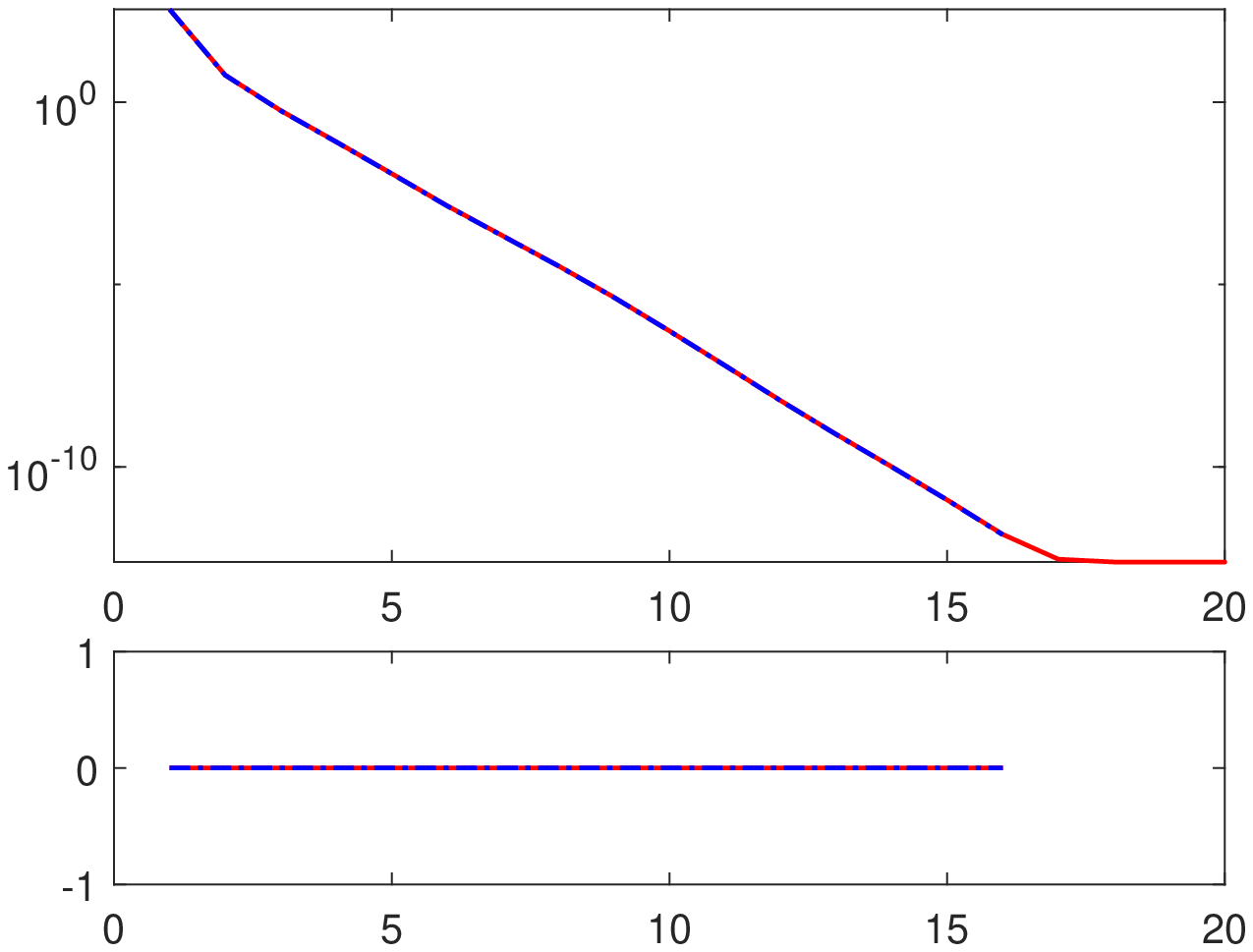}\\
    \textsf{iterations}
    \end{minipage}
    \caption{Matrix \texttt{\detokenize{Delor64K}} with preconditioning, PCGNE (left) and PCRAIG (right): error~$\| x-x_k \|$ and adaptive error estimate (top), adaptively chosen delay $k-\ell$ and its ideal value (bottom)}
    \label{pfig:Delor64K}
\end{figure}

\begin{figure}[htp]
    \centering
    \begin{minipage}{0.45\textwidth}
    \centering
    matrix \texttt{\detokenize{Delor338K}}, PCGNE\\
    \includegraphics[width=\textwidth]{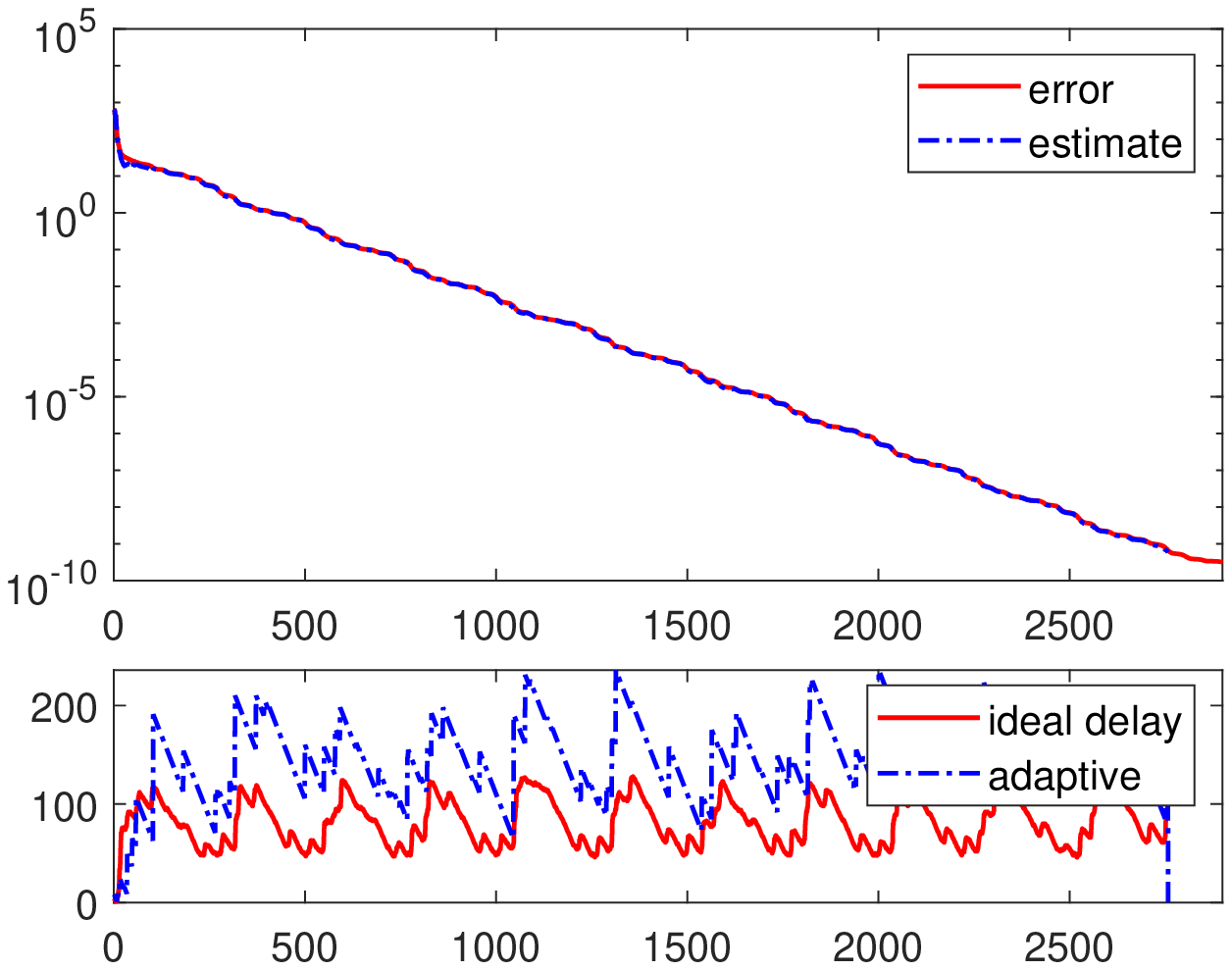}\\
    \textsf{iterations}
    \end{minipage}
    \hfill
    \begin{minipage}{0.45\textwidth}
    \centering
    matrix \texttt{\detokenize{Delor338K}}, PCRAIG\\
    \includegraphics[width=\textwidth]{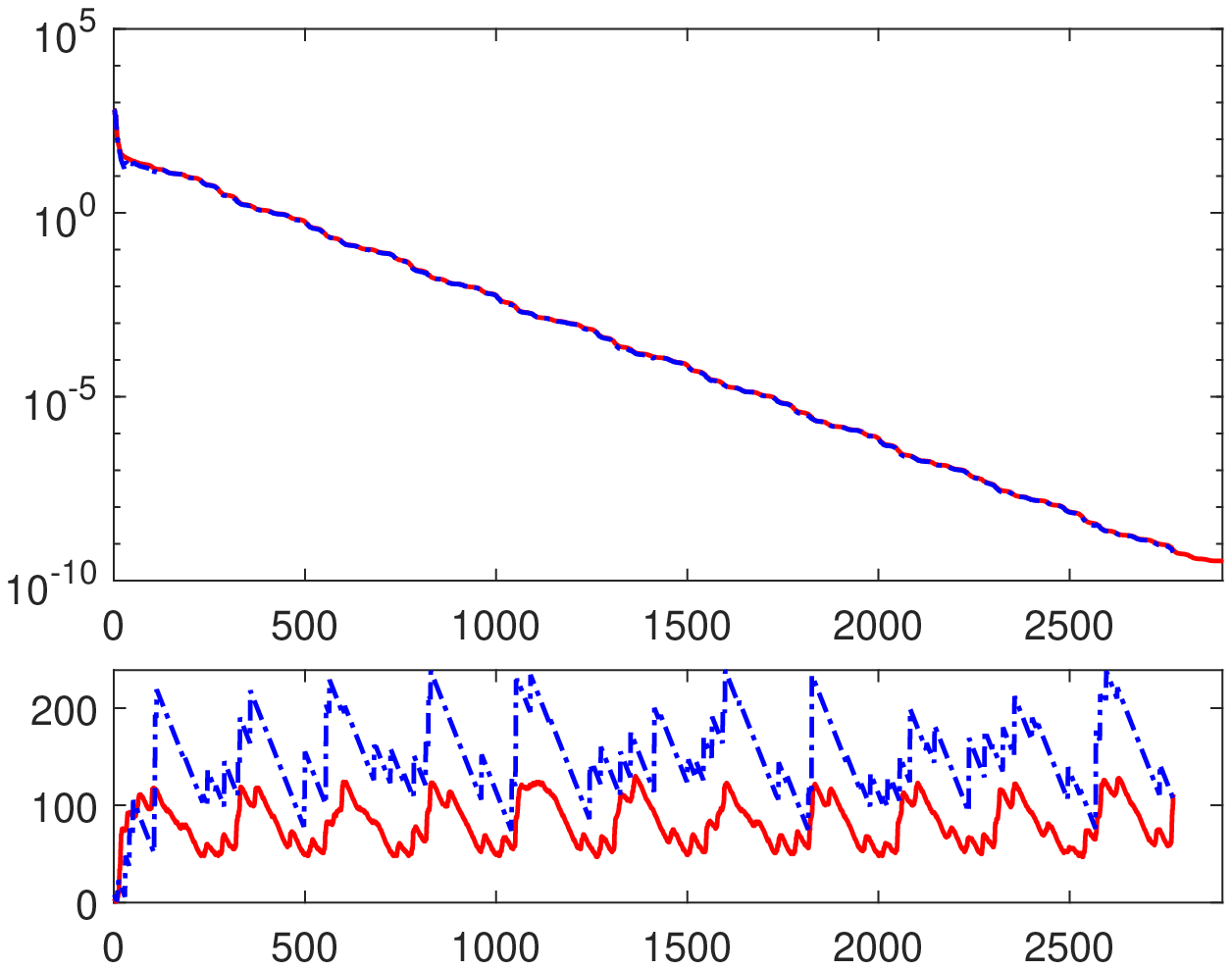}\\
    \textsf{iterations}
    \end{minipage}
    \caption{Matrix \texttt{\detokenize{Delor338K}} with preconditioning, PCGNE (left) and PCRAIG (right): error~$\| x-x_k \|$ and adaptive error estimate (top), adaptively chosen delay $k-\ell$ and its ideal value (bottom)}
    \label{pfig:Delor338K}
\end{figure}

\begin{figure}[htp]
    \centering
    \begin{minipage}{0.45\textwidth}
    \centering
    matrix \texttt{\detokenize{flower_7_4}}, PCGNE\\
    \includegraphics[width=\textwidth]{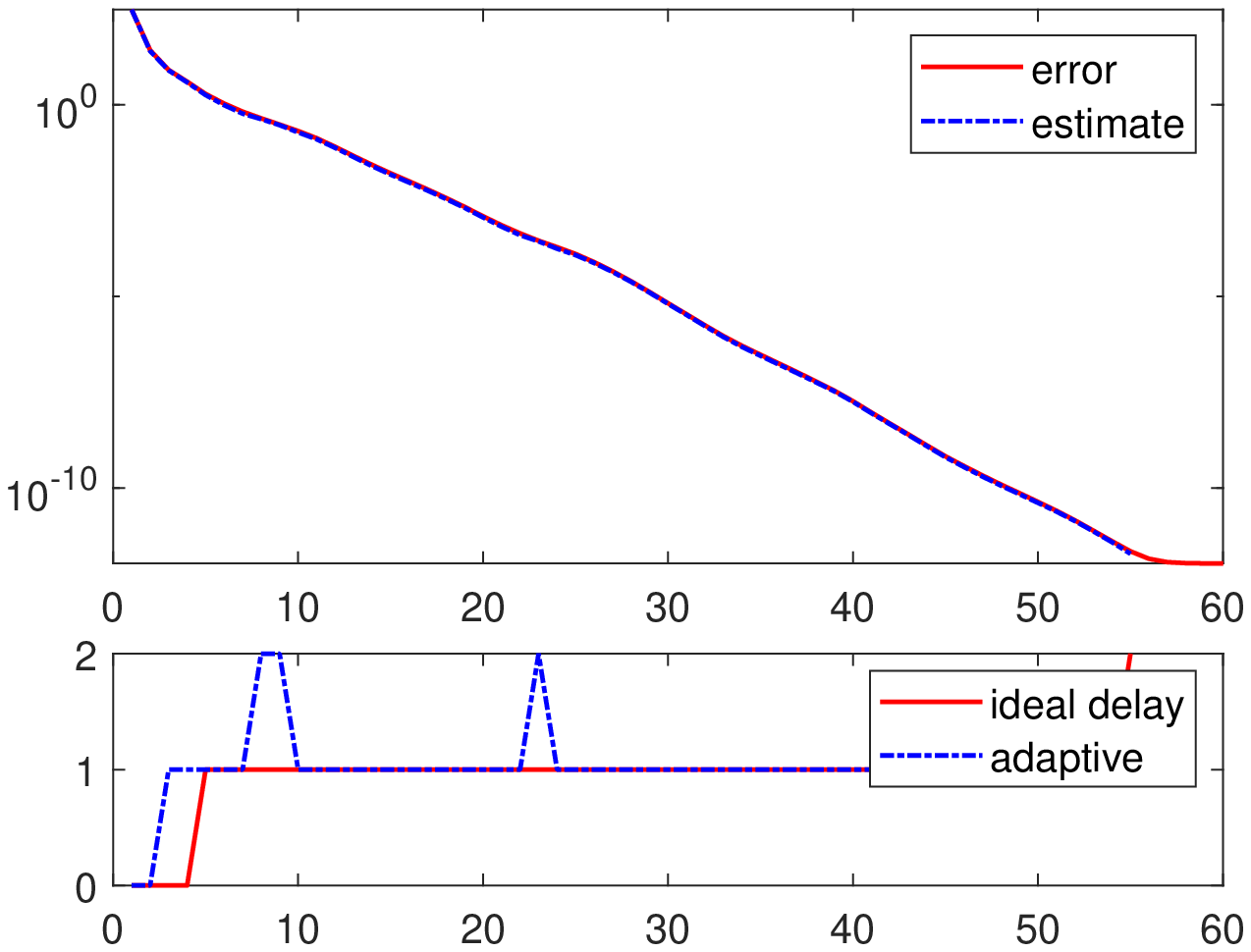}\\
    \textsf{iterations}
    \end{minipage}
    \hfill
    \begin{minipage}{0.45\textwidth}
    \centering
    matrix \texttt{\detokenize{flower_7_4}}, PCRAIG\\
    \includegraphics[width=\textwidth]{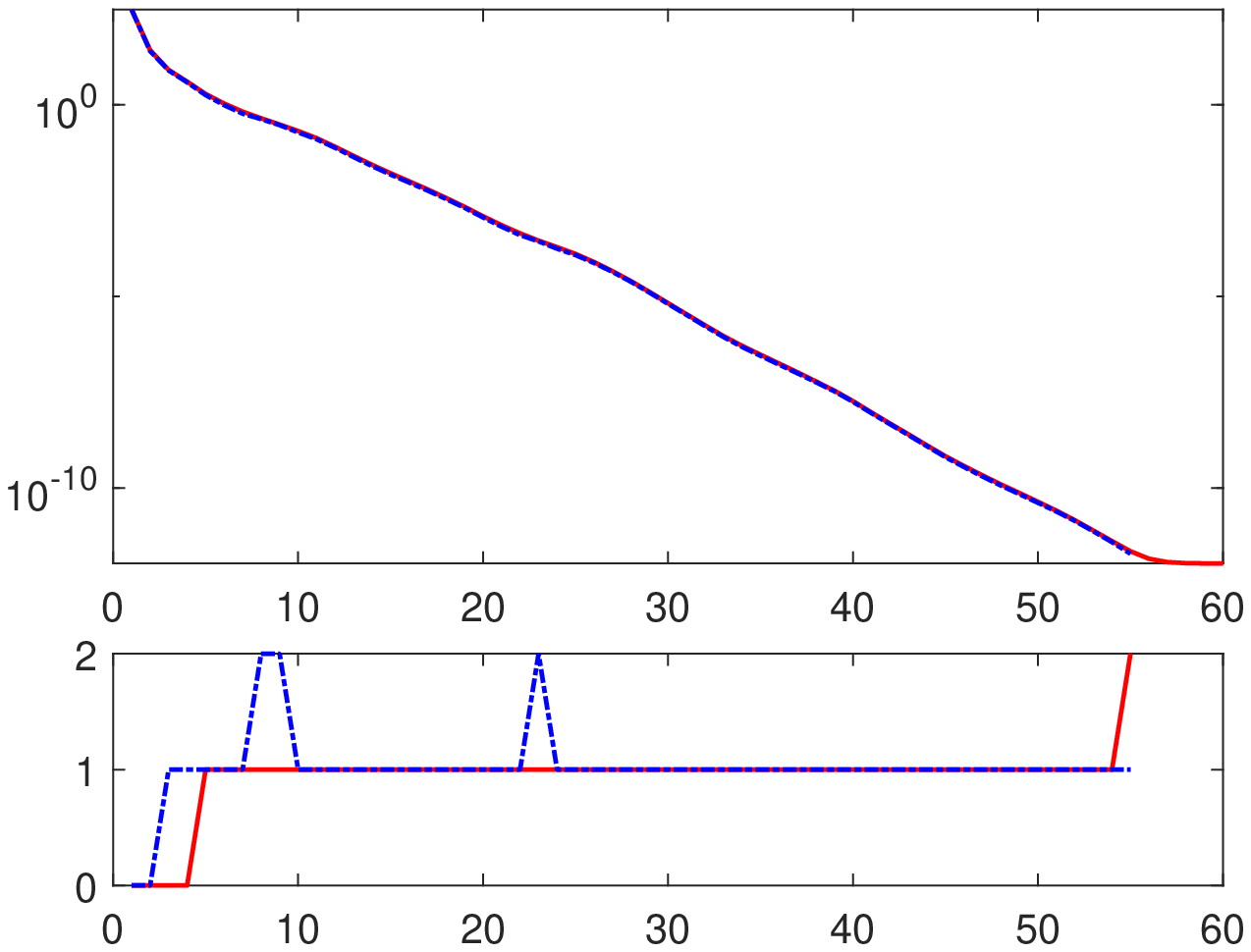}\\
    \textsf{iterations}
    \end{minipage}
    \caption{Matrix \texttt{\detokenize{flower_7_4}} with preconditioning, PCGNE (left) and PCRAIG (right): error~$\| x-x_k \|$ and adaptive error estimate (top), adaptively chosen delay $k-\ell$ and its ideal value (bottom)}
    \label{pfig:flower_7_4}
\end{figure}

\begin{figure}[htp]
    \centering
    \begin{minipage}{0.45\textwidth}
    \centering
    matrix \texttt{\detokenize{cat_ears_4_4}}, PCGNE\\
    \includegraphics[width=\textwidth]{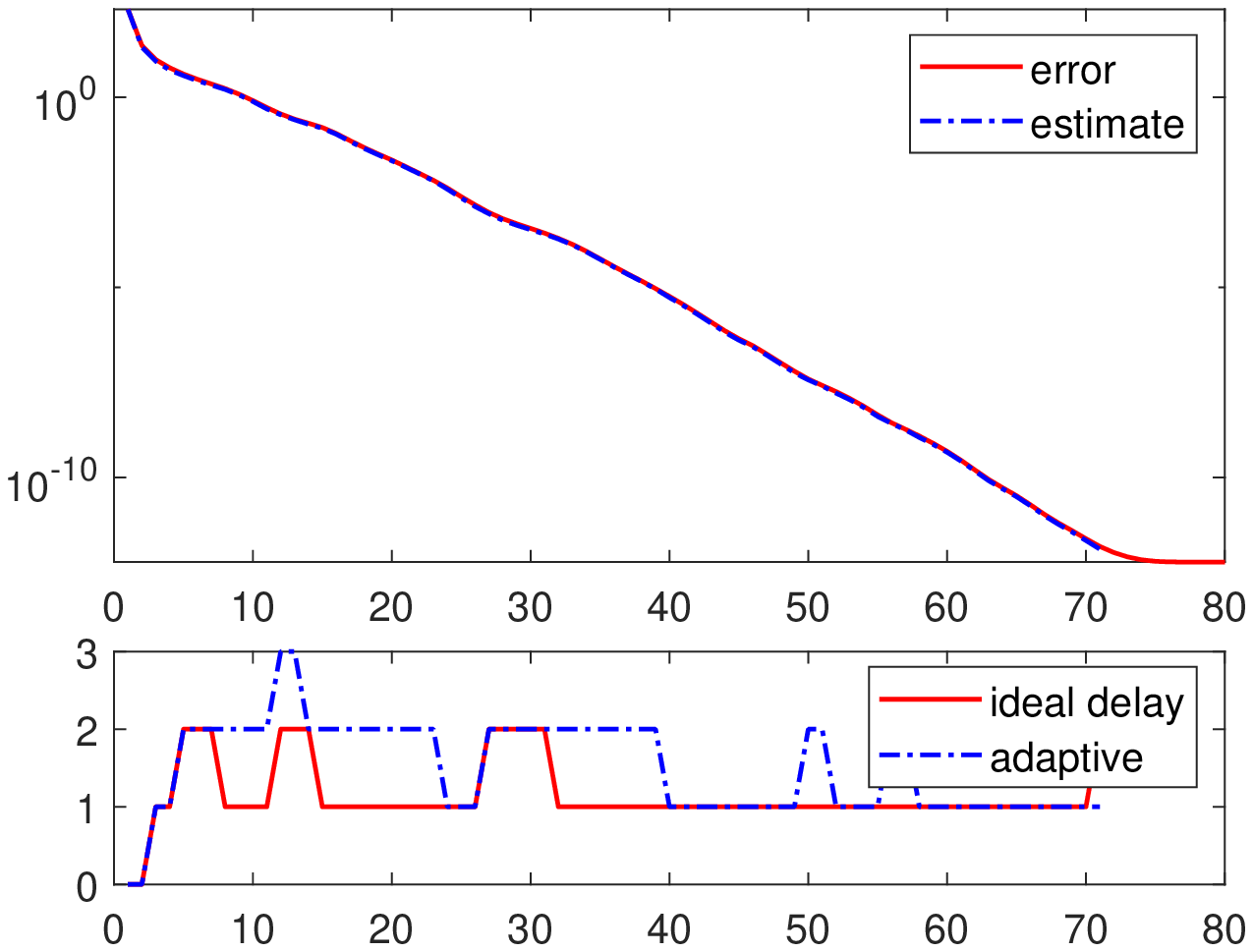}\\
    \textsf{iterations}
    \end{minipage}
    \hfill
    \begin{minipage}{0.45\textwidth}
    \centering
    matrix \texttt{\detokenize{cat_ears_4_4}}, PCRAIG\\
    \includegraphics[width=\textwidth]{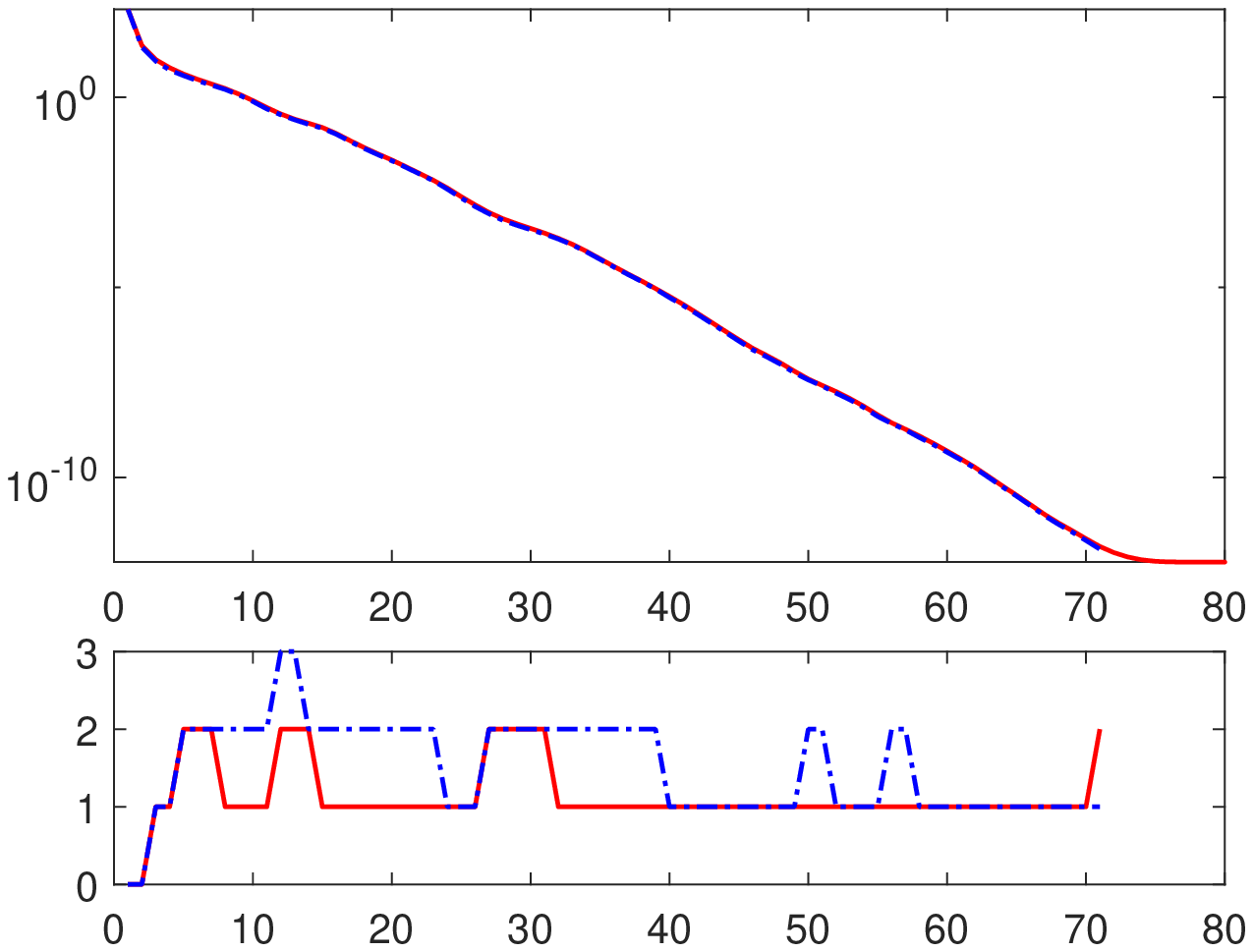}\\
    \textsf{iterations}
    \end{minipage}
    \caption{Matrix \texttt{\detokenize{cat_ears_4_4}} with preconditioning, PCGNE (left) and PCRAIG (right): error~$\| x-x_k \|$ and adaptive error estimate (top), adaptively chosen delay $k-\ell$ and its ideal value (bottom)}
    \label{pfig:cat_ears_4_4}
\end{figure}

\begin{figure}[htp]
    \centering
    \begin{minipage}{0.45\textwidth}
    \centering
    matrix \texttt{\detokenize{lp_pilot}}, PCGNE\\
    \includegraphics[width=\textwidth]{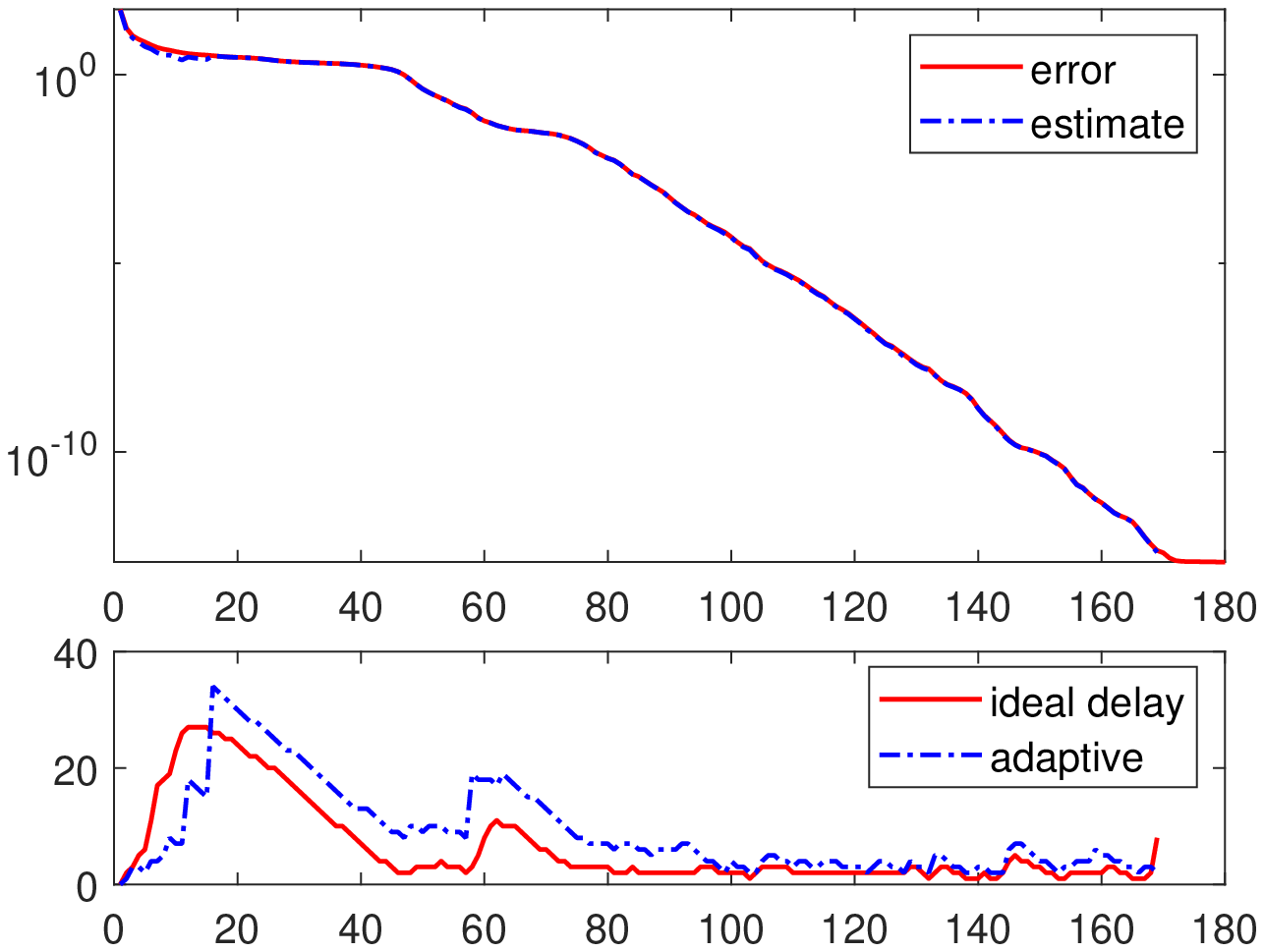}\\
    \textsf{iterations}
    \end{minipage}
    \hfill
    \begin{minipage}{0.45\textwidth}
    \centering
    matrix \texttt{\detokenize{lp_pilot}}, PCRAIG\\
    \includegraphics[width=\textwidth]{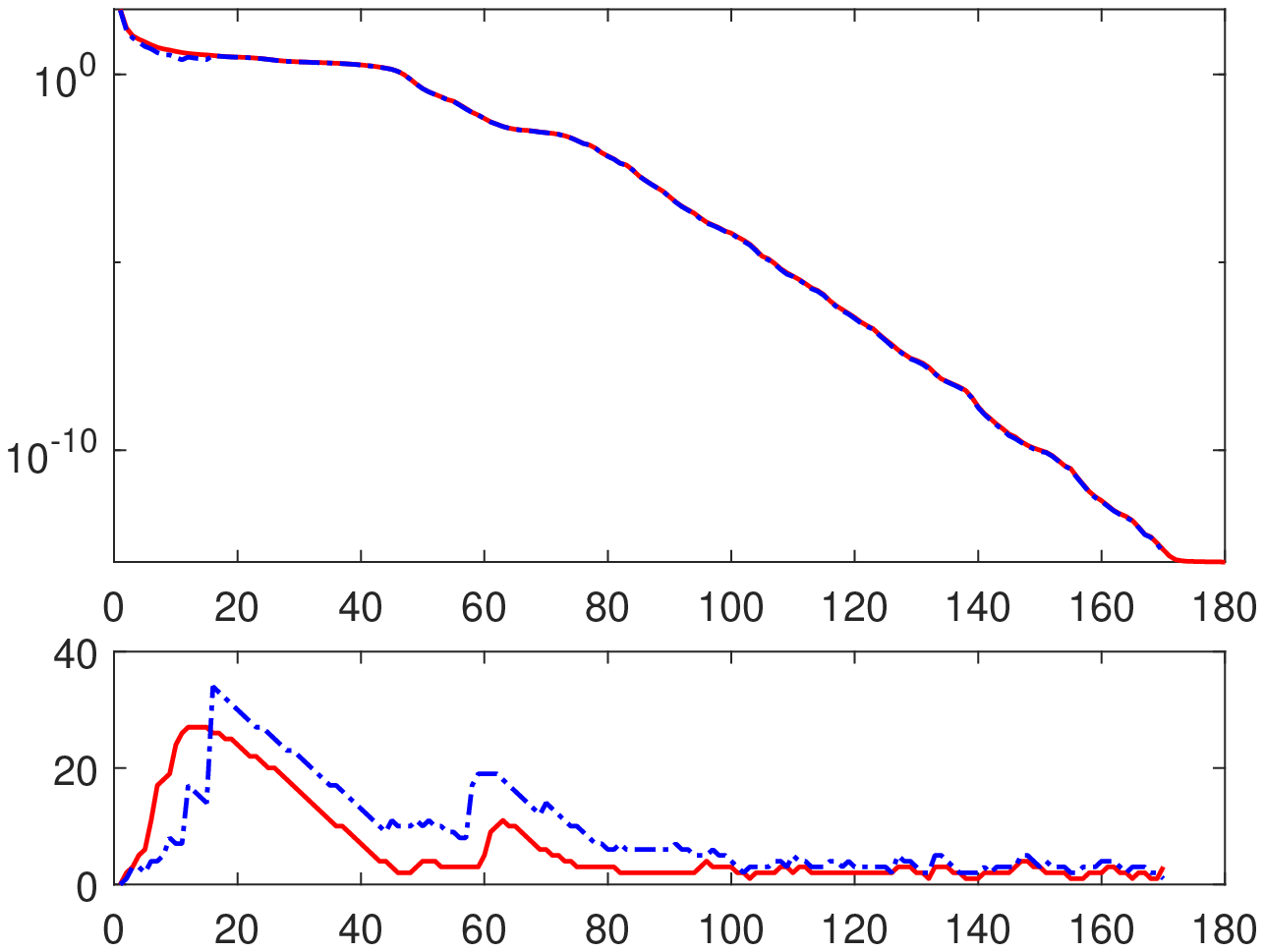}\\
    \textsf{iterations}
    \end{minipage}
    \caption{Matrix \texttt{\detokenize{lp_pilot}} with preconditioning, PCGNE (left) and PCRAIG (right): error~$\| x-x_k \|$ and adaptive error estimate (top), adaptively chosen delay $k-\ell$ and its ideal value (bottom)}
    \label{pfig:lp_pilot}
\end{figure}

\FloatBarrier

\section{Concluding discussion}
\label{sec:concl}

In this paper we focused on accurate estimation of errors in CG-like algorithms for solving least-squares and least-norm problems. Such estimates 
are computed just from the available coefficients, 
and their evaluation is very cheap.
Until a level of maximal attainable accuracy is reached,
the estimates are numerically reliable under the assumption
that the local orthogonality among consecutive vectors 
is preserved. The consecutive vectors we have in mind correspond to the underlying CG direction vector with the iteration index  $k-1$ and residual vector with index $k$. In this paper we did not analyze in detail the preservation of local orthogonality in the individual algorithms. However, based on previous results \cite{StTi02,StTi05,B:Me2006} for CG one can expect
that such an analysis is doable 
for all the algorithms discussed in this paper
with analogous results.
If needed, for example to verify the reliability of estimates in a particular application, the local-orthogonality terms can be evaluated, typically for the price of one extra inner product.

Our aim in this paper was to obtain an error estimate with a prescribed relative accuracy $\tau$. 
At the current iteration $k$, we estimate the quantity of interest related to some previous iteration $\ell\leq k$. We developed a heuristic strategy based on our previous results on CG \cite{MePaTi2021} to make the delay $k-\ell$ as short as possible. Our numerical results show that the suggested heuristic strategy is robust and reliable, and that the proposed delay $k-\ell$ is often almost optimal.

The suggested approach provides estimates that represent
lower bounds on the quantity of interest. Nevertheless, once a lower
bound $\Delta_{\ell:k}$ with a prescribed relative accuracy $\tau$
is obtained, the quantity
\[
\frac{\Delta_{\ell:k}}{1-\tau}
\]
represents an upper bound; see the discussion in \cite[Section 3.1]{MePaTi2021}.
Hence, we can also easily obtain tight (but not guaranteed) upper
bounds.

In summary, in CGLS and LSQR 
we obtain tight estimates of the quantity 
\[
\|x-x_{k}\|_{A^{T}A}^{2}=\left\Vert r_{k}\right\Vert ^{2}-\left\Vert b-b_{|\mathcal{R}(A)}\right\Vert ^{2}
\]
that can be used in stopping criteria of the algorithms. An example
of such a stopping criterion is discussed in \cite{ChPaTi2009} and
\cite{JiTi09}. It is suggested to stop the iterations when 
\[
\|x-x_{k}\|_{A^{T}A}^{2}\leq\alpha\|A\|\|x_{k}\|+\beta\|b\|,
\]
where $0\leq\alpha,\beta\ll1$ are some prescribed tolerances. In CGNE and CRAIG
we are able to estimate efficiently the quantity 
$
\|x-x_{k}\|^{2}.
$
Finally, assuming $x_{0}=0,$ our techniques can be straightforwardly applied  
for estimating the relative quantities 
\[
\frac{\|x-x_{k}\|_{A^{T}A}^{2}}{\|x\|_{A^{T}A}^{2}}\quad\mbox{and}\quad\frac{\|x-x_{k}\|^{2}}{\|x\|^{2}},
\]
in least-squares and least-norm problems, respectively; see also \cite{StTi05}.

We hope that the results presented in this paper will prove to be
useful in practical computations. They allow to approximate the errors
at a negligible cost during iterations of the considered algorithm,
while taking into account the prescribed
relative accuracy of the estimates.
The MATLAB codes of the algorithms with error estimates are available from the GitHub repository~\cite{Github_repo}.

\paragraph{Acknowledgement:} The work of Jan Papež has been supported by the Czech Academy of Sciences (RVO 67985840) and by the Grant Agency of the Czech Republic (grant no.~23-06159S).
The authors would like to thank G\'erard Meurant for careful reading of the manuscript and helpful comments, which have greatly improved the presentation.

\end{document}